\documentclass[12pt]{article}
\usepackage[mathscr]{euscript}
\usepackage[T1]{fontenc}
\usepackage[utf8]{inputenc}
\usepackage[english]{babel}
\usepackage{amsmath,amsfonts,amsthm,amssymb,MnSymbol}

\usepackage{tikz}
\usetikzlibrary{tikzmark}

\usepackage{hyperref,url}
\usepackage[figureright]{rotating}
\usepackage{changepage}

\usepackage{bussproofs}

\author{
Daniyar Shamkanov\\ \normalsize{\textit{Steklov Mathematical Institute of the Russian Academy of Sciences}}\\
\normalsize{\textit{National
Research University Higher School of Economics}}\\
\normalsize{\textit{daniyar.shamkanov@gmail.com}}\\
}
\title{On structural proof theory of the modal logic $\mathsf{K}^+$ extended with infinitary derivations}
\date{}

\theoremstyle{plain}
\newtheorem{theorem}{Theorem}
\newtheorem{lemma}{Lemma}
\newtheorem{propos}{Proposition}
\newtheorem{corollary}{Corollary}

\theoremstyle{definition}

\theoremstyle{remark}

\EnableBpAbbreviations

\begin{document}
\maketitle

\begin{abstract}
We consider an extension of the modal logic of transitive closure $\mathsf{K}^+$ with some inifinitary derivations and present a sequent calculus for this extension, which allows non-well-founded proofs. For the given calculus, we obtain the cut-elimination theorem following the lines of so called continuous cut elimination.
Our consideration also covers ordinary proofs of $\mathsf{K}^+$ since they correspond to cyclic cut-free proofs of the presented sequent calculus. 
\end{abstract}

\section{Introduction}
\label{s1}

One of the important goals of proof theory is to understand the structure of proofs. This article focuses on the structure of infinite proofs in the modal logic of transitive closure $\mathsf{K}^+$. Let us recall that $\mathsf{K^+}$ \cite{Kashima2010,DoSmo12,StShZo20} is a Kripke-complete modal propositional system whose language contains modal connectives $\Box$ and $\Box^+$. The corresponding class of Kripke frames consists of bimodal frames of the form $(W, R,R^+)$, where $R^+$ is the transitive closure of $R$. This system belongs to the family of modal fixed-point logics, which from the point of view of algebraic semantics can be explained as follows: in any $\mathsf{K^+}$-algebra $\mathcal{A}$, an element $\Box^+ a$ is the greatest fixed-point of a mapping $z \mapsto \Box a \wedge z$, i.e. $\Box^+ a = \nu z. (\Box a \wedge z)$. Consequently, $\mathsf{K^+}$ is a fragment of the modal $\mu$-calculus \cite{Kozen1983,Lenzi2010} and, like other systems reflecting various aspects of inductive reasoning, can be defined by a deductive system that allows cyclic and more general non-well-founded derivations. 


Non-well-founded derivations, or $\infty$-derivations for short, are defined as possibly infinite trees of formulas (sequents) constructed according to the inference rules of a deductive system.
Certain additional conditions are usually imposed on the infinite branches of these trees depending on the system under consideration. 
An important family of $\infty$-derivations is the family consisting of all so called regular ones, where an $\infty$-derivation is regular if it contains only finitely many non-isomorphic subtrees or, equivalently, can be obtained by unravelling a cyclic derivation. 
Cyclic derivations can be understood as finite directed graphs of formulas (sequents), the unravellings of which are non-well-founded derivations. In this article, we consider $\infty$-proofs in a sequent calculus for an infinitary extension of the logic $\mathsf{K^+}$ and address the problem of eliminating all applications of the cut rule from an $\infty$-proof syntactically.

Several articles can be mentioned that present cut elimination results for systems with non-well-founded proofs. Fortier and Santocanale worked with the $\mu$-calculus with additive connectives in \cite{ForSan2013}. Baelde, Doumane and Saurin considered non-well-founded proofs for a variant of the $\mu$-calculus called $\mu\mathsf{MALL}$ and obtained the corresponding cut elimination theorem in \cite{BaDoSa2016}. Together with Kuperberg, they further develop their approach in \cite{BaDoKuSa2022}. Das and Pous established the cut elimination theorem for the logic of $\ast$-continuous action lattices in \cite{DaPo2018}. In addition, Shamkanov and Savateev proved such a result in the case of the modal logic $\mathsf{Grz}$ \cite{SavSham21}. 

In all of these articles except \cite{SavSham21}, the authors first presented a procedure for eliminating all applications of the cut rule from a non-well-founded proof, which results in an infinite tree of sequents, and then showed that this tree satisfies the condition imposed on infinite branches\footnote{In \cite{ForSan2013}, Fortier and Santocanale didn't show that the resulting tree satisfies the corresponding condition of the $\mu$-calculus. However, Fortier noticed in private correspondence that this condition could be derived from the results of his thesis \cite{Fortier2014} (in French).}. A different approach was chosen in \cite{SavSham21}. Its authors defined a cut elimination operation, which directly provided a correct non-well-founded proof. In order to avoid nested co-inductive and inductive reasoning, they adopted a perspective from denotational semantics of computer languages, where program types are interpreted as ultrametric spaces and fixed-point combinators are embedded using the Banach fixeded-point theorem (see \cite{BaMaCe94}, \cite{Es1998} and \cite{Breu2001}). The authors considered the set of non-well-founded proofs of $\mathsf{Grz}$ and various sets of operations acting on these proofs as ultrametric spaces and defined the cut elimination operation following the lines of Prie\ss-Crampe's fixed-point theorem (see \cite{PriessCrampe1990} and \cite{Schoerner2003})\footnote{Essentially the same fixed-point result was obtained by Petalas and Vidalis in \cite{PetVid1993}.}, which is a strengthening of Banach's one for the case of spherically complete ultrametric spaces. 

In the given article, we want to further develop this approach to cut elimination. In our case, various sets of operations on non-well-founded sequent proofs of $\mathsf{K}^+$ could be considered as generalized ultrametric spaces. However, we decided to use a very similar approach, which also unifies recursive and co-recursive reasoning. Instead of generalized ultrametric spaces, we use sets equipped with well-ordered families of equivalence relations and define the cut elimination operation in accordance with the fixed-point theorem of Gianantonio and Miculan (see \cite{GiMi2003}). It is also interesting to note that the recursive equations we write down for the cut elimination operation are essentially the same as the equations in the standard case of finite sequent proofs. The difference is that we have to rely on a more complex fixed-point theorem in order to solve these equations in the case of non-well-founded proofs.



Recently we studied a fixed-point logic very similar to the system $\mathsf{K^+}$, namely the modal logic of common knowledge $\mathsf{S4}^C_I$. In \cite{Sham23}, we introduced an extension of the logic $\mathsf{S4}^{C}_I$ with some $\omega$-derivations and established strong neighbourhood completeness of the resulting system in the cases of local and global semantic consequence relations\footnote{A formula $A$ is a global semantic consequence of $\Gamma$ over the class of neighbourhood $\mathsf{S4}^{C}_I$-frames if, for any neighbourhood $\mathsf{S4}^{C}_I$-model $\mathcal{M}$,
\[(\forall B\in \Gamma\;\; \mathcal{M} \vDash B )\Longrightarrow \mathcal{M} \vDash A.\]}. 
For the system $\mathsf{K^+}$, similar $\omega$-derivations are defined as follows. An $\omega$-derivation in the logic $\mathsf{K}^+$ is a well-founded (possibly infinite) tree whose nodes are marked by formulas of $\mathsf{K}^+$ and that is constructed according to the rules
\[
\AXC{$A$}
\AXC{$A \rightarrow B$}
\LeftLabel{\textsf{mp}}
\RightLabel{ ,}
\BIC{$B$}
\DisplayProof \qquad
\AXC{$A$}
\LeftLabel{$\mathsf{nec}$}
\UIC{$\Box^+ A$}
\DisplayProof 
\]
and the inference rule
\[
\AXC{$B_0 \rightarrow \Box (A \wedge B_1)$}
\AXC{$B_1 \rightarrow \Box (A \wedge B_2)$}
\AXC{$B_2 \rightarrow \Box (A \wedge B_3)$}
\AXC{$\dots$}
\LeftLabel{$\omega$}
\RightLabel{ .}
\QIC{$B_0 \rightarrow \Box^+ A$}
\DisplayProof 
\]
We believe that this extension of $\mathsf{K^+}$ is also strongly locally and globally complete for its neighbourhood interpretation, just like the logic $\mathsf{S4}^{C}_I$, and, in the present article, focus on syntactic aspects of this extension. Inspired by a calculus from \cite{BuchKuznStud10}, we introduce our non-well-founded sequent calculus, which captures the consequence relation given by $\omega$-derivations, and obtain the corresponding cut-elimination theorem. This consideration also covers ordinary proofs of $\mathsf{K}^+$ since they correspond to regular cut-free $\infty$-proofs of the given calculus. 
It remains to note that cyclic sequent proofs of $\mathsf{K}^+$ were semantically studied by Doczkal and Smolka in \cite{DoSmo12}. For a survey of previous works on proof theory in the field of modal logic of common knowledge, we refer the reader to \cite{MaStu2018}.

The article is organized as follows. In Section \ref{s2}, we remind the reader of the axiomatic calculus for the logic $\mathsf{K}^+$ and define $\omega$-derivations in the given system. In Section \ref{s3}, we present our non-well-founded sequent calculus, which corresponds to the consequence relation given by $\omega$-derivations in $\mathsf{K}^+$. In Section
\ref{s4}, we introduce annotated versions of sequents, inference rules and $\infty$-proofs of this calculus. In Section \ref{s4.1}, we prove several fixed-point results, which allow us to recursively define operations acting on the set of $\infty$-proofs and show properties of such operations. In Sections \ref{s4.2} and \ref{s4.5}, we define the cut elimination operation. Finally, in Section \ref{s5}, we show that the logic $\mathsf{K}^+$ corresponds to the fragment of the sequent calculus obtained by allowing only regular cut-free $\infty$-proofs.


\section{Infinitary derivations in the logic $\mathsf{K}^+$}
\label{s2}
In this section, we briefly remind the reader of a Frege-Hilbert calculus for the bimodal logic of transitive closure $\mathsf{K}^+$ and consider infinitary derivations in the given system.


\textit{Formulas of $\mathsf{K}^+$} are built from propositional variables $p_0, p_1, p_2, \dotsc$ and the constant $\bot$ by means of propositional connectives $\to$, $\Box$ and $\Box^+$. We treat other Boolean connectives as abbreviations: $\neg A  := A\rightarrow \bot$, $\top  := \neg \bot$, $A \wedge B  := \neg (A \rightarrow \neg B)$, $A \vee B  := (\neg A \rightarrow B)$. In what follows, we denote the set of formulas of $\mathsf{K}^+$ by $\mathit{Fm}$. 


The calculus for the logic $\mathsf{K}^+$ is given by the following axioms and inference rules.\medskip

\textit{Axioms:}
\begin{itemize}
\item[(i)] the tautologies of classical propositional logic;
\item[(ii)] $\Box (A \rightarrow B) \rightarrow (\Box A \rightarrow \Box B)$;
\item[(iii)] $\Box^+ (A \rightarrow B) \rightarrow (\Box^+ A \rightarrow \Box^+ B)$;
\item[(iv)] $\Box^+ A \rightarrow \Box A \wedge \Box \Box^+ A$;
\item[(v)] $\Box A \wedge \Box^+(A \rightarrow \Box A) \rightarrow \Box^+ A$.
\end{itemize}

\textit{Inference rules:} 
\[
\AXC{$A$}
\AXC{$A \rightarrow B$}
\LeftLabel{\textsf{mp}}
\RightLabel{ ,}
\BIC{$B$}
\DisplayProof \qquad
\AXC{$A$}
\LeftLabel{$\mathsf{nec}$}
\RightLabel{ .}
\UIC{$\Box^+ A$}
\DisplayProof 
\]

We introduce infinitary derivations by adding to the given calculus a specific admissible rule with countably many premises. Recall that an inference rule is called \emph{admissible} (in a given deductive system) if, for every instance of the rule, the conclusion is provable whenever all premises are provable.
An \emph{$\omega$-derivation} is a well-founded tree whose nodes are marked by formulas of $\mathsf{K}^+$ and that is constructed according to the rules ($\mathsf{mp}$), ($\mathsf{nec}$) and the following rule:
\[
\AXC{$B_0 \rightarrow \Box (A \wedge B_1)$}
\AXC{$B_1 \rightarrow \Box (A \wedge B_2)$}
\AXC{$B_2 \rightarrow \Box (A \wedge B_3)$}
\AXC{$\dots$}
\LeftLabel{$\omega$}
\RightLabel{ .}
\QIC{$B_0 \rightarrow \Box^+ A$}
\DisplayProof 
\]
In the given inference, we call all premises except the leftmost one \emph{boxed}. Admissibility of the rule ($\omega$) in the system $\mathsf{K}^+$ is proved in the final section. 



We define an \emph{assumption leaf} of an $\omega$-derivation $\delta$ as a leaf that is not marked by an axiom of $\mathsf{K}^+$. An assumption leaf is \emph{boxed} if, on the path from the root of $\delta$ to the given leaf, there exists an application of the inference rule ($\mathsf{nec}$) or this path intersects an application of the rule ($\omega$) on a boxed premise. Note that any $\omega$-derivation contains only finitely many non-boxed assumption leaves.
We put $\Sigma;\Gamma \vdash_\omega A$ if there is an $\omega$-derivation with a root marked by $A$ in which every boxed assumption leaf is marked by some formula from $\Sigma$ and every non-boxed assumption leaf is marked by some formula from $\Gamma$. 

\section{A non-well-founded sequent calculus}
\label{s3}
We introduce a sequent system inspired by a calculus from \cite{BuchKuznStud10}. The system allows non-well-founded proofs and corresponds to the consequence relation given by $\omega$-derivations in $\mathsf{K}^+$. In this section, we prove that $\omega$-derivations of $\mathsf{K}^+$ can be transformed into non-well-founded proofs of the sequent system. The converse will be shown in the next section.

A \textit{sequent} is defined as an expression of the form $\Sigma;\Gamma \Rightarrow \Delta$, where $\Gamma$ and $\Delta$ are finite multisets of formulas and $\Sigma$ is an arbitrary set of formulas. 
For a multiset of formulas $\Gamma = A_1,\dotsc, A_n$, we set $\Box \Gamma := \Box A_1,\dotsc, \Box A_n$ and $\Box^+ \Gamma := \Box^+ A_1,\dotsc, \Box^+ A_n$.

The sequent calculus $\mathsf{S}$
has the following initial sequents and inference rules: 
\begin{gather*}
\AXC{$\Sigma;\Gamma, p \Rightarrow p, \Delta, $}
\DisplayProof  \qquad
\AXC{$\Sigma;\Gamma , \bot \Rightarrow  \Delta,$}
\DisplayProof 
\end{gather*}
\begin{gather*}
\AXC{$\Sigma;\Gamma , B \Rightarrow  \Delta$}
\AXC{$\Sigma;\Gamma \Rightarrow  A, \Delta$}
\LeftLabel{$\mathsf{\rightarrow_L}$}
\BIC{$\Sigma;\Gamma , A \rightarrow B \Rightarrow  \Delta$}
\DisplayProof \;,\qquad
\AXC{$\Sigma;\Gamma, A \Rightarrow  B ,\Delta$}
\LeftLabel{$\mathsf{\rightarrow_R}$}
\UIC{$\Sigma;\Gamma \Rightarrow  A \rightarrow B ,\Delta$}
\DisplayProof \;,
\end{gather*}
\begin{gather*}
\AXC{$\Sigma;\Sigma_0,\Lambda, \Pi, \Box^+ \Pi \Rightarrow A$}
\LeftLabel{$\mathsf{\Box}$}
\UIC{$\Sigma; \Phi, \Box \Lambda, \Box^+ \Pi \Rightarrow \Box A , \Psi$}
\DisplayProof \;,
\end{gather*}
\begin{gather*}
\AXC{$\Sigma;\Sigma_0,\Lambda, \Pi, \Box^+ \Pi \Rightarrow A$}
\AXC{$\Sigma;\Sigma_0,\Lambda, \Pi, \Box^+ \Pi \Rightarrow \Box^+ A$}
\LeftLabel{$\Box^+$}
\BIC{$\Sigma; \Phi, \Box \Lambda, \Box^+ \Pi \Rightarrow  \Box^+ A ,\Psi$}
\DisplayProof \;,
\end{gather*}
where $\Sigma_0$ is a finite subset of $\Sigma$. 

The last two inference rules of the calculus are called \emph{modal rules}. For a modal rule ($\Box$) (or ($\Box^+$)), the elements of $\Box \Lambda$ and $\Box^+ \Pi$ are \emph{side formulas} and the formula $\Box A$ (or $\Box^+ A$) is the \emph{principal formula} of the corresponding inference.
We call an application of a modal rule \emph{slim}\footnote{We coined this term from \cite{Iemhoff2016}.} if multisets $\Lambda $ and $\Pi$ do not contain repetitions (i.e., $\Lambda $ and $\Pi$ are sets). 


The cut rule is defined as 
\begin{gather*}
\AXC{$\Sigma;\Gamma\Rightarrow A,\Delta$}
\AXC{$\Sigma;\Gamma,A\Rightarrow\Delta$}
\LeftLabel{$\mathsf{cut}$}
\RightLabel{ ,}
\BIC{$\Sigma;\Gamma\Rightarrow\Delta$}
\DisplayProof
\end{gather*}
where $A$ is the \emph{cut formula} of the given inference. We define the calculus $\mathsf{S}+\mathsf{cut}$ by adding the inference rule ($\mathsf{cut}$) to the system $\mathsf{S}$.

An \emph{$\infty$-proof} is a (possibly infinite) tree whose nodes are marked by sequents and that is constructed according to the rules of the calculus $\mathsf{S}+\mathsf{cut}$. In addition, every leaf in an $\infty$-proof is marked by an initial sequent and every infinite branch in it must contain a tail satisfying the following conditions: (a) all applications of the rule ($\mathsf{\Box^+}$) in the tail have the same principal formula $\Box^+ A$; (b) the tail passes through the right premise of the rule ($\mathsf{\Box^+}$) infinitely many times; (c) the tail doesn't pass through the left premise of the rule ($\mathsf{\Box^+}$); (d) there are no applications of the rule ($\Box$) in the tail. An $\infty$-proof $\pi$ is called \emph{cut-free} if there are no applications of the rule ($\mathsf{cut}$) in $\pi$. Besides, an $\infty$-proof is \emph{regular} if it contains only finitely many non-isomorphic subtrees. We say that a sequent is \emph{provable in $\mathsf{S+\mathsf{cut}}$ (in $\mathsf{S}$)} if there is a (cut-free) $\infty$-proof with the root marked by this sequent. 

For example, consider the following cut-free $\infty$-proof:
\begin{gather}\label{Example}
\AXC{$\mathsf{Ax}$}
\noLine
\UIC{$\Sigma ; p , F, \Box^+ F  \Rightarrow   p$}
\AXC{$\pi$}
\noLine
\UIC{$\vdots$}
\noLine
\UIC{$\Sigma ; p ,  \Box p, \Box^+ F  \Rightarrow   \Box^+ p$}
\AXC{$\mathsf{Ax}$}
\noLine
\UIC{$\Sigma ; p ,   \Box^+ F  \Rightarrow    p, \Box^+ p$}
\LeftLabel{$\mathsf{\rightarrow_L}$}
\BIC{$\Sigma ;p , F, \Box^+ F  \Rightarrow   \Box^+ p$}
\LeftLabel{$\mathsf{\Box}^+$}
\RightLabel{ ,}
\BIC{$\Sigma ; p, \Box p , \Box^+ F  \Rightarrow  \Box^+ p$}
\DisplayProof 
\end{gather}
where $F= p \rightarrow \Box p$ and the subtree $\pi$ is isomorphic to the whole $\infty$-proof. Here we have the unique infinite branch, which passes through alternate applications of inference rules ($\mathsf{\rightarrow_L}$) and ($\mathsf{\Box^+}$) infinitely many times. Conditions (a)-(d) on infinite branches of $\infty$-proofs trivially hold if we consider the given branch as its own tail.  

Note that we can easily find an infinite tree  constructed according to the rules of $\mathsf{S}$ that is not an $\infty$-proof. Indeed, consider a sequent tree
\[
\AXC{$\mathsf{Ax}$}
\noLine
\UIC{$\Sigma ; p , \Box^+ (\Box p \to p)   \Rightarrow   p$}
\AXC{$\pi$}
\noLine
\UIC{$\vdots$}
\noLine
\UIC{$\Sigma ;\Box^+ (\Box p \to p)  \Rightarrow \Box p, p$}
\LeftLabel{$\mathsf{\rightarrow_L}$}
\BIC{$\Sigma ;\Box p \to p , \Box^+ (\Box p \to p)   \Rightarrow   p$}
\LeftLabel{$\mathsf{\Box}$}
\RightLabel{ ,}
\UIC{$\Sigma ; \Box^+ (\Box p \to p)  \Rightarrow  \Box p, p$}
\DisplayProof 
\]
where $\pi $ is isomorphic to the whole tree. The only infinite branch does not intersect applications of the rule ($\Box^+$) and, therefore, does not contain tails that satisfy condition (b).
 
Let us consider an example of an $\infty$-proof in which $\Sigma$ is involved significantly. In what follows, we abbreviate proof fragments of the form 
\begin{gather*}
\AXC{$\mathsf{Ax}$}
\noLine
\UIC{$\Sigma;\Gamma,  \bot \Rightarrow \Delta$}
\AXC{$\Sigma;\Gamma, A, B \Rightarrow   \bot,\Delta$}
\LeftLabel{$\mathsf{\to_R}$}
\UIC{$\Sigma;\Gamma, A \Rightarrow   B\to \bot,\Delta$}
\LeftLabel{$\mathsf{\to_R}$}
\UIC{$\Sigma;\Gamma \Rightarrow  A\to (B\to \bot),\Delta$}
\LeftLabel{$\mathsf{\to_L}$}
\RightLabel{ }
\BIC{$\Sigma;\Gamma, (A\to (B\to \bot))\to \bot \Rightarrow \Delta$}
\DisplayProof 
\end{gather*}
as
\begin{gather*}
\AXC{$\Sigma;\Gamma, A, B \Rightarrow \bot, \Delta$}
\LeftLabel{$\mathsf{\wedge^\prime_L}$}
\doubleLine
\RightLabel{ .}
\UIC{$\Sigma;\Gamma, A\wedge B \Rightarrow \Delta$}
\DisplayProof 
\end{gather*}
Consider the sequent trees of Fig. 1, where $\Sigma = \{p_j \to \Box (q \wedge  p_{j+1})\mid j \in \mathbb{N} \}$ and $i\in \mathbb{N}$. Let us glue these trees at the nodes marked by $\Sigma; p_i \to \Box (q \wedge p_{i+1}) ,q, p_i \Rightarrow   \Box^+ q$. The resulting tree contains the unique infinite branch, which obviously satisfies conditions (a)-(d). We see an $\infty$-proof of the sequent $\Sigma; p_0 \to \Box (q \wedge p_1), q, p_0\Rightarrow \Box^+ q$.

\begin{center}
\begin{sidewaystable}
\begin{adjustwidth}{0cm}{-2.7cm}
\begin{minipage}{\textwidth}
\bigskip
\begin{gather*}
\footnotesize
\AXC{$\mathsf{Ax}$}
\noLine
\UIC{$\Sigma; p_{i+1} \to \Box (q \wedge p_{i+2}) ,q, p_{i+1} \Rightarrow    q$}
\LeftLabel{$\mathsf{\wedge^\prime_L}$}
\doubleLine
\UIC{$\Sigma; p_{i+1} \to \Box (q \wedge p_{i+2}) ,q\wedge p_{i+1} \Rightarrow    q$}
\AXC{$\Sigma; p_{i+1} \to \Box (q \wedge p_{i+2}) ,q, p_{i+1} \Rightarrow   \Box^+ q$}
\LeftLabel{$\mathsf{\wedge^\prime_L}$}
\doubleLine
\UIC{$\Sigma; p_{i+1} \to \Box (q \wedge p_{i+2}) ,q\wedge p_{i+1} \Rightarrow   \Box^+ q$}
\LeftLabel{$\Box^+$}
\BIC{$\Sigma ;  \Box (q \wedge p_{i+1}),q,p_i\Rightarrow  \Box^+ q$}
\AXC{$\mathsf{Ax}$}
\noLine
\UIC{$\Sigma ;  q,p_i\Rightarrow p_i, \Box^+ q $}
\LeftLabel{$\mathsf{\to_L}$}
\RightLabel{ }
\BIC{$\Sigma; p_i \to \Box (q \wedge p_{i+1}), q,p_i\Rightarrow \Box^+ q$}
\DisplayProof
\end{gather*}
\medskip 
\end{minipage}

\center{\textbf{Fig. 1}}

\begin{minipage}{\textwidth}
\bigskip
\begin{gather*}
\footnotesize
\AXC{$\pi_i$}
\noLine
\UIC{$\vdots$}
\noLine
\UIC{$\Sigma;\Sigma_{i}\Rightarrow B_i\to \Box (C \wedge B_{i+1}) $}
\LeftLabel{$\mathsf{wk}$}
\UIC{$\Sigma;\Sigma_{i}, C,B_i\Rightarrow B_i\to \Box (C \wedge B_{i+1}), \Box^+ C $}
\AXC{$\theta_i$}
\noLine
\UIC{$\vdots$}
\noLine
\UIC{$\Sigma;\Sigma_{i+1} ,C, B_{i+1} \Rightarrow   C$}
\LeftLabel{$\mathsf{\wedge^\prime_L}$}
\doubleLine
\UIC{$\Sigma;\Sigma_{i+1} ,C\wedge B_{i+1} \Rightarrow   C$}
\AXC{$\Sigma;\Sigma_{i+1} ,C, B_{i+1} \Rightarrow   \Box^+ C$}
\LeftLabel{$\mathsf{\wedge^\prime_L}$}
\doubleLine
\UIC{$\Sigma;\Sigma_{i+1} ,C\wedge B_{i+1} \Rightarrow   \Box^+ C$}
\LeftLabel{$\Box^+$}
\BIC{$\Sigma;\Sigma_i, C, B_i,\Box (C \wedge B_{i+1}) \Rightarrow  \Box^+ C$}
\AXC{$\sigma_i$}
\noLine
\UIC{$\vdots$}
\noLine
\UIC{$\Sigma;\Sigma_i, C, B_i\Rightarrow B_i, \Box^+ C $}
\LeftLabel{$\mathsf{\to_L}$}
\BIC{$\Sigma;\Sigma_i, C, B_i,B_i\to \Box C \wedge \Box B_{i+1} \Rightarrow  \Box^+ C $}
\LeftLabel{$\mathsf{cut}$}
\RightLabel{ }
\BIC{$\Sigma;\Sigma_i, C,B_i\Rightarrow \Box^+ C$}
\LeftLabel{$\mathsf{\rightarrow_R}$}
\DisplayProof\\\\
\footnotesize
\AXC{$\pi_0$}
\noLine
\UIC{$\vdots$}
\noLine
\UIC{$\Sigma;\Gamma\Rightarrow B_0\to \Box (C \wedge B_1) $}
\LeftLabel{$\mathsf{wk}$}
\UIC{$\Sigma;\Gamma, B_0\Rightarrow B_0\to \Box (C \wedge B_1), \Box^+ C $}
\AXC{$\theta_0$}
\noLine
\UIC{$\vdots$}
\noLine
\UIC{$\Sigma;\Sigma_1 ,C, B_1 \Rightarrow   C$}
\LeftLabel{$\mathsf{\wedge^\prime_L}$}
\doubleLine
\UIC{$\Sigma;\Sigma_1 ,C\wedge B_1 \Rightarrow   C$}
\AXC{$\Sigma;\Sigma_1 ,C, B_1 \Rightarrow   \Box^+ C$}
\LeftLabel{$\mathsf{\wedge^\prime_L}$}
\doubleLine
\UIC{$\Sigma;\Sigma_1 ,C\wedge B_1 \Rightarrow   \Box^+ C$}
\LeftLabel{$\Box^+$}
\BIC{$\Sigma;\Gamma, B_0,\Box (C \wedge  B_1) \Rightarrow  \Box^+ C$}
\AXC{$\sigma_0$}
\noLine
\UIC{$\vdots$}
\noLine
\UIC{$\Sigma;\Gamma, B_0\Rightarrow B_0, \Box^+ C $}
\LeftLabel{$\mathsf{\to_L}$}
\BIC{$\Sigma;\Gamma, B_0,B_0\to \Box (C \wedge  B_1) \Rightarrow  \Box^+ C $}
\LeftLabel{$\mathsf{cut}$}
\BIC{$\Sigma;\Gamma, B_0\Rightarrow \Box^+ C$}
\LeftLabel{$\mathsf{\rightarrow_R}$}
\RightLabel{ }
\UIC{$\Sigma;\Gamma\Rightarrow B_0 \rightarrow \Box^+ C$}
\DisplayProof
\end{gather*}
\medskip 
\end{minipage}
\center{\textbf{Fig. 2}}
\end{adjustwidth}
\end{sidewaystable}
\end{center}

An $\infty$-proof is called \emph{slim} if every application of a modal rule in the given $\infty$-proof is slim. 
\begin{lemma}\label{AtoA}
For any set of formulas $\Sigma$, any formula $A$ and any finite multisets of formulas $\Gamma$ and $\Delta$, we have $\mathsf{S} \vdash \Sigma;\Gamma,A\Rightarrow A,\Delta$. Moreover, the corresponding cut-free $\infty$-proof can be chosen slim.
\end{lemma}
\begin{proof}
This lemma is proved by induction on the structure of $A$ in a straightforward manner. We check only the case when $A$ has the form $ \Box^+  A_0$. 

By the induction hypothesis, $\mathsf{S} \vdash\Sigma; \Gamma_0,A_0\Rightarrow A_0,\Delta_0$ for any finite multisets of formulas $\Gamma_0$ and $\Delta_0$. Hence, we have a cut-free $\infty$-proof of the sequent $\Sigma;A_0,\Box^+ A_0 \Rightarrow \Box^+ A_0 $:
\[
\AXC{$\sigma$}
\noLine
\UIC{$\vdots$}
\noLine
\UIC{$\Sigma;A_0, \Box^+ A_0 \Rightarrow A_0$}
\AXC{$\pi$}
\noLine
\UIC{$\vdots$}
\noLine
\UIC{$\Sigma;A_0, \Box^+ A_0 \Rightarrow \Box^+ A_0$}
\LeftLabel{$\Box^+$}
\RightLabel{ ,}
\BIC{$\Sigma;A_0,\Box^+ A_0 \Rightarrow \Box^+ A_0 $}
\DisplayProof\] 
where $\sigma$ is a 
slim $\infty$-proof of $\Sigma;A_0,\Box^+ A_0 \Rightarrow  A_0 $ in $\mathsf{S}$, which exists by the induction hypothesis, and $\pi$ is isomorphic to the whole $\infty$-proof. Continuing this $\infty$-proof, we obtain 
\begin{gather*}
\AXC{$\sigma$}
\noLine
\UIC{$\vdots$}
\noLine
\UIC{$\Sigma;A_0, \Box^+A_0 \Rightarrow A_0$}
\AXC{$\Sigma;A_0, \Box^+A_0 \Rightarrow \Box^+ A_0$}
\LeftLabel{$\Box^+$}
\RightLabel{ }
\BIC{$\Sigma;\Gamma,\Box^+ A_0 \Rightarrow \Box^+ A_0 ,\Delta$}
\DisplayProof
\end{gather*}
for arbitrary finite multisets of formulas $\Gamma$ and $\Delta$. Note that the resulting $\infty$-proof is slim.

\end{proof}

Given an $\infty$-proof $\pi$, let us consider the tree $\hat{\pi}$ obtained from $\pi$ by cutting every branch at the first from the root premise of a modal rule. 
The conditions on infinite branches of $\infty$-proofs imply that the resulting tree $\hat{\pi}$ is finite. For example, the tree obtained from (\ref{Example}) consists of a single node 
\[\Sigma ; p, \Box p , \Box^+ F  \Rightarrow  \Box^+ p.\]
We define the \emph{local height $\lvert \pi \rvert$ of an $\infty$-proof $\pi$} as the length of the longest branch in $\hat{\pi}$. If $\hat{\pi}$ consists of a unique node, then $\lvert \pi \rvert=0$.

\begin{lemma}\label{L2}
The inference rule
\begin{gather*}
\AXC{$\Sigma;\Gamma \Rightarrow \Delta$}
\LeftLabel{$\mathsf{wk}$}
\RightLabel{ }
\UIC{$\Sigma;\Phi  , \Gamma \Rightarrow \Delta, \Psi$}
\DisplayProof 
\end{gather*}
is admissible in $\mathsf{S}+\mathsf{cut}$. 
\end{lemma}
\begin{proof}
Admissibility of the rule is standardly obtained by induction on the local height of an $\infty$-proof for $\Sigma;\Gamma \Rightarrow \Delta$. We omit further details.
\end{proof}

\begin{lemma}\label{omegatoinfcut}
If $\Sigma ;\Gamma \vdash_\omega A $, where $\Gamma$ is finite, then $\mathsf{S}+\mathsf{cut} \vdash \Sigma ; \Gamma \Rightarrow A$. 
\end{lemma}
\begin{proof}
Assume there is an $\omega$-derivation $\delta$ with the root marked by $A$ in which every boxed assumption leafs is marked by some formula from $\Sigma$ and every non-boxed assumption leaf is marked by some formula from $\Gamma$. In addition, assume that $\Gamma$ is finite. We show that the sequent $\Sigma ; \Gamma \Rightarrow A$ is provable in $\mathsf{S}+\mathsf{cut}$ by induction on the ordinal height of $\delta$.

Suppose $\delta$ consists of a single node. If this node is marked by a formula from $\Gamma$, then $A\in \Gamma$ and the sequent $\Sigma ; \Gamma \Rightarrow A$ is provable in $\mathsf{S}+\mathsf{cut} $ by 
Lemma \ref{AtoA}. Otherwise, the single node is marked by an axiom of $\mathsf{K}^+$. We check only the case of Axiom (v): $A$ has the form $\Box A_0 \wedge \Box^+ (A_0\to \Box A_0)\to \Box^+ A_0$. In this case, we have   
\begin{gather*}
\AXC{$\sigma_0$}
\noLine
\UIC{$\vdots$}
\noLine
\UIC{$\Sigma ; A_0 , G, \Box^+ G  \Rightarrow   A_0$}
\AXC{$\pi$}
\noLine
\UIC{$\vdots$}
\noLine
\UIC{$\Sigma ; A_0 ,  \Box A_0, \Box^+ G  \Rightarrow   \Box^+ A_0$}
\AXC{$\sigma_1$}
\noLine
\UIC{$\vdots$}
\noLine
\UIC{$\Sigma ; A_0 ,   \Box^+ G  \Rightarrow    A_0, \Box^+ A_0$}
\LeftLabel{$\mathsf{\rightarrow_L}$}
\BIC{$\Sigma ;A_0 , G, \Box^+ G  \Rightarrow   \Box^+ A_0$}
\LeftLabel{$\mathsf{\Box}^+$}
\RightLabel{ ,}
\BIC{$\Sigma ; A_0, \Box A_0 , \Box^+ G  \Rightarrow  \Box^+ A_0$}
\DisplayProof 
\end{gather*}
where $G= A_0 \rightarrow \Box A_0$, $\sigma_0$ and $\sigma_1$ are 
$\infty$-proofs, which exist by Lemma \ref{AtoA}, and the subtree $\pi$ is isomorphic to the whole $\infty$-proof. Continuing this $\infty$-proof, we obtain 
\begin{gather*}
\AXC{$\sigma_0$}
\noLine
\UIC{$\vdots$}
\noLine
\UIC{$\Sigma ; A_0 , G, \Box^+ G  \Rightarrow   A_0$}
\AXC{$\Sigma ; A_0 ,  \Box A_0, \Box^+ G  \Rightarrow   \Box^+ A_0$}
\AXC{$\sigma_1$}
\noLine
\UIC{$\vdots$}
\noLine
\UIC{$\Sigma ; A_0 ,   \Box^+ G  \Rightarrow    A_0, \Box^+ A_0$}
\LeftLabel{$\mathsf{\rightarrow_L}$}
\BIC{$\Sigma ;A_0 , G, \Box^+ G  \Rightarrow   \Box^+ A_0$}
\LeftLabel{$\mathsf{\Box}^+$}
\BIC{$\Sigma ; \Gamma, \Box A_0 , \Box^+ G  \Rightarrow  \bot,\Box^+ A_0$}
\LeftLabel{$\mathsf{\wedge^\prime_L}$}
\doubleLine
\UIC{$\Sigma ; \Gamma, \Box A_0 \wedge \Box^+ G  \Rightarrow  \Box^+ A_0$}
\LeftLabel{$\mathsf{\to_R}$}
\RightLabel{ .}
\UIC{$\Sigma ; \Gamma \Rightarrow \Box A_0 \wedge \Box^+ G  \to \Box^+ A_0$}
\DisplayProof 
\end{gather*}
Thus, the sequent $ \Sigma ; \Gamma \Rightarrow \Box A_0 \wedge \Box^+ (A_0\to \Box A_0)\to \Box^+ A_0$ is provable in $\mathsf{S}+\mathsf{cut} $.

Suppose $\delta$ has the form
\[\AXC{$\delta_0$}
\noLine
\UIC{$\vdots$}
\noLine
\UIC{$ B$}
\AXC{$\delta_1$}
\noLine
\UIC{$\vdots$}
\noLine
\UIC{$ B\to A$}
\LeftLabel{$\mathsf{mp}$}
\RightLabel{ .}
\BIC{$A$}
\DisplayProof \]
By the induction hypothesis, there are $\infty$-proofs $\pi_0$ and $\pi_1$ for the sequents $\Sigma ; \Gamma \Rightarrow  B$ and $\Sigma ; \Gamma \Rightarrow B\to A$ respectively. We have
\begin{gather*}
\AXC{$\pi_1$}
\noLine
\UIC{$\vdots$}
\noLine
\UIC{$\Sigma ; \Gamma \Rightarrow B\to A$}
\LeftLabel{$\mathsf{wk}$}
\UIC{$\Sigma ; \Gamma \Rightarrow B\to A,A$}
\AXC{$\sigma$}
\noLine
\UIC{$\vdots$}
\noLine
\UIC{$\Sigma ; \Gamma ,  A\Rightarrow A$}
\AXC{$\pi_0$}
\noLine
\UIC{$\vdots$}
\noLine
\UIC{$\Sigma ; \Gamma \Rightarrow B$}
\LeftLabel{$\mathsf{wk}$}
\UIC{$\Sigma ; \Gamma \Rightarrow B,A$}
\LeftLabel{$\mathsf{\to_L}$}
\BIC{$\Sigma ; \Gamma , B\to A\Rightarrow A$}
\LeftLabel{$\mathsf{cut}$}
\RightLabel{ ,}
\BIC{$\Sigma ; \Gamma \Rightarrow A$}
\DisplayProof 
\end{gather*}
where $\sigma$ is an 
$\infty$-proof of $\Sigma ; \Gamma ,  A\Rightarrow A$, which exists by Lemma \ref{AtoA}.
Hence, $\mathsf{S}+\mathsf{cut} \vdash \Sigma ; \Gamma \Rightarrow A$. 

Suppose $\delta$ has the form
\[
\AXC{$\delta_0$}
\noLine
\UIC{$\vdots$}
\noLine
\UIC{$A_0$}
\LeftLabel{$\mathsf{nec}$}
\RightLabel{ .}
\UIC{$ \Box^+ A_0$}
\DisplayProof\]
In this case, all assumption leaves of $\delta_0$ are marked by some formulas from $\Sigma$. Note that $\delta_0$ contains only finitely many non-boxed assumption leaves. Applying the induction hypothesis, we can find an $\infty$-proof $\pi_0$ of a sequent $\Sigma ; \Sigma_0 \Rightarrow  A_0$, where $\Sigma_0\subset \Sigma$. We have
\[
\AXC{$\pi_0$}
\noLine
\UIC{$\vdots$}
\noLine
\UIC{$\Sigma; \Sigma_0\Rightarrow A_0$}
\AXC{$\pi$}
\noLine
\UIC{$\vdots$}
\noLine
\UIC{$\Sigma;\Sigma_0\Rightarrow \Box^+ A_0$}
\LeftLabel{$\Box^+$}
\RightLabel{ ,}
\BIC{$\Sigma;\Sigma_0\Rightarrow \Box^+ A_0$}
\DisplayProof\]
where $\pi$ is isomorphic to the whole $\infty$-proof. Continuing this $\infty$-proof, we obtain
\[
\AXC{$\pi_0$}
\noLine
\UIC{$\vdots$}
\noLine
\UIC{$\Sigma; \Sigma_0\Rightarrow A_0$}
\AXC{$\Sigma;\Sigma_0\Rightarrow \Box^+ A_0$}
\LeftLabel{$\Box^+$}
\RightLabel{ .}
\BIC{$\Sigma;\Gamma\Rightarrow \Box^+ A_0$}
\DisplayProof\]
Therefore, $\mathsf{S}+\mathsf{cut} \vdash \Sigma;\Gamma\Rightarrow \Box^+ A_0$. 


It remains to consider the case when $\delta$ has the form
\[
\AXC{$\delta_0$}
\noLine
\UIC{$\vdots$}
\noLine
\UIC{$B_0 \rightarrow \Box (C \wedge B_1)$}
\AXC{$\delta_1$}
\noLine
\UIC{$\vdots$}
\noLine
\UIC{$B_1 \rightarrow \Box (C \wedge B_2)$}
\AXC{$\delta_2$}
\noLine
\UIC{$\vdots$}
\noLine
\UIC{$B_2 \rightarrow \Box (C \wedge B_3)$}
\AXC{$\dots$}
\LeftLabel{$\omega$}
\RightLabel{ .}
\QIC{$B_0 \rightarrow \Box^+ C$}
\DisplayProof 
\]
Notice that $\Sigma;\Gamma \vdash_\omega B_0 \to \Box (C \wedge  B_1)$ and, for each $i \geqslant 1$, there is a finite subset $\Sigma_i$ of $\Sigma$ such that $\Sigma;\Sigma_i \vdash_\omega B_i \to \Box (C \wedge B_{i+1})$. By the induction hypothesis, there are an $\infty$-proof $\pi_0$ for the sequent $\Sigma ; \Sigma_0 \Rightarrow  B_0 \to \Box (C \wedge B_1)$ and $\infty$-proofs $\pi_i$ for the sequents $\Sigma;\Sigma_i \vdash_\omega B_i \to \Box (C \wedge B_{i+1})$. Consider the proof fragments of Fig. 2, where $\theta_i$ and $\sigma_i$ are $\infty$-proofs, which exist by Lemma \ref{AtoA}. Gluing these fragments at the nodes marked by $\Sigma;\Sigma_i, C,B_i\Rightarrow \Box^+ C$, we obtain an $\infty$-proof for the sequent $\Sigma;\Gamma\Rightarrow B_0 \rightarrow \Box^+ C$. Hence, $\mathsf{S}+\mathsf{cut} \vdash \Sigma ; \Gamma \Rightarrow A$, which concludes the proof.
\end{proof}

\section{Annotations, global heights and fragments}
\label{s4}
In this section, we introduce annotated versions of sequents, inference rules and $\infty$-proofs in order to facilitate our treatment of the calculus $\mathsf{S}+\mathsf{cut}$. We also define several important characteristics of $\infty$-proofs and show that any  $\infty$-proof can be transformed into an $\omega$-derivation of the Frege-Hilbert calculus of the logic $\mathsf{K}^+$.
 
An \emph{annotated sequent} is an expression of the form $\Sigma ;\Gamma \Rightarrow_s \Delta$, where $\Sigma ;\Gamma \Rightarrow \Delta$ is an ordinary sequent and the annotation $s$ is a formula or an auxiliary sign $\circ$. Also, if $s$ is a formula, then the musltiset $\Delta$ must contain $\Box^+ s$. In addition, we put $\mathit{Fm}_ \circ:=\mathit{Fm} \cup \{\circ\}$. Besides, for an $\infty$-proof $\pi$ of a sequent $\Sigma;\Gamma \Rightarrow \Delta$, we denote the finite subset of the set $\mathit{Fm}_\circ$ consisting  of all possible annotations of $\Sigma;\Gamma \Rightarrow \Delta$ by $\mathit{Ann}(\pi)$.  

Annotated versions of initial sequents and inference rules are defined as
\begin{gather*}
\AXC{ $\Sigma ;\Gamma, p \Rightarrow_s p, \Delta, $}
\DisplayProof  \qquad
\AXC{ $\Sigma ;\Gamma , \bot \Rightarrow_s  \Delta,$}
\DisplayProof 
\end{gather*}
\begin{gather*}
\AXC{$\Sigma ;\Gamma , B \Rightarrow_s  \Delta$}
\AXC{$\Sigma ;\Gamma \Rightarrow_s  A, \Delta$}
\LeftLabel{$\mathsf{\rightarrow_L}$}
\BIC{$\Sigma ;\Gamma , A \rightarrow B \Rightarrow_s  \Delta$}
\DisplayProof \;,\quad
\AXC{$\Sigma ;\Gamma, A \Rightarrow_s  B ,\Delta$}
\LeftLabel{$\mathsf{\rightarrow_R}$}
\UIC{$\Sigma ;\Gamma \Rightarrow_s  A \rightarrow B ,\Delta$}
\DisplayProof \;,
\end{gather*}
\begin{gather*}
\AXC{$\Sigma ;\Sigma_0,\Lambda, \Pi, \Box^+ \Pi \Rightarrow_\circ A$}
\LeftLabel{$\mathsf{\Box}$}
\UIC{$\Sigma ; \Phi, \Box \Lambda, \Box^+ \Pi \Rightarrow_s \Box A , \Psi$}
\DisplayProof \;,\\\\
\AXC{$\Sigma ;\Sigma_0,\Lambda, \Pi, \Box^+ \Pi \Rightarrow_\circ A$}
\AXC{$\Sigma ;\Sigma_0,\Lambda, \Pi, \Box^+ \Pi \Rightarrow_A \Box^+ A$}
\LeftLabel{$\Box^+$}
\BIC{$ \Sigma ;\Phi, \Box \Lambda, \Box^+ \Pi \Rightarrow_s  \Box^+ A ,\Psi$}
\DisplayProof \;,
\end{gather*}
\begin{gather*}
\AXC{$\Sigma ;\Gamma\Rightarrow_s \Delta, A$}
\AXC{$\Sigma ;A,\Gamma\Rightarrow_s\Delta$}
\LeftLabel{$\mathsf{cut}$}
\RightLabel{ ,}
\BIC{$\Sigma ;\Gamma\Rightarrow_s\Delta$}
\DisplayProof 
\end{gather*}
where $\Sigma_0$ is a finite subset of $\Sigma$.


An \emph{annotated $\infty$-proof} is a (possibly infinite) tree whose nodes are marked by annotated sequents and that is constructed according to annotated versions of inference rules. In addition, all leaves in an annotated $\infty$-proof are marked by annotated initial sequents, and every infinite branch in it must contain a tail satisfying the following conditions: all sequents in the tail are annotated with the same subscript formula $A$; the tail intersects applications of the annotated rule ($\mathsf{\Box^+}$) on right premises infinitely many times. An annotated $\infty$-proof is \emph{regular} if it contains only finitely many non-isomorphic subtrees with respect to annotations.


Notice that if we erase all annotations in an annotated $\infty$-proof, then the resulting tree is an ordinary $\infty$-proof. For instance, we obtain (\ref{Example}) if we erase annotations in the annotated $\infty$-proof below:
\begin{gather}\label{Example2}
\AXC{$\mathsf{Ax}$}
\noLine
\UIC{$\Sigma ; p , F,   \Box^+ F  \Rightarrow_\circ   p$}
\AXC{$\kappa$}
\noLine
\UIC{$\vdots$}
\noLine
\UIC{$\Sigma ; p ,  \Box p, \Box^+ F  \Rightarrow_p   \Box^+ p$}
\AXC{$\mathsf{Ax}$}
\noLine
\UIC{$\Sigma ; p ,   \Box^+ F  \Rightarrow_p    p, \Box^+ p$}
\LeftLabel{$\mathsf{\rightarrow_L}$}
\BIC{$\Sigma ; p , F, \Box^+ F  \Rightarrow_p   \Box^+ p$}
\LeftLabel{$\mathsf{\Box}^+$}
\RightLabel{ ,}
\BIC{$\Sigma ; p, \Box p , \Box^+ F  \Rightarrow_p  \Box^+ p$}
\DisplayProof 
\end{gather}
where $F= p \to \Box p$ and the subtree $\kappa$ is isomorphic to the whole annotated $\infty$-proof.

On the other hand, given an ordinary $\infty$-proof $\pi$ of a sequent $\Sigma;\Gamma \Rightarrow \Delta$, we can annotate the sequent with $s\in \mathit{Ann}(\pi)$ and, moving upwards away from the root of $\pi$, replace all applications of inference rules in $\pi$ with its annotated versions. Clearly, this procedure gives us an annotated $\infty$-proof of $\Sigma;\Gamma \Rightarrow_s \Delta$, which we denote by $\pi^s$. Notice that, for any regular $\infty$-proof $\pi$ and any $s\in \mathit{Ann}(\pi)$, the annotated $\infty$-proof $\pi^s$ is also regular.

Now we introduce an important notion of global height of an annotated $\infty$-proof.
For nodes $a$ and $b$ of an annotated $\infty$-proof $\kappa$, we set $a \approx_\kappa b$ if the shortest path from $a$ to $b$ in $\kappa$ satisfies the following conditions: it doesn't intersect applications of the rule ($\Box$); it doesn't intersect applications of the rule ($\Box^+$) on left premises; all sequents on the path have the same annotation. Obviously, $\approx_\kappa$ is an equivalence relation on the set of nodes of $\kappa$. Let us consider the quotient set of $\kappa$ by $\approx_\kappa$ as a graph, where an equivalence class $c_1$ is adjacent to another equivalence class $c_2$ if some element of $c_1$ is adjacent to some element of $c_2$ in $\kappa$. Since every equivalence class is a connected subset of $\kappa$, the given graph is a tree, denoted by $\mathit{tr}(\kappa)$, with the root inherited from $\kappa$. 


\begin{propos}
For any annotated $\infty$-proof $\kappa$, the tree $\mathit{tr}(\kappa)$ is well-founded, i.e. there is no infinite branch in $\mathit{tr}(\kappa)$.
\end{propos}
\begin{proof}
We show that $\mathit{tr}(\kappa)$ is well-founded by \emph{reductio ad absurdum}. Suppose there is an infinite branch in $\mathit{tr}(\kappa)$
\[c_0 \rightarrow c_1 \rightarrow c_2 \rightarrow  \dotsb,\]
where the edges directed away from the root.
Then we have an infinite branch in $\kappa$
\[a_{0,0} \rightarrow a_{0,1} \rightarrow \dotsb \rightarrow a_{0,k_0} \rightarrow a_{1,0}\rightarrow \dotsb \rightarrow a_{1,k_1} \rightarrow a_{2,0} \rightarrow \dotsb\]
such that, for each $i\in \mathbb{N}$, the nodes $a_{i,0}, \dotsc, a_{i, k_i}$ belong to $c_i$.
Since $\kappa$ is an annotated $\infty$-proof, this branch of $\kappa$ must contain a tail such that all sequents in the tail are annotated with the same subscript formula $A$ and the tail passes through the right premise of the rule ($\mathsf{\Box^+}$) infinitely many times. We see that all nodes in the tail are equivalent to each other, which is a contradiction since $c_i \neq c_{j}$ for $i \neq j$. 
Consequently, the tree $\mathit{tr}(\kappa)$ is well-founded.
\end{proof}
 
The \emph{global height $\lVert \kappa \rVert$} of an annotated $\infty$-proof $\kappa$ is defined as the ordinal height of (the root of) $\mathit{tr}(\kappa)$. If $\mathit{tr}(\kappa)$ has only one node, then $\lVert \kappa \rVert=0$. We note that the height $\lVert \kappa \rVert$ is at most countable ordinal since $\mathit{tr}(\kappa)$ has at most a countable number of nodes. For an $\infty$-proof $\pi$ and $s \in \mathit{Ann}(\pi) $, we define the \emph{$s$-height of $\pi$} by putting $\lVert \pi \rVert_s\coloneq\lVert \pi^s \rVert$. The \emph{global height of $\pi$} is defined by setting $\lVert \pi \rVert \coloneq \min\{\lVert \pi \rVert_s  \mid s\in \mathit{Ann}(\pi)\}$.



\begin{lemma}\label{annotation lemma}
For any $\infty$-proof $\pi$ and any $s_1,s_2\in \mathit{Ann}(\pi)$, we have $\lVert \pi \rVert_{s_1} \leqslant \lVert \pi \rVert_{s_2} +1$. 
\end{lemma} 
\begin{proof}
First, suppose $s_1=\circ$. In this case, $\mathit{tr}(\pi^{s_2})$ can obtained from $\mathit{tr}(\pi^{s_1})$ by merging the root of $\mathit{tr}(\pi^{s_1})$ with some (possibly none) of its children. 
Trivially, we have $\lVert \pi \rVert_{s_2} \leqslant \lVert \pi \rVert_{s_1} \leqslant \lVert \pi \rVert_{s_2}+1$. 


Now we consider the general case of arbitrary $s_1,s_2\in\mathit{Ann}(\pi)$. From the previous case, we see $\lVert \pi \rVert_{s_1} \leqslant \lVert \pi \rVert_\circ \leqslant \lVert \pi \rVert_{s_1}+1$ and $\lVert \pi \rVert_{s_2} \leqslant \lVert \pi \rVert_\circ \leqslant \lVert \pi \rVert_{s_2}+1$. Consequently, we obtain $\lVert \pi \rVert_{s_1} \leqslant \lVert \pi \rVert_\circ \leqslant \lVert \pi \rVert_{s_2}+1$, which concludes the proof.
\end{proof}

Now, for an annotated $\infty$-proof $\kappa$, we define a series of its fragments approximating $\kappa$. Let us consider an application of the rule ($\Box^+$) in $\kappa$ whose right premise is equivalent to the root of $\kappa$ with respect to $\approx_\kappa$.
This application has \emph{level $n$}, for $n>0$, if the path from the root of $\kappa$ to its right premise passes through exactly $n$ applications of the rule ($\Box^+$). 
We define the \emph{$n$-fragment of $\kappa$} as the tree obtained from $\kappa$ by cutting off every right premise of the rule ($\Box^+$) of level $n$ so that this premise and all its descendants are removed. For example, consider the $2$-fragment of (\ref{Example2}):
\begin{gather*}
\AXC{$\mathsf{Ax}$}
\noLine
\UIC{$\Sigma ;p , F,   \Box^+ F  \Rightarrow_\circ   p$}
\AXC{$\mathsf{Ax}$}
\noLine
\UIC{$\Sigma ; p , F,   \Box^+ F  \Rightarrow_\circ   p$}
\AXC{$\qquad \qquad \quad$}
\LeftLabel{$\mathsf{\Box}^+$}
\BIC{$\Sigma ; p, \Box p , \Box^+ F  \Rightarrow_p  \Box^+ p$}
\AXC{$\mathsf{Ax}$}
\noLine
\UIC{$\Sigma ; p ,   \Box^+ F  \Rightarrow_p    p, \Box^+ p$}
\LeftLabel{$\mathsf{\rightarrow_L}$}
\BIC{$\Sigma ; p , F, \Box^+ F  \Rightarrow_p   \Box^+ p$}
\LeftLabel{$\mathsf{\Box}^+$}
\RightLabel{ .}
\BIC{$\Sigma ; p, \Box p , \Box^+ F  \Rightarrow_p  \Box^+ p$}
\DisplayProof 
\end{gather*}

By $\mathsf{P}$, we denote the set of all $\infty$-proofs of the calculus $\mathsf{S}=\mathsf{cut}$. In addition, we put $\mathsf{P}(s)\coloneq \{\pi\in \mathsf{P}\mid s\in \mathit{Ann}(\pi)\}$. For $\pi, \tau \in \mathsf{P}(s)$ and $n>0$, we write $\pi \sim^s_n \tau$ if the $n$-fragments of annotated $\infty$-proofs $\pi^s$ and $\tau^s$ coincide. Also, we set $\pi \sim^s_0 \tau$ for all $\pi, \tau \in \mathsf{P}(s)$. 
Notice that, for $s\in \mathit{Fm}_\circ$, $(\sim^s_n)_{n\in\mathbb{N}}$ is a family of equivalence relations on $\mathsf{P}(s)$. Besides, $\pi \sim^s_n \tau$ if $\pi \sim^s_{n+1} \tau$, and $\pi = \tau$ whenever $\pi \sim^s_n \tau$ for every $n\in \mathbb{N}$. Note also that $\pi=\tau$ whenever $\pi \sim^s_1 \tau$ and $\pi \sim^w_1 \tau$ for $ s\neq w$.

By means of equivalence relations $(\sim^s_n)_{n\in\mathbb{N}}$, we define limits of sequences of $\infty$-proofs.

\begin{propos}\label{Completeness1} For any $s\in \mathit{Fm}_\circ$ and any sequence of $\infty$-proofs
\[\pi_0 \sim^s_0 \pi_1 \sim^s_1 \pi_2 \sim^s_2 \dotso,\]
there exists a unique $\infty$-proof denoted by $\lim\limits_{i\to \infty} \pi_i$ such that $\pi_n \sim^s_n \lim\limits_{i \to \infty} \pi_i $ for each $n\in \mathbb{N}$. Moreover, $\lim\limits_{i\to \infty} \pi_i$ does not depend on $s$. In other words, if additionally
\[\pi_0 \sim^w_0 \pi_1 \sim^w_1 \pi_2 \sim^w_2 \dotso\]
for $w\neq s$, then $\pi_n \sim^w_n \lim\limits_{i \to \infty} \pi_i $ for each $n\in \mathbb{N}$. Finally, $\mathit{Ann}(\lim\limits_{i\to \infty} \pi_i)= \mathit{Ann}(\pi_1)$. 
\end{propos} 
\begin{proof}
Trivially, for a sequence of $\infty$-proofs
\[\pi_0 \sim^s_0 \pi_1 \sim^s_1 \pi_2 \sim^s_2 \dotso,\]
there exists the required $\infty$-proof $\lim\limits_{i\to \infty} \pi_i$, and $\lim\limits_{i\to \infty} \pi_i$ is uniquely defined. 

In addition, if
\[\pi_0 \sim^w_0 \pi_1 \sim^w_1 \pi_2 \sim^w_2 \dotso\]
for $w\neq s$, then $\pi_1 = \pi_2 = \dotso= \lim\limits_{i\to \infty} \pi_i$. Consequently, $\pi_n \sim^w_n \lim\limits_{i \to \infty} \pi_i $ for each $n\in \mathbb{N}$. 

Since $\pi_1 \sim^s_1\lim\limits_{i\to \infty} \pi_i$, we obtain $\mathit{Ann}(\lim\limits_{i\to \infty} \pi_i)= \mathit{Ann}(\pi_1)$.

\end{proof}

In the rest of the section, we show that any $\infty$-proof can be transformed into an $\omega$-derivation of the logic $\mathsf{K}^+$.
\begin{lemma}\label{inftoomega}
Suppose there is an annotated $\infty$-proof $\kappa$ of a sequent $\Sigma ; \Gamma \Rightarrow_s \Delta$. Then $\Sigma ;\emptyset \vdash_\omega \bigwedge \Gamma\to \bigvee \Delta $.
\end{lemma}
\begin{proof}
We prove that $\Sigma ;\emptyset \vdash_\omega \bigwedge \Gamma\to \bigvee \Delta $ by induction on $\lVert \kappa \rVert$.

Case 1. Suppose $s=\circ$. We proceed by subinduction on the local height of $\kappa$. 

If $\kappa$ consists of a single node, then $\Sigma ; \Gamma \Rightarrow_s \Delta$ is an initial sequent. We see that $\mathsf{K}^+\vdash \bigwedge \Gamma\to \bigvee \Delta  $ and $\Sigma ;\emptyset \vdash_\omega \bigwedge \Gamma\to \bigvee \Delta $.

If $\kappa$ has the form 
\[\AXC{$\kappa_0$}
\noLine
\UIC{$\vdots$}
\noLine
\UIC{$\Sigma;\Sigma_0, \Lambda, \Pi, \Box^+ \Pi \Rightarrow_\circ  B$}
\LeftLabel{$\Box$}
\RightLabel{ ,}
\UIC{$\Sigma;\Phi ,\Box \Lambda, \Box^+ \Pi \Rightarrow_\circ \Box B,  \Psi$}
\DisplayProof\]
then, by the induction hypothesis, $\Sigma ;\emptyset \vdash_\omega \bigwedge (\Sigma_0\cup\Lambda\cup \Pi
\cup \Box^+ \Pi) \to B$, $\Sigma ;\Sigma  \vdash_\omega \bigwedge (\Lambda\cup \Pi
\cup \Box^+ \Pi) \to B$ , $\Sigma ;\emptyset \vdash_\omega \Box^+ (\bigwedge (\Lambda\cup \Pi
\cup \Box^+ \Pi) \to B)$ and $\Sigma ;\emptyset \vdash_\omega \Box (\bigwedge (\Lambda\cup \Pi
\cup \Box^+ \Pi) \to B)$. In addition, we see that $\mathsf{K}^+\vdash  \Box (\bigwedge (\Lambda\cup \Pi
\cup \Box^+ \Pi) \to B)\to (\bigwedge (\Phi
\cup
\Box \Lambda\cup \Box^+ \Pi)\to \bigvee (\{\Box B\}\cup \Psi) )$. Hence, $\Sigma ;\emptyset \vdash_\omega \bigwedge (\Phi
\cup
\Box \Lambda\cup \Box^+ \Pi)\to \bigvee (\{\Box B\}\cup \Psi) $.

Suppose $\kappa$ has the form 
\[\AXC{$\kappa_0$}
\noLine
\UIC{$\vdots$}
\noLine
\UIC{$\Sigma; \Sigma_0, \Lambda, \Pi, \Box^+ \Pi \Rightarrow_\circ  B$}
\AXC{$\kappa_1$}
\noLine
\UIC{$\vdots$}
\noLine
\UIC{$\Sigma; \Sigma_0,\Lambda, \Pi, \Box^+ \Pi \Rightarrow_B \Box^+ B$}
\LeftLabel{$\Box^+$}
\RightLabel{ .}
\BIC{$\Sigma;\Phi, \Box \Lambda, \Box^+ \Pi \Rightarrow_\circ \Box^+ B,  \Psi$}
\DisplayProof 
\]
Notice that $\lVert \kappa_0\rVert <\lVert \kappa\rVert$ and $\lVert \kappa_1\rVert <\lVert \kappa\rVert$. From the induction hypothesis, $\Sigma ;\emptyset \vdash_\omega \bigwedge (\Sigma_0\cup\Lambda\cup \Pi
\cup \Box^+ \Pi) \to B$ and $\Sigma ;\emptyset \vdash_\omega \bigwedge (\Sigma_0\cup\Lambda\cup \Pi
\cup \Box^+ \Pi) \to \Box^+ B$. Hence, $\Sigma ;\emptyset \vdash_\omega \bigwedge (\Sigma_0\cup\Lambda\cup \Pi
\cup \Box^+ \Pi) \to B\wedge \Box^+ B$ and $\Sigma ;\emptyset \vdash_\omega \Box (\bigwedge (\Lambda\cup \Pi
\cup \Box^+ \Pi) \to B\wedge \Box^+ B)$. Since $\mathsf{K}^+\vdash \Box (B\wedge \Box^+ B)\to \Box^+ B$, we obtain $\Sigma ;\emptyset \vdash_\omega \bigwedge (\Phi\cup\Box \Lambda\cup \Box^+ \Pi)\to \bigvee (\{\Box^+ B\}\cup \Psi) $. 

Suppose $\kappa$ has the form  
\[
\AXC{$\kappa_0$}
\noLine
\UIC{$\vdots$}
\noLine
\UIC{$\Sigma;\Lambda , B \Rightarrow_\circ  \Delta$}
\AXC{$\kappa_1$}
\noLine
\UIC{$\vdots$}
\noLine
\UIC{$\Sigma;\Lambda \Rightarrow_\circ  A, \Delta$}
\LeftLabel{$\mathsf{\rightarrow_L}$}
\RightLabel{ .}
\BIC{$\Sigma;\Lambda , A \rightarrow B \Rightarrow_\circ  \Delta$}
\DisplayProof 
\]
Since $\lvert \kappa_0\rvert <\lvert \kappa\rvert$ and $\lvert \kappa_1\rvert <\lvert \kappa\rvert$, $\Sigma ;\emptyset \vdash_\omega \bigwedge (\Lambda\cup \{B\})\to \bigvee\Delta$ and $\Sigma ;\emptyset \vdash_\omega \bigwedge \Lambda\to \bigvee(\{A\}\cup\Delta)$ by the subinduction hypothesis for $\kappa_0$ and $\kappa_1$. We see $\mathsf{K}^+\vdash (\bigwedge (\Lambda\cup \{B\})\to \bigvee\Delta)\wedge (\bigwedge \Lambda\to \bigvee(\{A\}\cup\Delta))\to \bigwedge (\Lambda\cup \{A\to B\})\to \bigvee\Delta$. Therefore, $\Sigma ;\emptyset \vdash_\omega \bigwedge (\Lambda\cup \{A\to B\})\to \bigvee\Delta$. 

We omit the subcases of inference rules ($\mathsf{\rightarrow_R}$) and ($\mathsf{cut}$), because they are analogous to the subcase of the rule ($\mathsf{\rightarrow_L}$).

Case 2. Suppose $s=C\in \mathit{Fm}$. Let $R$ be the equivalence class of the root of $\kappa$ with respect to $\approx_\kappa$, i.e., the root of $\mathit{tr}(\kappa)$. Note that, for any $a\in R$, the sequent of the node $a$ has the form $\Sigma;\Gamma_a\Rightarrow_C\Delta_a, \Box^+ C$. Besides, for any $a\in R$, we put $G_a\coloneq \bigwedge \Gamma_a\wedge \neg
\bigvee \Delta_a$ and $\mathit{rk}(a) \coloneq \lvert\kappa_a \rvert $, where $\kappa_a$ is the subtree of $\kappa$ with the root $a$. Trivially, $\mathsf{K}^+\vdash \bigwedge \Gamma_a\to \bigvee (\Delta_a\cup \{G_a\})$ and \begin{gather}\label{form1}
\Sigma ;\emptyset \vdash_\omega \bigwedge \Gamma_a\to \bigvee (\Delta_a\cup \{G_a\}).
\end{gather}

Recall that a node $a$ from $R$ belongs to the $n$-fragment of $\kappa$ if there are precisely $n-1$ applications of the rule ($\Box^+$) on the path from this node to the root of $\kappa$. Let $H_i\coloneq \bigvee \{G_a\mid a\in R \text{ and $a$ belongs to the ($i+1$)-fragment of $\kappa$}\}$.

We claim that $\Sigma ;\emptyset \vdash_\omega  H_i\to \Box (C\wedge H_{i+1}) $ for any $i\in \mathbb{N}$. It is sufficient to show that $\Sigma ;\emptyset \vdash_\omega  G_a\to \Box (C\wedge H_{i+1}) $ whenever $a$ belongs to the ($i+1$)-fragment of $\kappa$ and $a\in R$. We proceed by induction on $\mathit{rk}(a)$.

If $\kappa_a$ consists of a single node, then $\Sigma ; \Gamma_a \Rightarrow_C \Delta_a, \Box^+ C$ is an annotated initial sequent. We see that $\mathsf{K}^+\vdash \bigwedge \Gamma_a\to \bigvee \Delta_a  $, $\mathsf{K}^+\vdash \neg G_a$, $\Sigma ;\emptyset \vdash_\omega\neg G_a$ and $G_a\to \Box (C\wedge H_{i+1})  $.

Suppose $\kappa_a$ has the form 
\[\AXC{$\kappa^\prime_a$}
\noLine
\UIC{$\vdots$}
\noLine
\UIC{$\Sigma;\Sigma_0, \Lambda, \Pi, \Box^+ \Pi \Rightarrow_\circ  B$}
\LeftLabel{$\Box$}
\RightLabel{ .}
\UIC{$\Sigma;\Phi ,\Box \Lambda, \Box^+ \Pi \Rightarrow_C \Box B,  \Psi, \Box^+C$}
\DisplayProof\]
We see that $\lVert \kappa^\prime_a\rVert <\lVert \kappa_a\rVert\leqslant\lVert \kappa\rVert$. By the induction hypothesis, $\Sigma ;\emptyset \vdash_\omega \bigwedge (\Sigma_0\cup\Lambda\cup \Pi\cup \Box^+ \Pi) \to B$, $\Sigma ;\Sigma \vdash_\omega \bigwedge (\Lambda\cup \Pi \cup \Box^+ \Pi) \to B$ and $\Sigma ;\emptyset \vdash_\omega \bigwedge (\Box \Lambda\cup  \Box^+ \Pi) \to\Box  B$. Hence, $\Sigma ;\emptyset \vdash_\omega \bigwedge (\Phi\cup\Box \Lambda\cup \Box^+ \Pi)\to \bigvee (\{\Box B\}\cup \Psi\cup \{\Box (C\wedge H_{i+1})\}) $. Note that $\Gamma_a=\Phi,\Box \Lambda, \Box^+ \Pi$ and $\Delta_a=\Box B, \Psi$. Therefore, $\Sigma ;\emptyset \vdash_\omega G_a\to \Box (C\wedge H_{i+1})  $.

Suppose $\kappa_a$ has the form 
\[\AXC{$\kappa^\prime_a$}
\noLine
\UIC{$\vdots$}
\noLine
\UIC{$\Sigma; \Sigma_0, \Lambda, \Pi, \Box^+ \Pi \Rightarrow_\circ  B$}
\AXC{$\kappa^{\prime\prime}_a$}
\noLine
\UIC{$\vdots$}
\noLine
\UIC{$\Sigma; \Sigma_0,\Lambda, \Pi, \Box^+ \Pi \Rightarrow_B \Box^+ B$}
\LeftLabel{$\Box^+$}
\RightLabel{ ,}
\BIC{$\Sigma;\Phi, \Box \Lambda, \Box^+ \Pi \Rightarrow_C \Box^+ B,  \Psi, \Box^+ C$}
\DisplayProof 
\]
where $C\neq B$. Notice that $\lVert \kappa^\prime_a\rVert <\lVert \kappa_a\rVert\leqslant\lVert \kappa\rVert$ and $\lVert \kappa^{\prime\prime}_a\rVert <\lVert \kappa_a\rVert\leqslant\lVert \kappa\rVert$. From the induction hypothesis, $\Sigma ;\emptyset \vdash_\omega \bigwedge (\Sigma_0\cup\Lambda\cup \Pi
\cup \Box^+ \Pi) \to B$ and $\Sigma ;\Sigma \vdash_\omega \bigwedge (\Sigma_0\cup\Lambda\cup \Pi
\cup \Box^+ \Pi) \to \Box^+ B$. Hence, $\Sigma ;\emptyset \vdash_\omega \bigwedge (\Box\Lambda\cup  \Box^+ \Pi) \to \Box^+ B$ and $\Sigma ;\emptyset \vdash_\omega \bigwedge (\Phi\cup\Box \Lambda\cup \Box^+ \Pi)\to \bigvee (\{\Box^+ B\}\cup \Psi\cup \{\Box (C\wedge H_{i+1}\}) $.  Note that $\Gamma_a=\Phi,\Box \Lambda, \Box^+ \Pi$ and $\Delta_a=\Box^+ B, \Psi$. Consequently, $\Sigma ;\emptyset \vdash_\omega G_a\to \Box (C\wedge H_{i+1})  $.

Suppose $\kappa_a$ has the form 
\[\AXC{$\kappa^\prime_a$}
\noLine
\UIC{$\vdots$}
\noLine
\UIC{$\Sigma; \Sigma_0, \Lambda, \Pi, \Box^+ \Pi \Rightarrow_\circ  C$}
\AXC{$\kappa^{\prime\prime}_a$}
\noLine
\UIC{$\vdots$}
\noLine
\UIC{$\Sigma; \Sigma_0,\Lambda, \Pi, \Box^+ \Pi \Rightarrow_C \Box^+ C$}
\LeftLabel{$\Box^+$}
\RightLabel{ .}
\BIC{$\Sigma;\Phi, \Box \Lambda, \Box^+ \Pi \Rightarrow_C   \Delta_a, \Box^+ C$}
\DisplayProof 
\]
Since $\lVert \kappa^\prime_a\rVert <\lVert \kappa_a\rVert\leqslant\lVert \kappa\rVert$, $\Sigma ;\emptyset \vdash_\omega \bigwedge (\Sigma_0\cup\Lambda\cup \Pi \cup \Box^+ \Pi) \to C$. Let us denote the right premise of $a$ by $b$. From (\ref{form1}), we have $\Sigma ;\emptyset \vdash_\omega \bigwedge \Gamma_b\to \bigvee (\Delta_b\cup \{G_b\})$, i.e., $\Sigma ;\emptyset \vdash_\omega \bigwedge (\Sigma_0\cup\Lambda\cup \Pi \cup \Box^+ \Pi) \to G_b$. Since $\mathsf{K}^+\vdash G_b\to H_{i+1}$, we obtain $\Sigma ;\emptyset \vdash_\omega \bigwedge (\Sigma_0\cup\Lambda\cup \Pi \cup \Box^+ \Pi) \to H_{i+1}$. Therefore, $\Sigma ;\emptyset \vdash_\omega \bigwedge (\Sigma_0\cup\Lambda\cup \Pi \cup \Box^+ \Pi) \to C\wedge H_{i+1}$, $\Sigma ;\Sigma \vdash_\omega \bigwedge (\Lambda\cup \Pi \cup \Box^+ \Pi) \to C\wedge H_{i+1}$ and $\Sigma ;\emptyset \vdash_\omega \bigwedge (\Box\Lambda\cup  \Box^+ \Pi) \to \Box (C\wedge H_{i+1})$. We see that $\Sigma ;\emptyset \vdash_\omega \bigwedge (\Phi\cup\Box \Lambda\cup \Box^+ \Pi)\to \bigvee (\Delta_a\cup \{\Box (C\wedge H_{i+1}\}) $ and $\Sigma ;\emptyset \vdash_\omega G_a\to \Box (C\wedge H_{i+1})  $.

Suppose $\kappa_a$ has the form  
\[
\AXC{$\kappa^{\prime}_a$}
\noLine
\UIC{$\vdots$}
\noLine
\UIC{$\Sigma;\Gamma_a, A  \Rightarrow_C B ,\Lambda,\Box^+ C$}
\LeftLabel{$\mathsf{\rightarrow_R}$}
\RightLabel{ .}
\UIC{$\Sigma;\Gamma_a  \Rightarrow_C A \rightarrow B ,\Lambda,\Box^+ C$}
\DisplayProof 
\]
We denote the premise of $a$ by $b$. 
Since $\mathit{rk}(b)=\lvert \kappa^\prime_a\rvert <\lvert \kappa_a\rvert= \mathit{rk}(b)$, $\Sigma ;\emptyset \vdash_\omega G_b\to \Box (C\wedge H_{i+1})  $. Note that $\mathsf{K}^+\vdash G_a\to G_b$. Hence, $\Sigma ;\emptyset \vdash_\omega G_a\to \Box (C\wedge H_{i+1})  $.

The remaining cases of inference rules ($\mathsf{\rightarrow_L}$) and ($\mathsf{cut}$) are analogous to the case of ($\mathsf{\rightarrow_R}$). Therefore, we omit them.

The claim is established. We have $\Sigma ;\emptyset \vdash_\omega  H_i\to \Box (C\wedge H_{i+1}) $ for any $i\in \mathbb{N}$. Applying the inference rule ($\omega$), we obtain $\Sigma ;\emptyset \vdash_\omega H_0\to \Box^+ C$. Note that $\mathsf{K}^+\vdash G_r\to H_0$, where $r$ is the root of $\kappa$. Thus, $\Sigma ;\emptyset \vdash_\omega G_r\to \Box^+ C$. From (\ref{form1}), we have $\Sigma ;\emptyset \vdash_\omega \bigwedge \Gamma_r\to \bigvee (\Delta_r\cup \{G_r\})$. Consequently, $\Sigma ;\emptyset \vdash_\omega \bigwedge \Gamma_r\to \bigvee (\Delta_r\cup \{\Box^+ C\})$, i.e., $\Sigma ;\emptyset \vdash_\omega \bigwedge \Gamma\to \bigvee \Delta $.
\end{proof}


Now the equivalence of the calculus $\mathsf{S}+ \mathsf{cut}$ and the logic  $\mathsf{K}^+$ with $\omega$-derivations can be easily established.
\begin{propos}\label{seqeqinf}
We have $\Sigma ;\Gamma \vdash_\omega A$, where $\Gamma$ is finite, if and only if $\mathsf{S}+ \mathsf{cut} \vdash\Sigma ; \Gamma \Rightarrow A$.
\end{propos}
\begin{proof}
If $\Sigma ;\Gamma \vdash_\omega A $ and the set $\Gamma$ is finite, then $\mathsf{S}+ \mathsf{cut}\vdash\Sigma ; \Gamma \Rightarrow A$ from Lemma \ref{omegatoinfcut}. Now assume $\pi$ is an $\infty$-proof of the sequent $\Sigma ; \Gamma \Rightarrow A$. Applying Lemma \ref{inftoomega} to the annotated $\infty$-proof $\pi^\circ$, we obtain $\Sigma ;\emptyset \vdash_\omega \bigwedge \Gamma\to A $.
Consequently, $\Sigma ;\Gamma  \vdash_\omega \bigwedge \Gamma\to A $, $\Sigma ;\Gamma  \vdash_\omega \bigwedge \Gamma$ and $\Sigma ;\Gamma \vdash_\omega  A$.

\end{proof}

\section{Fixed-points and continuous families}
\label{s4.1}


In this section, we prove several fixed-point results, which allow us to recursively define operations acting on the set of $\infty$-proofs $\mathsf{P}$ and show properties of such operations.  




For $\vec{\pi}\in \mathsf{P}^m$, we set $\mathit{Ann}(\vec{\pi})\coloneq \mathit{Ann}(\pi_1)\cap \dotsb \cap \mathit{Ann}(\pi_m)$. By $\mathscr{O}_m$, we denote the set of all functions from $\mathsf{P}^m$ to $\mathsf{P}$ that preserve annotations, i.e. $\mathscr{O}_m\coloneq \{\mathsf{a}\colon \mathsf{P}^m \to \mathsf{P}\mid \mathit{Ann}(\vec{\pi})\subset \mathsf{a}(\pi) \text{ for any $\vec{\pi}\in \mathsf{P}^m$}\}$. In what follows, elements of the set $\mathscr{O}_m$ will simply be called \emph{$m$-ary operations on $\mathsf{P}$}. 

We also put 
\[\lVert \vec{\pi}\rVert_s:=\bigoplus\limits^m_{i=1} \lVert \pi_i\rVert_s, \qquad  \lVert \vec{\pi}\rVert:=\min \{\lVert \vec{\pi}\rVert_s \mid s\in \mathit{Ann}(\vec{\pi}) \}, \qquad \lvert \vec{\pi}\rvert:=\sum^m_{i=1}\lvert\pi_i\rvert,  \]
where $\oplus$ is the Hessenberg sum of ordinals \cite{Hes06}. 

Recall that the first uncountable ordinal is standardly denoted by $\omega_1$.
For $\mathsf{a,b}\in \mathscr{O}_m$, $n,k\in \mathbb{N}$ and $\alpha < \omega_1$, we write $\mathsf a\sim_{\alpha, n,k}\mathsf b$ if, for any $\vec{\pi}\in \mathsf{P}^m$, the following conditions hold: 
\begin{itemize}
\item $\mathsf{a}(\vec{\pi}) = \mathsf{b}(\vec{\pi})$ whenever $\lVert \vec{\pi}\rVert < \alpha$;
\item $\mathsf a(\vec{\pi})\sim^s_n\mathsf b(\vec{\pi})$ whenever $s\in \mathit{Ann}(\vec{\pi})$ and $\lVert \vec{\pi}\rVert_s  = \alpha$;
\item $\mathsf a(\vec{\pi})\sim^s_{n+1}\mathsf b(\vec{\pi})$ whenever $s\in \mathit{Ann}(\vec{\pi})$, $\lVert \vec{\pi}\rVert_s  = \alpha$ and $\lvert \vec{\pi}\rvert< k$.
\end{itemize}
We stress that
\begin{gather*}
\mathsf a\sim_{\alpha,n+1,0}\mathsf b  \Longleftrightarrow \forall k\in \mathbb{N}\; (\mathsf a\sim_{\alpha,n,k}\mathsf b), \quad
\mathsf a\sim_{\alpha+1,0,0}\mathsf b  \Longleftrightarrow \forall n\in \mathbb{N}\; (\mathsf a\sim_{\alpha,n,0}\mathsf b).
\end{gather*} 


We call a mapping $\mathscr M$ from $\mathscr{O}_m$ to $\mathscr{O}_m$ \emph{contractive} if, for any $\mathsf a,\mathsf b\in\mathscr{O}_m$, any $n,k\in \mathbb{N}$ and any $\alpha < \omega_1$, we have
\[\mathsf a\sim_{\alpha,n,k}\mathsf b\Longrightarrow \mathscr{M}(\mathsf a)\sim_{\alpha,n,k+1}\mathscr{M}(\mathsf b).\]

Now we state a fixed-point theorem for contractive mappings from $\mathscr{O}_m$ to $\mathscr{O}_m$. Notice that the theorem can be derived from several known results. If we consider the set $\mathscr{O}_m$ as a non-empty spherically complete generalized ultrametric space, then the theorem is a consequence of Prie\ss-Crampe's fixed-point result \cite{PriessCrampe1990}. If we note that $\sim_{\alpha,n,k}$ is a complete well-ordered family of equivalence relations on $\mathscr{O}_m$, then the theorem becomes a consequence of Theorem 1 from \cite{GiMi2003}. However, we give its full proof for the sake of completeness.
\begin{theorem}\label{explicit fixed-point}
Every contractive mapping $\mathscr{M}$ from $\mathscr{O}_m$ to $\mathscr{O}_m$ has a unique fixed-point. 
\end{theorem}
\begin{proof}
Since the set $\mathscr{O}_m$ is non-empty, there is an operation $\mathsf{a}\in \mathscr{O}_m$.
For any ordinal $\lambda < \omega_1$, we define an operation $\mathsf{a}_\lambda \colon \mathsf{P}^m\to\mathsf{P}$ by transfinite recursion.
We put $\mathsf{a}_0= \mathsf{a}$ and $\mathsf{a}_{\lambda+1}= \mathscr{M} (\mathsf{a}_\lambda)$. For a limit ordinal $\lambda$ of the form $\omega^2 \cdot \alpha +\omega \cdot (n+1)$, we set 
\[\mathsf{a}_{\omega^2 \cdot \alpha +\omega \cdot (n+1)}(\vec{\pi}) = \mathsf{a}_{\omega^2 \cdot \alpha +\omega \cdot n+\lvert\vec{\pi}\rvert +1}(\vec{\pi}).\]
In addition, for $\lambda=\omega^2 \cdot (\alpha +1)$, we put
\[\mathsf{a}_{\omega^2 \cdot (\alpha+1) }(\vec{\pi}) =\lim\limits_{n \to +\infty} \mathsf{a}_{\omega^2\cdot \alpha + \omega \cdot n }(\vec{\pi})\]
if $\lVert \vec{\pi}\rVert  \leqslant \alpha$ and $\lim\limits_{n \to +\infty} \mathsf{a}_{\omega^2\cdot \alpha + \omega \cdot n }(\vec{\pi})$ is defined. Recall that, by Proposition \ref{Completeness1}, $\lim\limits_{n \to +\infty} \mathsf{a}_{\omega^2\cdot \alpha + \omega \cdot n }(\vec{\pi})$ is defined whenever there exists $s\in \mathit{Fm}_\circ$ such that
\[\mathsf{a}_{\omega^2\cdot \alpha  }(\vec{\pi})\sim^s_0\mathsf{a}_{\omega^2\cdot \alpha + \omega }(\vec{\pi})\sim^s_1
\mathsf{a}_{\omega^2\cdot \alpha + \omega \cdot 2 }(\vec{\pi}) \sim^s_2\dotso .\]  
In this case, $\mathit{Ann}(\vec{\pi})\subset \mathit{Ann}(\mathsf{a}_{\omega^2\cdot \alpha + \omega  }(\vec{\pi}) )= \mathit{Ann}(\lim\limits_{n \to +\infty} \mathsf{a}_{\omega^2\cdot \alpha + \omega \cdot n }(\vec{\pi}))$. Otherwise, we put 
\[\mathsf{a}_{\omega^2 \cdot (\alpha+1) }(\vec{\pi}) =\mathsf{a}_0 (\vec{\pi}).\]
Finally, for $\lambda =\omega^2 \cdot \alpha$, where $\alpha$ is a limit ordinal, we set
\begin{gather*}
\mathsf{a}_{\omega^2 \cdot \alpha }(\vec{\pi}) = \begin{cases} 
\mathsf{a}_{\omega^2 \cdot ( \lVert \vec{\pi}\rVert +1)}(\vec{\pi}), & \text{if $ \lVert \vec{\pi}\rVert  < \alpha$;}\\
\mathsf{a}_0 (\vec{\pi}), & \text{otherwise.}
\end{cases}
\end{gather*} 
Now the operation $\mathsf{a}_\lambda$ is well-defined for any ordinal $\lambda < \omega_1$.

In this proof, for $\mathsf{c}, \mathsf{d}\in \mathscr{O}_m$, we will write $\mathsf{ c}\simeq_\gamma \mathsf{d}$ if $\mathsf{ c}\sim_{\alpha, n,k} \mathsf{d}$ and $\gamma= \omega^2 \cdot \alpha+\omega \cdot n +k$. Notice that 
\begin{gather*}
\mathsf{c}\simeq_\gamma \mathsf{d} \Longrightarrow \forall \delta< \gamma \; (\mathsf{c}\simeq_\delta \mathsf{d}),  \quad \mathsf{c} =\mathsf{d} \Longleftrightarrow \forall \delta < \omega_1 \; (\mathsf{c} \simeq_{\delta}\mathsf{d}).
\end{gather*} 

We shall check that, for any ordinal $\gamma< \omega_1$ and any ordinal $\beta < \gamma$, we have $\mathsf{a}_\beta \simeq_\beta \mathsf{a}_\gamma$.
We prove this claim by transfinite induction on $\gamma$. The case $\gamma=0$ is trivial.

Case 1: $\gamma$ is a successor ordinal, i.e. $\gamma = \gamma_0 +1$ for some ordinal $\gamma_0< \omega_1$. 
If $\gamma_0 =0$, then we have $\mathsf{a}_0 \simeq_0 \mathsf{a}_1$ and the assertion holds. 

If $\gamma_0 = \gamma^\prime_0+1$ for some ordinal $\gamma^\prime_0< \omega_1$, then $\mathsf{a}_{\gamma^\prime_0} \simeq_{\gamma^\prime_0}\mathsf{a}_{\gamma_0}$ by the induction hypothesis. Since $\mathscr{M}$ is contractive, we have $\mathsf{a}_{\gamma_0} = \mathscr{M}(\mathsf{a}_{\gamma^\prime_0} ) \simeq_{\gamma^\prime_0+1} \mathscr{M}(\mathsf{a}_{\gamma_0} ) = \mathsf{a}_\gamma$, i.e. $\mathsf{a}_{\gamma_0} \simeq_{\gamma_0}\mathsf{a}_\gamma$. By the induction hypothesis, we have $\mathsf{a}_{\beta_0} \simeq_{\beta_0} \mathsf{a}_{\gamma_0}$ for any ordinal $\beta_0 < \gamma_0$. Consequently, for any $\beta_0<\gamma_0$, we obtain $\mathsf{a}_{\beta_0} \simeq_{\beta_0} \mathsf{a}_{\gamma_0} \simeq_{\gamma_0} \mathsf{a}_\gamma $ and $\mathsf{a}_{\beta_0} \simeq_{\beta_0} \mathsf{a}_\gamma $. 
Hence, $\mathsf{a}_\beta \simeq_\beta \mathsf{a}_\gamma$ for any ordinal $\beta
 < \gamma$.

Suppose $\gamma_0$ is a limit ordinal. For any $\beta_0<\gamma_0$, we have 
\[\mathsf{a}_{\gamma_0} \simeq_{\beta_0} \mathsf{a}_{\beta_0} \simeq_{\beta_0} \mathsf{a}_{\beta_0+1} = \mathscr{M}(\mathsf{a}_{\beta_0}) \simeq_{\beta_0 +1} \mathscr{M}(\mathsf{a}_{\gamma_0}) = \mathsf{a}_{\gamma},\]
where $\mathsf{a}_{\beta_0} \simeq_{\beta_0} \mathsf{a}_{\gamma_0}$ and $\mathsf{a}_{\beta_0} \simeq_{\beta_0} \mathsf{a}_{\beta_0+1}$ follow from the induction hypothesis. We obtain $\mathsf{a}_{\gamma_0} \simeq_{\beta_0} \mathsf{a}_{\gamma}$ for any $\beta_0<\gamma_0$. Hence, $\mathsf{a}_{\gamma_0} \simeq_{\gamma_0} \mathsf{a}_{\gamma}$. In addition, for any $\beta_0<\gamma_0$, we have $\mathsf{a}_{\beta_0} \simeq_{\beta_0} \mathsf{a}_{\gamma_0} \simeq_{\gamma_0} \mathsf{a}_\gamma $ and $\mathsf{a}_{\beta_0} \simeq_{\beta_0} \mathsf{a}_\gamma $. Therefore $\mathsf{a}_\beta \simeq_\beta \mathsf{a}_\gamma$ for any ordinal $\beta
 < \gamma$.


Case 2: $\gamma = \omega^2\cdot \alpha_0 +\omega \cdot (n_0+1)$ for some $\alpha_0< \omega_1$ and $n_0 \in \mathbb{N}$. Consider any $\beta< \gamma$. There is a natural number $k_0$ such that $\beta < \omega^2\cdot \alpha_0 +\omega \cdot n_0+k_0$. From the induction hypothesis, we have $\mathsf{a}_\beta \simeq_\beta \mathsf{a}_{\omega^2\cdot \alpha_0 +\omega \cdot n_0+k_0}$. 

Now we claim that $\mathsf{a}_{\omega^2\cdot \alpha_0 +\omega \cdot n_0+k_0} \simeq_{\omega^2\cdot \alpha_0 +\omega \cdot n_0+k_0} \mathsf{a}_\gamma$. 
Consider any tuple $\vec{\pi} \in \mathcal P^m$. From the definition of $\mathsf{a}_\gamma$, we have $\mathsf{a}_{\gamma}(\vec{\pi}) = \mathsf{a}_{\omega^2\cdot \alpha_0 +\omega \cdot n_0+\lvert\vec{\pi}\rvert +1}(\vec{\pi})$.
By the induction hypothesis, we also have 
\[\mathsf{a}_{\omega^2\cdot \alpha_0 +\omega \cdot n_0+\lvert\vec{\pi}\rvert +1}\sim_{\alpha_0, n_0, \min \{\lvert\vec{\pi}\rvert +1, k_0\} } \mathsf{a}_{\omega^2\cdot \alpha_0 +\omega \cdot n_0+k_0}.\]
Consequently, if $\lVert \vec{\pi}\rVert  < \alpha_0$, then 
\[\mathsf{a}_{\gamma}(\vec{\pi}) = \mathsf{a}_{\omega^2\cdot \alpha_0 +\omega \cdot n_0+\lvert\vec{\pi}\rvert +1}(\vec{\pi})= \mathsf{a}_{\omega^2\cdot \alpha_0 +\omega \cdot n_0+k_0} (\vec{\pi}).\] 
If $ \lVert \vec{\pi}\rVert_s =  \alpha_0$ for some $s\in \mathit{Ann}(\vec{\pi})$, then 
\[\mathsf{a}_{\gamma}(\vec{\pi}) = \mathsf{a}_{\omega^2\cdot \alpha_0 +\omega \cdot n_0+\lvert\vec{\pi}\rvert +1}(\vec{\pi})\sim^s_{n_0} \mathsf{a}_{\omega^2\cdot \alpha_0 +\omega \cdot n_0+k_0} (\vec{\pi}).\]
If $s\in \mathit{Ann}(\vec{\pi})$, $\lVert \vec{\pi}\rVert_s = \alpha_0$ and $\lvert\vec{\pi}\rvert< k_0$, then $\lvert\vec{\pi}\rvert < \min \{\lvert\vec{\pi}\rvert +1, k_0\}$ and 
\[\mathsf{a}_{\gamma}(\vec{\pi}) = \mathsf{a}_{\omega^2\cdot \alpha_0 +\omega \cdot n_0+\lvert\vec{\pi}\rvert +1}(\vec{\pi})\sim^s_{n_0+1} \mathsf{a}_{\omega^2\cdot \alpha_0 +\omega \cdot n_0+k_0} (\vec{\pi}).\]
The claim is checked. We see that $\mathsf{a}_{\omega^2\cdot \alpha_0 +\omega \cdot n_0+k_0} \simeq_{\omega^2\cdot \alpha_0 +\omega \cdot n_0+k_0} \mathsf{a}_\gamma$. 

We obtain $\mathsf{a}_\beta \simeq_\beta  \mathsf{a}_\gamma$ since
\[\mathsf{a}_\beta \simeq_\beta \mathsf{a}_{\omega^2\cdot \alpha_0 +\omega \cdot n_0+k_0}  \simeq_{\omega^2\cdot \alpha_0 +\omega \cdot n_0+k_0} \mathsf{a}_\gamma .\]

Case 3: $\gamma = \omega^2\cdot (\alpha_0 + 1)$ for some $\alpha_0< \omega_1$. Consider any $\beta< \gamma$. There is a natural number $n_0$ such that $\beta < \omega^2\cdot \alpha_0 +\omega \cdot n_0$. From the induction hypothesis, we have $\mathsf{a}_\beta \simeq_\beta \mathsf{a}_{\omega^2\cdot \alpha_0 +\omega \cdot n_0}$ and
\begin{gather}\label{series}
\mathsf{a}_{\omega^2\cdot \alpha_0 +\omega \cdot 0} \sim_{\alpha_0,0,0} \mathsf{a}_{\omega^2\cdot \alpha_0 +\omega \cdot 1 }\sim_{\alpha_0,1,0} \mathsf{a}_{\omega^2\cdot \alpha_0 +\omega \cdot 2} \sim_{\alpha_0,2,0} \dotsb .
\end{gather}
 
We claim that $\mathsf{a}_{\omega^2\cdot \alpha_0 +\omega \cdot n_0} \simeq_{\omega^2\cdot \alpha_0 +\omega \cdot n_0} \mathsf{a}_\gamma$. 
Consider any tuple $\vec{\pi} \in \mathsf{P}^m$. 
If $ \lVert \vec{\pi}\rVert  < \alpha_0$, then, from (\ref{series}), we have
\[\mathsf{a}_{\omega^2\cdot \alpha_0 +\omega \cdot 0} (\vec{\pi}) = \mathsf{a}_{\omega^2\cdot \alpha_0 +\omega \cdot 1 }(\vec{\pi}) = \mathsf{a}_{\omega^2\cdot \alpha_0 +\omega \cdot 2} (\vec{\pi}) = \dotsb \]
and $ \lim\limits_{n \to +\infty} \mathsf{a}_{\omega^2\cdot \alpha_0 + \omega \cdot n }(\vec{\pi})$ is defined and equals to $\mathsf{a}_{\omega^2\cdot \alpha_0 + \omega \cdot n_0 }(\vec{\pi})$. 
Consequently, 
\[\mathsf{a}_{\gamma}(\vec{\pi}) = \lim\limits_{n \to +\infty} \mathsf{a}_{\omega^2\cdot \alpha_0 + \omega \cdot n }(\vec{\pi})= \mathsf{a}_{\omega^2\cdot \alpha_0 + \omega \cdot n_0 }(\vec{\pi}).\] 
If $s\in \mathit{Ann}(\vec{\pi})$ and $\lVert\vec{\pi}\rVert_s  = \alpha_0$, then, from (\ref{series}), we see 
\[\mathsf{a}_{\omega^2\cdot \alpha_0 +\omega \cdot 0} (\vec{\pi}) \sim^s_0 \mathsf{a}_{\omega^2\cdot \alpha_0 +\omega \cdot 1 }(\vec{\pi}) \sim^s_1 \mathsf{a}_{\omega^2\cdot \alpha_0 +\omega \cdot 2} (\vec{\pi}) \sim^s_2 \dotso \]
and $ \lim\limits_{n \to +\infty} \mathsf{a}_{\omega^2\cdot \alpha_0 + \omega \cdot n }(\vec{\pi})$ is defined by Proposition \ref{Completeness1}. We obtain that
\[\mathsf{a}_{\gamma}(\vec{\pi}) = \lim\limits_{n \to +\infty} \mathsf{a}_{\omega^2\cdot \alpha_0 + \omega \cdot n }(\vec{\pi})\sim^s_{n_0} \mathsf{a}_{\omega^2\cdot \alpha_0 + \omega \cdot n_0 }(\vec{\pi}).\]
The claim is checked. We see that $\mathsf{a}_{\omega^2\cdot \alpha_0 +\omega \cdot n_0} \simeq_{\omega^2\cdot \alpha_0 +\omega \cdot n_0} \mathsf{a}_\gamma$. 

We obtain $\mathsf{a}_\beta \simeq_\beta  \mathsf{a}_\gamma$ since
\[\mathsf{a}_\beta \simeq_\beta \mathsf{a}_{\omega^2\cdot \alpha_0 +\omega \cdot n_0}  \simeq_{\omega^2\cdot \alpha_0 +\omega \cdot n_0} \mathsf{a}_\gamma .\]


Case 4: $\gamma = \omega^2\cdot \alpha_0 $ for some limit ordinal $\alpha_0< \omega_1$. Consider any $\beta< \gamma$. There is an ordinal $\alpha^\prime_0 < \alpha_0$ such that $\beta < \omega^2\cdot \alpha^\prime_0$. From the induction hypothesis, we have $\mathsf{a}_\beta \simeq_\beta \mathsf{a}_{\omega^2\cdot \alpha^\prime_0}$.

We claim that $\mathsf{a}_{\omega^2\cdot \alpha^\prime_0} \simeq_{\omega^2\cdot \alpha^\prime_0} \mathsf{a}_\gamma$. 
Consider any tuple $\vec{\pi} \in \mathsf{P}^m$. 
If $ \lVert \vec{\pi}\rVert  < \alpha^\prime_0$, then $\lVert \vec{\pi}\rVert +1  \leqslant \alpha^\prime_0$ and 
\[\mathsf{a}_{\omega^2 \cdot (\lVert \vec{\pi}\rVert +1)} \simeq_{\omega^2 \cdot (\lVert \vec{\pi}\rVert +1)} \mathsf{a}_{\omega^2\cdot \alpha^\prime_0} \]
from the induction hypothesis.
We have $\lVert \vec{\pi}\rVert < \alpha_0$ and $\lVert \vec{\pi}\rVert  < \lVert \vec{\pi}\rVert +1$. Therefore, 
\[\mathsf{a}_{\gamma}(\vec{\pi}) = \mathsf{a}_{\omega^2 \cdot (\lVert \vec{\pi}\rVert +1)}(\vec{\pi})= \mathsf{a}_{\omega^2\cdot \alpha^\prime_0} (\vec{\pi}).\]
The claim is checked. We see that $\mathsf{a}_{\omega^2\cdot \alpha^\prime_0} \simeq_{\omega^2\cdot \alpha^\prime_0} \mathsf{a}_\gamma$. 

We obtain $\mathsf{a}_\beta \simeq_\beta  \mathsf{a}_\gamma$ since
\[\mathsf{a}_\beta \simeq_\beta \mathsf{a}_{\omega^2\cdot \alpha^\prime_0}  \simeq_{\omega^2\cdot \alpha^\prime_0} \mathsf{a}_\gamma .\]

We have proved that, for any ordinal $\gamma< \omega_1$ and any ordinal $\beta < \gamma$, the assertion $\mathsf{a}_\beta \simeq_\beta \mathsf{a}_\gamma$ holds. 

Now we define the required fixed-point $\mathsf{a}_{\omega_1}$ by putting 
\[\mathsf{a}_{\omega_1} (\vec{\pi}) = \mathsf{a}_{\omega^2 \cdot (\lVert \vec{\pi}\rVert +1)}(\vec{\pi}).\] 
Let us check that $\mathsf a_{\omega_1} \simeq_{\omega^2\cdot \alpha} \mathsf{a}_{\omega^2 \cdot \alpha} $ for any $\alpha< \omega_1$. Consider any tuple $\vec{\pi}\in \mathsf{P}^m$ such that $\lVert \vec{\pi} \rVert <\alpha$. We have $\mathsf{a}_{\omega_1} (\vec{\pi}) = \mathsf{a}_{\omega^2 \cdot (\lVert \vec{\pi}\rVert +1)}(\vec{\pi}) = \mathsf{a}_{\omega^2 \cdot \alpha} (\vec{\pi})$ since $\mathsf a_{\omega^2 \cdot (\lVert \vec{\pi}\rVert +1)} \simeq_{\omega^2 \cdot (\lVert \vec{\pi}\rVert +1)} \mathsf{a}_{\omega^2 \cdot \alpha} $. Therefore $\mathsf a_{\omega_1} \simeq_{\omega^2\cdot \alpha} \mathsf{a}_{\omega^2 \cdot \alpha} $ for any $\alpha< \omega_1$. 

We claim that $ \mathsf{a}_{\omega_1} =\mathscr{M}(\mathsf a_{\omega_1})$. It is sufficient to show that $\mathsf a_{\omega_1} \simeq_{\omega^2 \cdot \alpha} \mathscr{M}(\mathsf a_{\omega_1})$ for any $\alpha < \omega_1$. We see 
\[\mathsf a_{\omega_1} \simeq_{\omega^2\cdot \alpha} \mathsf{a}_{\omega^2 \cdot \alpha} \simeq_{\omega^2\cdot \alpha} \mathsf{a}_{\omega^2\cdot \alpha+1 }= \mathscr{M}(\mathsf a_{\omega^2\cdot \alpha}) \simeq_{\omega^2\cdot \alpha+1} \mathscr{M}(\mathsf a_{\omega_1}).\] Consequently, $\mathsf a_{\omega_1} \simeq_{\omega^2\cdot \alpha} \mathscr{M}(\mathsf a_{\omega_1})$ for any $\alpha< \omega_1$. Hence, $ \mathsf{a}_{\omega_1} =\mathscr{M}(\mathsf a_{\omega_1})$. Obviously, the fixed-point is unique.



\end{proof}

A family of subsets $\mathscr{S}_{\alpha,n,k}$ of the set $\mathscr{O}_m$, where $\alpha<\omega_1$ and $n,k\in \mathbb{N}$, is called \emph{continuous} if it satisfies the following conditions:
\begin{gather*}
\mathscr{S}_{0,0,0} = \mathscr{O}_m, \qquad \mathscr{S}_{\alpha,n+1,0} =\bigcap\limits_{k\in \mathbb{N}} \mathscr{S}_{\alpha,n,k}, \qquad
\mathscr{S}_{\alpha+1,0,0} = \bigcap\limits_{n\in \mathbb{N}} \mathscr{S}_{\alpha,n,0}.
\end{gather*}
Besides, if $\alpha$ is a countable limit ordinal, then
\[\mathscr{S}_{\alpha,0,0} = \bigcap\limits_{\alpha^\prime<\alpha} \mathscr{S}_{\alpha^\prime,0,0}.\] 
\begin{theorem}\label{co-induction principle}
Suppose we have a continuous family of subsets $\mathscr{S}_{\alpha,n,k}$ of the set $\mathscr{O}_m$. In addition, suppose $\mathscr{M}$ is a mapping from $\mathscr{O}_m$ to $\mathscr{O}_m$ such that $\mathscr{M}(\mathscr{S}_{\alpha, n, k}) \subset \mathscr{S}_{\alpha, n, k+1}$. Then any fixed-point of $\mathscr{M}$ belongs to $\bigcap\limits_{\alpha < \omega_1} \mathscr{S}_{\alpha, 0, 0}$.
\end{theorem}
\begin{proof}
Assume $\mathsf b $ is a fixed-point of $\mathscr{M}$. We shall check that $\mathsf b \in \bigcap\limits_{\alpha < \omega_1} \mathscr{S}_{\alpha, 0, 0}$. It is sufficient to prove that $\mathsf b \in \mathscr{S}_{\alpha,n,k}$ for any $\alpha<\omega_1$ and any $n,k \in \mathbb{N}$. The proof is by transfinite induction on $\gamma  = \omega^2 \cdot \alpha+\omega \cdot n +k$. The case $\gamma=0$ is trivial.

Case 1: $\gamma$ is a successor ordinal, i.e. $\gamma = \omega^2 \cdot \alpha+\omega \cdot n +k+1$ for some $\alpha<\omega_1$ and some $n,k \in \mathbb{N}$. By the induction hypothesis for $\gamma_0=\omega^2 \cdot \alpha+\omega \cdot n +k$, we have $\mathsf{ b} \in \mathscr{S}_{\alpha,n,k}$. Hence,
\[\mathsf{b}=\mathscr{M}(\mathsf{b})\in \mathscr{M}(\mathscr{S}_{\alpha,n,k})\subset \mathscr{S}_{\alpha,n,k+1}.\] 

Case 2: $\gamma = \omega^2\cdot \alpha +\omega \cdot (n+1)$ for some $\alpha< \omega_1$ and $n \in \mathbb{N}$. By the induction hypothesis, we have $\mathsf{ b} \in \mathscr{S}_{\alpha,n,k}$ for each $k \in \mathbb{N}$. Hence,
\[\mathsf{b}\in  \bigcap_{k \in \mathbb{N}} \mathscr{S}_{\alpha,n,k}=\mathscr{S}_{\alpha,n+1,0}.\]

Case 3: $\gamma = \omega^2\cdot (\alpha + 1)$ for some $\alpha< \omega_1$. By the induction hypothesis, we have $\mathsf{ b} \in \mathscr{S}_{\alpha,n,0}$ for all $n \in \mathbb{N}$. Consequently,
\[\mathsf{b}\in  \bigcap_{n \in \mathbb{N}} \mathscr{S}_{\alpha,n,0}=\mathscr{S}_{\alpha+1,0,0}.\]

Case 4: $\gamma = \omega^2\cdot \alpha$ for some limit ordinal $\alpha< \omega_1$. By the induction hypothesis, we have $\mathsf{ b} \in \mathscr{S}_{\alpha^\prime,0,0}$ for $\alpha^\prime <\alpha$. Therefore
\[\mathsf{b}\in  \bigcap_{\alpha^\prime < \alpha} \mathscr{S}_{\alpha^\prime,0,0}=\mathscr{S}_{\alpha+1,0,0}.\]

We conclude that $\mathsf b \in \mathscr{S}_{\alpha,n,k}$ for all $\alpha<\omega_1$ and all $n,k \in \mathbb{N}$.
\end{proof}


Let us consider several continuous families of operations on $\mathsf{P}$, which will play an important role in what follows. 

For $n>0$, we define the set of $\infty$-proofs $\mathsf{CF}_n(s)$ ($\mathsf{SM}_n(s)$) by setting $\pi \in \mathsf{CF}_n(s)$ ($\pi \in \mathsf{SM}_n(s)$) if and only if $s\in\mathit{Ann}(\pi)$ and, in the $n$-fragment of $\pi^s$, there are no applications of the inference rule ($\mathsf{cut}$) ($s\in\mathit{Ann}(\pi)$ and all applications of modal rules whose conclusions belong to the $n$-fragment of $\pi^s$ are slim). Also, $\mathsf{CF}_0 (s)\coloneq \mathsf{P}(s)$ ($\mathsf{SM}_0 (s)\coloneq \mathsf{P}(s)$).

A unary operation $\mathsf{a} $ is \emph{cut-eliminating} (\emph{slimming}) if the $\infty$-proof $\mathsf{a}(\pi)$ is cut-free (slim) for each $\pi \in \mathsf{P}$. For $\alpha<\omega_1$ and $n,k \in \mathbb{N}$, a unary operation $\mathsf{a} $ is called  \emph{$(\alpha, n, k)$-cut-eliminating} (\emph{$(\alpha, n, k)$-slimming}) if, for any $\pi \in \mathsf{P}$, 
\begin{itemize}
\item $\mathsf{a}(\pi) $ is cut-free (slim) whenever $\lVert \pi\rVert  < \alpha$,
\item $\mathsf a(\pi)\in \mathsf{CF}_n(s)$ ($\mathsf a(\pi)\in \mathsf{SM}_n(s)$) whenever $s\in\mathit{Ann}(\pi)$ and $ \lVert \pi\rVert_s  = \alpha$,
\item $\mathsf a(\pi)\in \mathsf{CF}_{n+1}(s)$ ($\mathsf a(\pi)\in \mathsf{SM}_{n+1}(s)$) if $s\in\mathit{Ann}(\pi)$, $ \lVert \pi\rVert_s  = \alpha$ and $\lvert\pi\rvert< k$.
\end{itemize}
We denote the set of $(\alpha, n, k)$-cut-eliminating ($(\alpha, n, k)$-slimming) operations by $\mathscr{CE}_{\alpha,n,k}$ ($\mathscr{SM}_{\alpha,n,k}$). Clearly, both families are continuous. Besides,
\[\mathscr{CE} = \bigcap\limits_{\alpha < \omega_1} \mathscr{CE}_{\alpha, 0, 0} \quad \text{and} \quad \mathscr{SM} = \bigcap\limits_{\alpha < \omega_1} \mathscr{SM}_{\alpha, 0, 0},\]
where $\mathscr{CE}$ and $\mathscr{SM}$ are the sets of cut-eliminating and slimming operations respectively.

An operation $\mathsf{a}\colon \mathsf{P} \to\mathsf{P}$ is called \emph{root-preserving} if it maps an $\infty$-proof to an $\infty$-proof of the same sequent. Let $\mathscr{RP}$ denote the set of all root-preserving operations on $\mathsf{P}$. Note that the set $\mathscr{RP}$ is non-empty since the identity function belongs to $\mathscr{RP}$. \emph{Contractive mappings from $\mathscr{RP}$ to $\mathscr{RP}$} are defined analogously to the case of unary operations $\mathscr{O}_1$.

Cut elimination for the sequent calculus $\mathsf{S}+ \mathsf{cut}$ will be established by means of the following theorem.
\begin{theorem}\label{explicit fixed-point2}
Every contractive mapping $\mathscr M$ from $\mathscr{RP}$ to $\mathscr{RP}$ has a unique fixed-point. Moreover, if 
\[\mathscr{M}(\mathscr{CE}_{\alpha, n, k}\cap\mathscr{RP}) \subset \mathscr{CE}_{\alpha, n, k+1} \quad (\mathscr{M}(\mathscr{SM}_{\alpha, n, k}\cap\mathscr{RP}) \subset \mathscr{SM}_{\alpha, n, k+1}),\] 
then the fixed-point of $\mathscr{M}$ is cut-eliminating (slimming).
\end{theorem}
\begin{proof}
This result is obtained in the same way as Theorem \ref{explicit fixed-point} and Theorem \ref{co-induction principle}, where $\mathscr{O}_m$ is replaced with $\mathscr{RP}$ and $\mathsf{a}_0$ is defined as the identity function. 
\end{proof}



In the next section, as an intermediate step in the proof of cut-elimination, we define operations that eliminate single applications of the rule ($\mathsf{cut}$), not all of them at once. In order to prove their properties accurately, we introduce three more continuous families of operations on $\mathsf{P}$. 
 
An operation $\mathsf a$ from $ \mathscr{O}_2$ is \emph{non-expansive} if
\[\vec{\pi}\sim^s_i \vec{\pi}^\prime \Longrightarrow \mathsf{a}(\vec{\pi})\sim^s_i \mathsf{a}(\vec{\pi}^\prime),\]
where $\vec{\pi}\sim^s_i \vec{\pi}^\prime$ is an abbreviation for $\pi_1 \sim^s_i \pi^\prime_1$, $\pi_2 \sim^s_i \pi^\prime_2$.
We call a binary operation $\mathsf{a} $ \emph{$(\alpha, n, k)$-non-expansive} if, for any $\vec{\pi}, \vec{\pi}^\prime \in \mathsf{P}^2$ such that $\vec{\pi}\sim^s_i \vec{\pi}^\prime $, the following conditions hold: 
\begin{itemize}
\item $\mathsf{a}(\vec{\pi}) \sim^s_i \mathsf{a}(\vec{\pi}^\prime)$ whenever $\max \{\lVert \vec{\pi}\rVert_s, \lVert \vec{\pi}^\prime\rVert_s\} < \alpha$,
\item $\mathsf{a}(\vec{\pi}) \sim^s_i \mathsf{a}(\vec{\pi}^\prime)$ whenever $\max \{\lVert \vec{\pi}\rVert_s, \lVert \vec{\pi}^\prime\rVert_s\} = \alpha$ and $i \leqslant n$,
\item $\mathsf{a}(\vec{\pi}) \sim^s_i \mathsf{a}(\vec{\pi}^\prime)$ if $\max \{\lVert \vec{\pi} \rVert_s,\lVert \vec{\pi}^\prime\rVert_s \}=\alpha$, $i = n+1$ and $\max \{\lvert \vec{\pi} \rvert,\lvert \vec{\pi}^\prime\rvert \}<k$.
\end{itemize}
We denote the set of non-expansive operations by $\mathscr{NE}$ and the set of $(\alpha, n, k)$-non-expansive operations by $\mathscr{NE}_{\alpha,n,k}$. Trivially, the family $\mathscr{NE}_{\alpha,n,k}$ is continuous and
\[\mathscr{NE} = \bigcap\limits_{\alpha < \omega_1} \mathscr{NE}_{\alpha, 0, 0}.\]

An operation $\mathsf a$ from $ \mathscr{O}_2$ is \emph{non-pollutive} (\emph{non-fattening}) if
\[\vec{\pi}\in \mathsf{CF}_i(s)\times \mathsf{CF}_i(s) \Rightarrow \mathsf{a}(\vec{\pi})\in \mathsf{CF}_i(s) \quad (\vec{\pi}\in \mathsf{SM}_i(s)\times \mathsf{SM}_i(s) \Rightarrow \mathsf{a}(\vec{\pi})\in \mathsf{SM}_i(s)).\]
For $\alpha<\omega_1$ and $n,k \in \mathbb{N}$, a binary operation $\mathsf{a} $ is \emph{$(\alpha, n, k)$-non-pollutive} (\emph{$(\alpha, n, k)$-non-fattening}) if, for each $s\in\mathit{Fm}_\circ$, each $i\in \mathbb{N}$ and each $\vec{\pi}\in \mathsf{CF}_i(s)\times \mathsf{CF}_i(s)$ ($\vec{\pi}\in \mathsf{SM}_i(s)\times \mathsf{SM}_i(s)$), 
\begin{itemize}
\item $\mathsf{a}(\vec{\pi}) \in \mathsf{CF}_i(s)$ ($\mathsf{a}(\vec{\pi}) \in \mathsf{SM}_i(s)$) whenever $\lVert \vec{\pi}\rVert_s < \alpha$,
\item $\mathsf{a}(\vec{\pi}) \in \mathsf{CF}_i(s)$ ($\mathsf{a}(\vec{\pi}) \in \mathsf{SM}_i(s)$) whenever $\lVert \vec{\pi}\rVert_s = \alpha$ and $i \leqslant n$,
\item $\mathsf{a}(\vec{\pi}) \in \mathsf{CF}_i(s)$ ($\mathsf{a}(\vec{\pi}) \in \mathsf{SM}_i(s)$) whenever $\lVert \vec{\pi}\rVert_s = \alpha$, $i = n+1$ and $\lvert \vec{\pi} \rvert<k$.
\end{itemize}
The set of $(\alpha, n, k)$-non-pollutive ($(\alpha, n, k)$-non-fattening) operations is denoted by $\mathscr{NP}_{\alpha,n,k}$ ($\mathscr{NF}_{\alpha,n,k}$). Both families are continuous and
\[\mathscr{NP} = \bigcap\limits_{\alpha < \omega_1} \mathscr{NP}_{\alpha, 0, 0} \quad \text{and} \quad \mathscr{NF} = \bigcap\limits_{\alpha < \omega_1} \mathscr{NF}_{\alpha, 0, 0},\]
where $\mathscr{NP}$ and $\mathscr{NF}$ are the sets of non-pollutive and non-fattening operations respectively.

A pair $(\pi,\tau)$ is called \emph{cut pair} if $\pi$ is an $\infty$-proof of $\Sigma;\Gamma\Rightarrow A, \Delta$ and $\tau$ is an $\infty$-proof of $\Sigma; \Gamma, A\Rightarrow \Delta$. For a cut pair $(\pi,\tau)$, the sequent $\Sigma;\Gamma\Rightarrow \Delta$ is its \emph{cut result} and the formula $A$ is its \emph{cut formula}.
A binary operation $\mathsf{a}\colon \mathsf{P}\times \mathsf{P} \to \mathsf{P}$ is \emph{$A$-removing} if it maps every cut pair $(\pi,\tau)$ with the cut formula $A$ to an $\infty$-proof of its cut result. We denote the set of $A$-removing operations on $\mathsf{P}$ by $\mathscr{RE}_A$ and define a notion of \emph{contractive mapping from $\mathscr{RE}_A$ to $\mathscr{RE}_A$} analogously to the case of binary operations $\mathscr{O}_2$.

In the next section, we will prove properties of operations eliminating single applications of the rule ($\mathsf{cut}$) by means of the following theorem.

 
\begin{theorem}\label{explicit fixed-point3}
Every contractive mapping $\mathscr M$ from $\mathscr{RE}_A$ to $\mathscr{RE}_A$ has a unique fixed-point. In addition, if $\mathscr{M}(\mathscr{NE}_{\alpha, n, k}\cap\mathscr{RE}_A) \subset \mathscr{NE}_{\alpha, n, k+1}$ ($\mathscr{M}(\mathscr{NP}_{\alpha, n, k} \cap\mathscr{RE}_A) \subset \mathscr{NP}_{\alpha, n, k+1}$, $\mathscr{M}(\mathscr{NF}_{\alpha, n, k} \cap\mathscr{RE}_A) \subset \mathscr{NF}_{\alpha, n, k+1}$), then the fixed-point of $\mathscr{M}$ is non-expansive (non-pollutive, non-fattening). 
\end{theorem}
\begin{proof}
The theorem is obtained in the same way as Theorem \ref{explicit fixed-point} and Theorem \ref{co-induction principle}, where $\mathscr{O}_m$ is replaced with $\mathscr{RE}_A$ and the operation $\mathsf{a}_0\colon \mathsf{P}\times\mathsf{P}\to \mathsf{P}$ is defined as follows: for a cut pair $(\pi,\tau)$ with the cut formula $A$, this operation joins $\infty$-proofs $\pi$ and $\tau$ with an appropriate instance of the rule $(\mathsf{cut})$; for all other pairs, the operation $\mathsf{a}_0$ returns the first argument. 
\end{proof}

\section{Elimination of single cuts}
\label{s4.2}

In the given section, we define an operation that eliminates one root application of the rule ($\mathsf{cut}$).  

\begin{propos}\label{reabadeq}
For any formula $A$, there exists a non-expansive, non-pollutive and non-fattening operation $\mathsf{re}_A\in \mathscr{RE}_A$.
\end{propos}
In order to prove this proposition, we first establish admissibility of several standard auxiliary rules. A unary operation $\mathsf{a}:\mathsf{P}\to \mathsf{P}$ is \emph{non-expansive} if
\[\pi\sim^s_i \pi^\prime \Longrightarrow \mathsf{a}(\pi)\sim^s_i \mathsf{a}(\pi^\prime) .\]
An operation $\mathsf{a}:\mathsf{P}\to \mathsf{P}$ is \emph{non-pollutive} (\emph{non-fattening}) if
\[ \pi\in\mathsf{CF}_n (s) \Longrightarrow \mathsf{a}(\pi)\in\mathsf{CF}_n (s) \quad \left(   \pi\in\mathsf{SM}_n (s) \Longrightarrow \mathsf{a}(\pi)\in\mathsf{SM}_n (s) \right) .\]
A single-premise inference rule is called \emph{strongly admissible in the system $\mathsf{S} +\mathsf{cut}$} if there exists a non-expansive, non-pollutive and non-fattening operation $\mathsf{a}\colon\mathsf{P} \to \mathsf{P}$ that maps any $\infty$-proof of the premise to an $\infty$-proof of the conclusion. The operation $\mathsf{a}$ must also satisfy the conditions: $\lvert \mathsf{a}(\pi)\rvert \leqslant \lvert \pi \rvert$ and $\lVert \mathsf{a}(\pi)\rVert_s \leqslant \lVert \pi\rVert_s$ for any $\pi \in \mathsf{P}$ and any $s\in\mathit{Ann}(\pi)$.

In the following lemmas, the required non-expansive, non-pollutive and non-fattening operations can be defined in the usual way by induction on the local heights of $\infty$-proofs for premises. Therefore, we omit the proofs.
 
\begin{lemma}\label{strongweakening}
For any finite multisets of formulas $\Phi$ and $\Psi$, the inference rule
\begin{gather*}
\AXC{$\Sigma ; \Gamma\Rightarrow\Delta$}
\LeftLabel{$\mathsf{wk}_{\Phi; \Psi}$}
\UIC{$\Sigma ;\Phi,\Gamma\Rightarrow\Delta,\Psi$}
\DisplayProof
\end{gather*}
is strongly admissible in $\mathsf{S} +\mathsf{cut}$. 
\end{lemma}

\begin{lemma}\label{inversion}
For any formulas $A$ and $B$, the following inference rules are strongly admissible in $\mathsf{S} +\mathsf{cut}$:
\begin{gather*}
\AXC{$\Sigma ;\Gamma , A \rightarrow B \Rightarrow  \Delta$}
\LeftLabel{$\mathsf{li}_{A \to B}$}
\RightLabel{ ,}
\UIC{$\Sigma ;\Gamma ,B \Rightarrow  \Delta$}
\DisplayProof\qquad
\AXC{$\Sigma ;\Gamma , A \rightarrow B \Rightarrow  \Delta$}
\LeftLabel{$\mathsf{ri}_{A \to B}$}
\RightLabel{ ,}\UIC{$\Sigma ;\Gamma  \Rightarrow  A, \Delta$}
\DisplayProof\\\\
\AXC{$\Sigma ;\Gamma  \Rightarrow   A \rightarrow B, \Delta$}
\LeftLabel{$\mathsf{i}_{A \to B}$}
\RightLabel{ ,}
\UIC{$\Sigma ;\Gamma ,A \Rightarrow B, \Delta$}
\DisplayProof\qquad
\AXC{$\Sigma ;\Gamma  \Rightarrow   \bot, \Delta$}
\LeftLabel{$\mathsf{i}_{\bot}$}
\RightLabel{ .}
\UIC{$\Sigma ;\Gamma  \Rightarrow \Delta$}
\DisplayProof
\end{gather*}
\end{lemma}

\begin{lemma}\label{weakcontraction}
For any propositional variable $p$, the inference rules
\begin{gather*}
\AXC{$\Sigma ;\Gamma , p,p \Rightarrow  \Delta$}
\LeftLabel{$\mathsf{acl}_{p}$}
\RightLabel{ ,}
\UIC{$\Sigma ;\Gamma ,p \Rightarrow  \Delta$}
\DisplayProof\qquad
\AXC{$\Sigma ;\Gamma \Rightarrow p,p, \Delta$}
\LeftLabel{$\mathsf{acr}_{p}$}
\UIC{$\Sigma ;\Gamma \Rightarrow p, \Delta$}
\DisplayProof
\end{gather*}
are strongly admissible in $\mathsf{S} +\mathsf{cut}$.
\end{lemma}

Now three lemmas follow, whose proofs have a common structure, but the number of cases in each subsequent proof will increase.\begin{lemma}\label{repadeq}
For any propositional variable $p$, there exists a non-expansive, non-pollutive and non-fattening operation $\mathsf{re}_p\in \mathscr{RE}_p$.
\end{lemma}
\begin{proof}

We will define the required operation $\mathsf{re}_p$ as a fixed-point of a contractive mapping $\mathscr{M}$ from $\mathscr{RE}_p$ to $\mathscr{RE}_p$.

For an operation $\mathsf a\in \mathscr{RE}_p$, we define $\mathscr M (\mathsf a)\in \mathscr{RE}_p$ following standard reductions for an application of the rule ($\mathsf{cut}$) with an atomic cut formula. Consider arbitrary $\infty$-proofs $\pi$ and $\tau$. If the pair $(\pi,\tau)$ is not a cut pair or is a cut pair with the cut formula different from $p$, then we put $\mathscr{M}(\mathsf{a})(\pi,\tau):=\pi$. 
Otherwise, we define $\mathscr{M}(\mathsf{a})(\pi,\tau)$ according to the following cases, where $\Sigma;\Gamma\Rightarrow \Delta$ is the cut result of $(\pi,\tau)$.
 
Case 1. Suppose $ \pi$ consists only of one initial sequent $\Sigma;\Gamma\Rightarrow \Delta , p$. If $\Sigma;\Gamma\Rightarrow \Delta$ is also an initial sequent, then $\mathscr{M}(\mathsf{a})(\pi,\tau)$ is defined as the $\infty$-proof consisting of the sequent $\Sigma;\Gamma\Rightarrow \Delta$. If $\Sigma;\Gamma\Rightarrow \Delta$ is not an initial sequent, then $\Gamma$ has the form $p,\Gamma_0$ and the root of $\tau$ is marked by $\Sigma;p,p,\Gamma_0 \Rightarrow \Delta$. Applying the operation $\mathsf{acl}_p$ from Lemma \ref{weakcontraction}, we put $\mathscr{M}(\mathsf{a})(\pi,\tau) := \mathsf{acl}_p (\tau)$.   

Case 2. If $ \pi$ consists of more than one sequent, then $\mathscr M (\mathsf a)(\pi,\tau)$ is defined as shown in Fig. 3, where $Q=p$, $\mathsf{R}=\mathsf a$ and $ \mathsf{wk}_{\Phi; \Psi}$, $\mathsf{li}_{A\to C}$, $\mathsf{ri}_{A\to C}$, $\mathsf{i}_{A\to C}$ are operations from Lemma \ref{strongweakening} and Lemma \ref{inversion}.

\begin{center}
\begin{sidewaystable}
\begin{adjustwidth}{0cm}{-2.7cm}
\begin{minipage}{\textwidth}
\medskip
\begin{align*}
\left(
\AXC{$\pi_0$}
\noLine
\UIC{$\vdots$}
\noLine
\UIC{$\Sigma; \Sigma_0,\Lambda, \Pi, \Box^+ \Pi \Rightarrow  A$}
\LeftLabel{$\Box$}
\UIC{$\Sigma;\Phi ,\Box \Lambda, \Box^+ \Pi \Rightarrow \Box A,  \Psi, Q$}
\DisplayProof, \; \tau
\right)
&\longmapsto
\AXC{$\pi_0$}
\noLine
\UIC{$\vdots$}
\noLine
\UIC{$\Sigma; \Sigma_0,\Lambda, \Pi, \Box^+ \Pi \Rightarrow  A$}
\LeftLabel{$\Box$}
\RightLabel{ ,}
\UIC{$\Sigma;\Phi ,\Box \Lambda, \Box^+ \Pi \Rightarrow \Box A,  \Psi$}
\DisplayProof \\\\
\left(
\AXC{$\pi_0$}
\noLine
\UIC{$\vdots$}
\noLine
\UIC{$\Sigma;\Sigma_0, \Lambda, \Pi, \Box^+ \Pi \Rightarrow  A$}
\AXC{$\pi_1$}
\noLine
\UIC{$\vdots$}
\noLine
\UIC{$\Sigma; \Sigma_0,\Lambda, \Pi, \Box^+ \Pi \Rightarrow \Box^+ A$}
\LeftLabel{$\Box^+$}
\BIC{$\Sigma;\Phi ,\Box \Lambda, \Box^+ \Pi \Rightarrow \Box^+ A,  \Psi, Q$}
\DisplayProof, \; \tau
\right)
&\longmapsto
\AXC{$\pi_0$}
\noLine
\UIC{$\vdots$}
\noLine
\UIC{$\Sigma; \Sigma_0,\Lambda, \Pi, \Box^+ \Pi \Rightarrow  A$}
\AXC{$\pi_1$}
\noLine
\UIC{$\vdots$}
\noLine
\UIC{$\Sigma; \Sigma_0,\Lambda, \Pi, \Box^+ \Pi \Rightarrow \Box^+ A$}
\LeftLabel{$\Box^+$}
\RightLabel{ ,}
\BIC{$\Sigma;\Phi ,\Box \Lambda, \Box^+ \Pi \Rightarrow \Box^+ A,  \Psi$}
\DisplayProof \\\\
\left(
\AXC{$\pi_0$}
\noLine
\UIC{$\vdots$}
\noLine
\UIC{$\Sigma;\Gamma,A\Rightarrow C,\Lambda,Q$}
\LeftLabel{$\mathsf{\to_R}$}
\UIC{$\Sigma;\Gamma\Rightarrow A\to C,\Lambda,Q$}
\DisplayProof
, \; \tau
\right)
&\longmapsto
\AXC{$\mathsf{R}(\pi_0, \mathsf{i}_{A \to C}(\tau))$}
\noLine
\UIC{$\vdots$}
\noLine
\UIC{$\Sigma;\Gamma,A\Rightarrow C,\Lambda$}
\LeftLabel{$\mathsf{\to_R}$}
\RightLabel{ ,}
\UIC{$\Sigma;\Gamma\Rightarrow A\to C,\Lambda$}
\DisplayProof \\\\
\left(
\AXC{$\pi_0$}
\noLine
\UIC{$\vdots$}
\noLine
\UIC{$\Sigma;\Lambda, C\Rightarrow \Delta, Q$}
\AXC{$\pi_1$}
\noLine
\UIC{$\vdots$}
\noLine
\UIC{$\Sigma;\Lambda \Rightarrow A,\Delta, Q$}
\LeftLabel{$\mathsf{\to_L}$}
\BIC{$\Sigma;\Lambda, A\to C\Rightarrow \Delta, Q$}
\DisplayProof
,\;\tau
\right)
&\longmapsto
\AXC{$\mathsf{R} (\pi_0, \mathsf{li}_{A\to C} (\tau))$}
\noLine
\UIC{$\vdots$}
\noLine
\UIC{$\Sigma;\Lambda, C\Rightarrow \Delta$}
\AXC{$\mathsf{R} (\pi_1, \mathsf{ri}_{A\to C} (\tau))$}
\noLine
\UIC{$\vdots$}
\noLine
\UIC{$\Sigma;\Lambda \Rightarrow A,\Delta$}
\LeftLabel{$\mathsf{\to_L}$}
\RightLabel{ ,}
\BIC{$\Sigma;\Lambda, A\to C\Rightarrow \Delta$}
\DisplayProof \\\\
\left(
\AXC{$\pi_0$}
\noLine
\UIC{$\vdots$}
\noLine
\UIC{$\Sigma;\Gamma\Rightarrow \Delta, A, Q$}
\AXC{$\pi_1$}
\noLine
\UIC{$\vdots$}
\noLine
\UIC{$\Sigma;A, \Gamma \Rightarrow \Delta, Q$}
\LeftLabel{$\mathsf{cut}$}
\BIC{$\Sigma;\Gamma\Rightarrow \Delta, Q$}
\DisplayProof
,\; \tau
\right)
&\longmapsto
\AXC{$\mathsf{R} (\pi_0, \mathsf{wk}_{\emptyset;A} (\tau))$}
\noLine
\UIC{$\vdots$}
\noLine
\UIC{$\Sigma;\Gamma\Rightarrow \Delta, A$}
\AXC{$\mathsf{R} (\pi_1, \mathsf{wk}_{A;\emptyset} (\tau))$}
\noLine
\UIC{$\vdots$}
\noLine
\UIC{$\Sigma;A, \Gamma\Rightarrow \Delta$}
\LeftLabel{$\mathsf{cut}$}
\RightLabel{ .}
\BIC{$\Sigma;\Gamma\Rightarrow \Delta$}
\DisplayProof 
\end{align*}
\medskip 
\end{minipage}
\\\\
\center{\textbf{Fig. 3}}
\end{adjustwidth}
\end{sidewaystable}
\end{center}

The mapping $\mathscr{M}$ is well defined. We claim that $\mathscr M\colon \mathscr{RE}_{p} \to \mathscr{RE}_{p}$ is contractive. Assume we have two $p$-removing operations $\mathsf a$ and $\mathsf b$ such that $\mathsf a\sim_{\alpha, n,k}\mathsf b$. We shall check that $\mathscr M (\mathsf a)\sim_{\alpha, n,k+1}\mathscr M(\mathsf b)$. In other words, for an arbitrary pair of $\infty$-proofs $(\pi,\tau)$ and any $s\in\mathit{Ann}(\pi)\cap \mathit{Ann}(\tau)$, we shall prove that 
\begin{itemize}
\item $\mathscr M(\mathsf a)(\pi, \tau) = \mathscr M(\mathsf b)(\pi, \tau)$ whenever $\lVert (\pi,\tau)\rVert  < \alpha$,
\item $\mathscr M(\mathsf a)(\pi, \tau)\sim^s_n\mathscr M(\mathsf b)(\pi, \tau)$ whenever $\lVert \pi\rVert_s \oplus \lVert \tau\rVert_s = \alpha$,
\item $\mathscr M(\mathsf a)(\pi, \tau)\sim^s_{n+1}\mathscr M(\mathsf b)(\pi, \tau)$ whenever $\lVert \pi\rVert_s \oplus \lVert \tau\rVert_s = \alpha$ and $\lvert\pi\rvert + \rvert \tau \rvert< k+1$.
\end{itemize}

If the pair $(\pi,\tau)$ is not a cut pair or is a cut pair with the cut formula different from $p$, then $\mathscr{M}(\mathsf{a})(\pi,\tau) =\pi=\mathscr{M}(\mathsf{b})(\pi,\tau)$. The aforementioned conditions hold. Now suppose $(\pi,\tau)$ is a cut pair and the roots of $\pi$ and $\tau$ are marked by $\Sigma;\Gamma\Rightarrow \Delta , p$ and $\Sigma; p,\Gamma\Rightarrow \Delta$ respectively.

If $ \pi$ consists only of one initial sequent $\Sigma;\Gamma\Rightarrow \Delta , p$ and $\Sigma;\Gamma\Rightarrow \Delta $ is also an initial, then $\mathscr{M}(\mathsf{a})(\pi,\tau)$ and $\mathscr{M}(\mathsf{b})(\pi,\tau)$ are equal to the $\infty$-proof consisting of the sequent $\Sigma;\Gamma\Rightarrow \Delta $. In this case, $\mathscr{M}(\mathsf{a})(\pi,\tau) =\mathscr{M}(\mathsf{b})(\pi,\tau)$.
If $\Sigma;\Gamma\Rightarrow \Delta $ is not an initial sequent, then $\mathscr{M}(\mathsf{a})(\pi,\tau) =\mathsf{acl}_p(\tau)=\mathscr{M}(\mathsf{b})(\pi,\tau)$. Obviously, the aforementioned conditions hold.

Let us consider the case when the last inference of the $\infty$-proof $\pi $ is an application of the rule ($\mathsf{\to_L}$): 
\begin{gather}\label{form2}
\AXC{$\pi_0$}
\noLine
\UIC{$\vdots$}
\noLine
\UIC{$\Sigma;\Lambda, C\Rightarrow \Delta, p$}
\AXC{$\pi_1$}
\noLine
\UIC{$\vdots$}
\noLine
\UIC{$\Sigma;\Lambda \Rightarrow A,\Delta, p$}
\LeftLabel{$\mathsf{\to_L}$}
\RightLabel{ .}
\BIC{$\Sigma;\Lambda, A\to C\Rightarrow \Delta, p$}
\DisplayProof
\end{gather}
By the definition of $\mathscr{M}$, the $\infty$-proofs $\mathscr{M}(\mathsf{a})(\pi,\tau)$ and $\mathscr{M}(\mathsf{b})(\pi,\tau)$ have the form 
\begin{gather}\label{form3}
\AXC{$\mathsf{R} (\pi_0, \mathsf{li}_{A\to C} (\tau))$}
\noLine
\UIC{$\vdots$}
\noLine
\UIC{$\Sigma;\Lambda, C\Rightarrow \Delta$}
\AXC{$\mathsf{R} (\pi_1, \mathsf{ri}_{A\to C} (\tau))$}
\noLine
\UIC{$\vdots$}
\noLine
\UIC{$\Sigma;\Lambda \Rightarrow A,\Delta$}
\LeftLabel{$\mathsf{\to_L}$}
\RightLabel{ .}
\BIC{$\Sigma;\Lambda, A\to C\Rightarrow \Delta$}
\DisplayProof
\end{gather}
Notice that, for any $s\in\mathit{Ann}(\pi)\cap \mathit{Ann}(\tau)$, we have $s\in \mathit{Ann}(\pi_0)\cap\mathit{Ann}(\pi_1)\cap \mathit{Ann}(\tau)$, $\lVert \pi_0\rVert_s  \leqslant \lVert \pi\rVert_s$ and $\lVert \pi_1\rVert_s  \leqslant \lVert \pi\rVert_s$. 
Recall that $\lVert \mathsf{li}_{A\to C}(\tau)\rVert_s\leqslant\lVert \tau\rVert_s$. 
Thus, $\lVert \pi_0\rVert_s\oplus \lVert \mathsf{li}_{A\to C}(\tau)\rVert_s\leqslant \lVert \pi_0\rVert_s\oplus \lVert \tau\rVert_s \leqslant\lVert \pi\rVert_s\oplus \lVert \tau\rVert_s$ and $\lVert (\pi_0,\mathsf{li}_{A\to C}(\tau))\rVert  \leqslant\lVert (\pi,\tau)\rVert  $. Analogously, $\lVert \pi_1\rVert_s\oplus \lVert \mathsf{ri}_{A\to C}(\tau)\rVert_s\leqslant\lVert \pi\rVert_s\oplus \lVert \tau\rVert_s$ and $\lVert (\pi_1,\mathsf{ri}_{A\to C}(\tau))\rVert  \leqslant\lVert (\pi,\tau)\rVert  $.

If $\lVert (\pi,\tau)\rVert<\alpha$, then $\lVert (\pi_0,\mathsf{li}_{A\to C}(\tau))\rVert  <\alpha$ and $\lVert (\pi_1,\mathsf{ri}_{A\to C}(\tau))\rVert <\alpha$. Since $\mathsf a\sim_{\alpha, n,k}\mathsf b$, we obtain $\mathsf{a} (\pi_0, \mathsf{li}_{A\to C} (\tau))=\mathsf{b} (\pi_0, \mathsf{li}_{A\to C} (\tau))$ and $\mathsf{a} (\pi_1, \mathsf{ri}_{A\to C} (\tau))=\mathsf{b} (\pi_1, \mathsf{ri}_{A\to C} (\tau))$. Therefore, $\mathscr{M}(\mathsf{a})(\pi,\tau)=\mathscr{M}(\mathsf{b})(\pi,\tau)$.

If $\lVert \pi\rVert_s\oplus \lVert \tau\rVert_s=\alpha$, then $\lVert (\pi_0,\mathsf{li}_{A\to C}(\tau))\rVert \leqslant\lVert \pi_0\rVert_s\oplus \lVert \mathsf{li}_{A\to C}(\tau)\rVert_s\leqslant\alpha$ and $\lVert (\pi_1,\mathsf{ri}_{A\to C}(\tau))\rVert\leqslant\lVert \pi_1\rVert_s\oplus \lVert \mathsf{ri}_{A\to C}(\tau)\rVert_s\leqslant\alpha$. Since $\mathsf a\sim_{\alpha, n,k}\mathsf b$, we have $\mathsf{a} (\pi_0, \mathsf{li}_{A\to C} (\tau))\sim^s_n\mathsf{b} (\pi_0, \mathsf{li}_{A\to C} (\tau))$ and $\mathsf{a} (\pi_1, \mathsf{ri}_{A\to C} (\tau))\sim^s_n\mathsf{b} (\pi_1, \mathsf{ri}_{A\to C} (\tau))$. Thus, $\mathscr{M}(\mathsf{a})(\pi,\tau)\sim^s_n\mathscr{M}(\mathsf{b})(\pi,\tau)$.

Suppose $\lVert \pi\rVert_s\oplus \lVert \tau\rVert_s=\alpha$ and $\lvert \pi\rvert + \lvert \tau\rvert<k+1$. Notice that $\lvert \pi_0\rvert  <\lvert \pi\rvert$ and $\lvert \pi_1\rvert  < \lvert \pi\rvert$. In addition,    $\lvert \mathsf{li}_{A\to C}(\tau)\rvert\leqslant\lvert \tau\rvert$ and $\lvert \mathsf{ri}_{A\to C}(\tau)\rvert\leqslant\lvert \tau\rvert$. Therefore, $\lvert \pi_0\rvert + \lvert \mathsf{li}_{A\to C}(\tau)\rvert<k$ and $\lvert \pi_1\rvert + \lvert \mathsf{ri}_{A\to C}(\tau)\rvert<k$. Also, $\lVert (\pi_0,\mathsf{li}_{A\to C}(\tau))\rVert \leqslant\lVert \pi_0\rVert_s\oplus \lVert \mathsf{li}_{A\to C}(\tau)\rVert_s\leqslant\alpha$ and $\lVert (\pi_1,\mathsf{ri}_{A\to C}(\tau))\rVert\leqslant\lVert \pi_1\rVert_s\oplus \lVert \mathsf{ri}_{A\to C}(\tau)\rVert_s\leqslant\alpha$. From $\mathsf a\sim_{\alpha, n,k}\mathsf b$, we obtain $\mathsf{a} (\pi_0, \mathsf{li}_{A\to C} (\tau))\sim^s_{n+1}\mathsf{b} (\pi_0, \mathsf{li}_{A\to C} (\tau))$ and $\mathsf{a} (\pi_1, \mathsf{ri}_{A\to C} (\tau))\sim^s_{n+1}\mathsf{b} (\pi_1, \mathsf{ri}_{A\to C} (\tau))$. Consequently, $\mathscr{M}(\mathsf{a})(\pi,\tau)\sim^s_{n+1}\mathscr{M}(\mathsf{b})(\pi,\tau)$.

The case has been checked. We omit the cases of other inference rules since they are similar to those already considered. The mapping $\mathscr M\colon \mathscr{RE}_{p} \to \mathscr{RE}_{p}$ is contractive. 
 
Now we claim that 
\[\mathscr{M}(\mathscr{NE}_{\alpha, n, k}\cap\mathscr{RE}_{p}) \subset \mathscr{NE}_{\alpha, n, k+1}.\] 
For a $p$-removing $(\alpha,n,k)$-non-expansive operation $\mathsf a$, we shall check that $\mathscr M (\mathsf a)$ is $(\alpha,n,k+1)$-non-expansive. In other words, for arbitrary pairs of $\infty$-proofs $(\pi,\tau)$, $(\sigma,\eta)$ such that $\pi\sim^s_i \sigma $ and $\tau\sim^s_i \eta $, we shall prove
\begin{itemize}
\item $\mathscr M (\mathsf a)(\pi,\tau) \sim^s_i \mathscr M (\mathsf a)(\sigma,\eta)$ whenever $\max \{\lVert \pi\rVert_s\oplus \lVert \tau\rVert_s, \lVert \sigma\rVert_s\oplus \lVert \eta\rVert_s\} < \alpha$;
\item $\mathscr M (\mathsf a)(\pi,\tau) \sim^s_i \mathscr M (\mathsf a)(\sigma,\eta)$ whenever $\max \{\lVert \pi\rVert_s\oplus \lVert \tau\rVert_s, \lVert \sigma\rVert_s\oplus \lVert \eta\rVert_s\} = \alpha$ and $i \leqslant n$;
\item $\mathscr M (\mathsf a)(\pi,\tau) \sim^s_i \mathscr M (\mathsf a)(\sigma,\eta)$ whenever $\max \{\lVert \pi\rVert_s\oplus \lVert \tau\rVert_s, \lVert \sigma\rVert_s\oplus \lVert \eta\rVert_s\}=\alpha$, $i = n+1$ and $\max \{\lvert \pi \rvert+ \lvert \tau \rvert,\lvert \sigma\rvert + \lvert \eta\rvert\}<k+1$.
\end{itemize}

Since $\pi\sim^s_i \sigma $ and $\tau\sim^s_i \eta $, $s\in \mathit{Ann}(\pi)\cap \mathit{Ann}(\tau)$ and $s\in \mathit{Ann}(\sigma)\cap \mathit{Ann}(\eta)$. Therefore, $s\in \mathit{Ann}(\mathsf{a}(\pi,\tau))\cap \mathit{Ann}(\mathsf{a}(\sigma,\eta))$ and $\mathsf{a}(\pi,\tau) \sim^s_0 \mathsf{a}(\sigma,\eta)$. If $i=0$, then the aforementioned conditions trivially hold. Otherwise, $i>0$ and at least the $1$-fragments of $\pi$ and $\sigma$ ($\tau$ and $\eta$) coincide.

We consider only the case when $(\pi,\tau)$ is a cut pair with the cut formula $p$ and the $\infty$-proof $\pi $ has the form of (\ref{form2}). In this case, $\sigma$ has the same form as $\pi$, but with $\sigma_0$ and $\sigma_1$ instead of $\pi_0$ and $\pi_1$. By the definition of $\mathscr{M}$, the $\infty$-proof $\mathscr{M}(\mathsf{a})(\pi,\tau)$ has the form of (\ref{form3}), where $\mathsf{R}=\mathsf{a}$. The $\infty$-proof $\mathscr{M}(\mathsf{a})(\sigma,\eta)$ is the same as  $\mathscr{M}(\mathsf{a})(\pi,\tau)$, only $\pi_0$, $\pi_1$ and $\tau$ are replaced with $\sigma_0$, $\sigma_1$ and $\eta$.

If $\max \{\lVert \pi\rVert_s\oplus \lVert \tau\rVert_s, \lVert \sigma\rVert_s\oplus \lVert \eta\rVert_s\} < \alpha$, then $\lVert \pi_0\rVert_s\oplus \lVert \mathsf{li}_{A\to C}(\tau)\rVert_s\leqslant\lVert \pi\rVert_s\oplus \lVert \tau\rVert_s<\alpha$ and $\lVert \sigma_0\rVert_s\oplus \lVert \mathsf{li}_{A\to C}(\eta)\rVert_s\leqslant\lVert \sigma\rVert_s\oplus \lVert \eta\rVert_s<\alpha$. Note that  $\pi_0 \sim^s_i \sigma_0$. Moreover, $\mathsf{li}_{A\to C} (\tau) \sim^s_i \mathsf{li}_{A\to C} (\eta)$, because the operation $\mathsf{li}_{A\to C}$ is non-expansive. Since $\mathsf{a}$ is $(\alpha,n,k)$-non-expansive, $\mathsf{a} (\pi_0, \mathsf{li}_{A\to C} (\tau))\sim^s_i\mathsf{a} (\sigma_0, \mathsf{li}_{A\to C} (\eta))$. Analogously, $\mathsf{a} (\pi_1, \mathsf{ri}_{A\to C} (\tau))\sim^s_i\mathsf{a} (\sigma_1, \mathsf{ri}_{A\to C} (\eta))$. Hence, $\mathscr M (\mathsf a)(\pi,\tau) \sim^s_i \mathscr M (\mathsf a)(\sigma,\eta)$.

Suppose $\max \{\lVert \pi\rVert_s\oplus \lVert \tau\rVert_s, \lVert \sigma\rVert_s\oplus \lVert \eta\rVert_s\} = \alpha$ and $i\leqslant n$. We see that $\max\{\lVert \pi_0\rVert_s\oplus \lVert \mathsf{li}_{A\to C}(\tau)\rVert_s, \lVert \sigma _0\rVert_s\oplus \lVert \mathsf{li}_{A\to C}(\eta)\rVert_s \}\leqslant \alpha$. Also, $\pi_0 \sim^s_i \sigma_0$ and $\mathsf{li}_{A\to C} (\tau) \sim^s_i \mathsf{li}_{A\to C} (\eta)$. Since $\mathsf{a}$ is $(\alpha,n,k)$-non-expansive and $i\leqslant n$, $\mathsf{a} (\pi_0, \mathsf{li}_{A\to C} (\tau))\sim^s_i\mathsf{a} (\sigma_0, \mathsf{li}_{A\to C} (\eta))$. Analogously, $\mathsf{a} (\pi_1, \mathsf{ri}_{A\to C} (\tau))\sim^s_i\mathsf{a} (\sigma_1, \mathsf{ri}_{A\to C} (\eta))$. Therefore, $\mathscr M (\mathsf a)(\pi,\tau) \sim^s_i \mathscr M (\mathsf a)(\sigma,\eta)$.

Now suppose $\max \{\lVert \pi\rVert_s\oplus \lVert \tau\rVert_s, \lVert \sigma\rVert_s\oplus \lVert \eta\rVert_s\}=\alpha$, $i = n+1$ and $\max \{\lvert \pi \rvert+ \lvert \tau \rvert,\lvert \sigma\rvert + \lvert \eta\rvert\}<k+1$. We have $\max\{\lVert \pi_0\rVert_s\oplus \lVert \mathsf{li}_{A\to C}(\tau)\rVert_s, \lVert \sigma _0\rVert_s\oplus \lVert \mathsf{li}_{A\to C}(\eta)\rVert_s \}\leqslant \alpha$ and $\max \{\lvert \pi_0\rvert + \lvert \mathsf{li}_{A\to C}(\tau)\rvert, \lvert \sigma_0\rvert + \lvert \mathsf{li}_{A\to C}(\eta)\rvert\}<k$. Since $\mathsf{a}$ is $(\alpha,n,k)$-non-expansive and, in addition, $\pi_0 \sim^s_i \sigma_0$ and $\mathsf{li}_{A\to C} (\tau) \sim^s_i \mathsf{li}_{A\to C} (\eta)$, we have $\mathsf{a} (\pi_0, \mathsf{li}_{A\to C} (\tau))\sim^s_i\mathsf{a} (\sigma_0, \mathsf{li}_{A\to C} (\eta))$. Analogously, $\mathsf{a} (\pi_1, \mathsf{ri}_{A\to C} (\tau))\sim^s_i\mathsf{a} (\sigma_1, \mathsf{ri}_{A\to C} (\eta))$. Consequently, $\mathscr M (\mathsf a)(\pi,\tau) \sim^s_i \mathscr M (\mathsf a)(\sigma,\eta)$. 

The case has been checked. We see that 
\[\mathscr{M}(\mathscr{NE}_{\alpha, n, k}\cap\mathscr{RE}_{p}) \subset \mathscr{NE}_{\alpha, n, k+1}.\] 

We claim that 
\[ \mathscr{M}(\mathscr{NP}_{\alpha, n, k} \cap\mathscr{RE}_{p}) \subset \mathscr{NP}_{\alpha, n, k+1}.\]
Assume we have a $p$-removing $(\alpha, n, k)$-non-pollutive operation $\mathsf{a}$. For a pair of $\infty$-proofs $(\pi,\tau)\in \mathsf{CF}_i(s)\times \mathsf{CF}_i(s)$, we shall check that
\begin{itemize}
\item $\mathscr{M}(\mathsf{a})(\pi,\tau) \in \mathsf{CF}_i(s)$ whenever $\lVert \pi\rVert_s \oplus \lVert \tau\rVert_s  < \alpha$;
\item $\mathscr{M}(\mathsf{a})(\pi,\tau) \in \mathsf{CF}_i(s)$ whenever $\lVert \pi\rVert_s \oplus \lVert \tau\rVert_s  = \alpha$ and $i \leqslant n$;
\item $\mathscr{M}(\mathsf{a})(\pi,\tau) \in \mathsf{CF}_i(s)$ whenever $\lVert \pi\rVert_s \oplus \lVert \tau\rVert_s  = \alpha$, $i = n+1$ and $\lvert \pi\rvert + \lvert \tau\rvert <k+1$.
\end{itemize}

The argument as a whole is very close to what it was before. To complete the picture, we consider the case when $(\pi,\tau)$ is a cut pair with the cut formula $p$ and the $\infty$-proof $\pi $ has the form of (\ref{form2}). By the definition of $\mathscr{M}$, the $\infty$-proof $\mathscr{M}(\mathsf{a})(\pi,\tau)$ has the form of (\ref{form3}), where $\mathsf{R}=\mathsf{a}$. 

If $\lVert \pi\rVert_s\oplus \lVert \tau\rVert_s<\alpha$, then $\lVert \pi_0\rVert_s\oplus \lVert \mathsf{li}_{A\to C}(\tau)\rVert_s  <\alpha$ and $\lVert \pi_1\rVert_s\oplus \lVert \mathsf{ri}_{A\to C}(\tau)\rVert_s <\alpha$. Since $\mathsf{li}_{A\to C} $ is non-pollutive and $\mathsf{a}$ is $(\alpha,n,k)$-non-pollutive, $\mathsf{li}_{A\to C} (\tau)\in \mathsf{CF}_i(s)$ and $\mathsf{a} (\pi_0, \mathsf{li}_{A\to C} (\tau))\in \mathsf{CF}_i(s)$. Analogously, $\mathsf{a} (\pi_1, \mathsf{ri}_{A\to C} (\tau))\in \mathsf{CF}_i(s)$. Hence, $\mathscr M (\mathsf a)(\pi,\tau) \in \mathsf{CF}_i(s)$.

Suppose $\lVert \pi\rVert_s\oplus \lVert \tau\rVert_s= \alpha$ and $i\leqslant n$. We see that $\lVert \pi_0\rVert_s\oplus \lVert \mathsf{li}_{A\to C}(\tau)\rVert_s \leqslant\alpha$ and $\lVert \pi_1\rVert_s\oplus \lVert \mathsf{ri}_{A\to C}(\tau)\rVert_s \leqslant\alpha$. In addition, $\mathsf{li}_{A\to C} (\tau)\in \mathsf{CF}_i(s)$, because $\mathsf{li}_{A\to C} $ is non-pollutive. Since $\mathsf{a}$ is $(\alpha,n,k)$-non-pollutive and $i\leqslant n$, $\mathsf{a} (\pi_0, \mathsf{li}_{A\to C} (\tau))\in \mathsf{CF}_i(s)$. Analogously, $\mathsf{a} (\pi_1, \mathsf{ri}_{A\to C} (\tau))\in \mathsf{CF}_i(s)$. Therefore, $\mathscr M (\mathsf a)(\pi,\tau) \in \mathsf{CF}_i(s)$.

Suppose $\lVert \pi\rVert_s\oplus \lVert \tau\rVert_s=\alpha$, $i=n+1$ and $\lvert \pi\rvert + \lvert \tau\rvert<k+1$. We have $\lVert \pi_0\rVert_s\oplus \lVert \mathsf{li}_{A\to C}(\tau)\rVert_s \leqslant\alpha$, $\lVert \pi_1\rVert_s\oplus \lVert \mathsf{ri}_{A\to C}(\tau)\rVert_s \leqslant\alpha$, $\lvert \pi_0\rvert + \lvert \mathsf{li}_{A\to C}(\tau)\rvert<k$ and $\lvert \pi_1\rvert + \lvert \mathsf{ri}_{A\to C}(\tau)\rvert<k$. In addition, $\mathsf{li}_{A\to C} (\tau)\in \mathsf{CF}_i(s)$, because $\mathsf{li}_{A\to C} $ is non-pollutive. Since $\mathsf{a}$ is $(\alpha,n,k)$-non-pollutive, $i= n+1$ and $\lvert \pi_0\rvert + \lvert \mathsf{li}_{A\to C}(\tau)\rvert<k$, $\mathsf{a} (\pi_0, \mathsf{li}_{A\to C} (\tau))\in \mathsf{CF}_i(s)$. Analogously, $\mathsf{a} (\pi_1, \mathsf{ri}_{A\to C} (\tau))\in \mathsf{CF}_i(s)$. Hence, $\mathscr M (\mathsf a)(\pi,\tau) \in \mathsf{CF}_i(s)$.

The case has been checked. We obtain that 
\[ \mathscr{M}(\mathscr{NP}_{\alpha, n, k} \cap\mathscr{RE}_{p}) \subset \mathscr{NP}_{\alpha, n, k+1}.\]

We leave it to the reader to check that
\[ \mathscr{M}(\mathscr{NF}_{\alpha, n, k} \cap\mathscr{RE}_{p}) \subset \mathscr{NF}_{\alpha, n, k+1}.\]

Now, applying Theorem \ref{explicit fixed-point3}, we define $\mathsf{re}_{p}$ as the fixed-point of $\mathscr{M}$. From Theorem \ref{explicit fixed-point3}, the operation $\mathsf{re}_{p}$ is non-expansive, non-pollutive and non-fattening.

\end{proof}

\begin{lemma}\label{reboxadeq}
Suppose there exists a non-expansive, non-pollutive and non-fattening operation $\mathsf{re}_B\in \mathscr{RE}_B$. Then there exists a non-expansive, non-pollutive and non-fattening operation $\mathsf{re}_{\Box B}\in \mathscr{RE}_{\Box B}$.
\end{lemma}
\begin{proof}


Assume we have a non-expansive, non-pollutive and non-fattening operation $\mathsf{re}_B\in \mathscr{RE}_B$. We define the required operation $\mathsf{re}_{\Box B}\in \mathscr{RE}_{\Box B}$ as a fixed-point of a contractive mapping $\mathscr{M}\colon\mathscr{RE}_{\Box B}\to\mathscr{RE}_{\Box B}$ in the same way as in the proof of the previous lemma.

For an operation $\mathsf a\in \mathscr{RE}_{\Box B}$, we define $\mathscr{M} (\mathsf a)\in \mathscr{RE}_{\Box B}$ as follows. If a pair of $\infty$-proofs $(\pi,\tau)$ is not a cut pair or is a cut pair with the cut formula different from $\Box B$, then $\mathscr{M} (\mathsf a) (\pi,\tau)\coloneq\pi$. Otherwise, $\mathscr{M} (\mathsf a) (\pi,\tau)$ is defined according to the following cases, where $\Sigma;\Gamma\Rightarrow \Delta$ is the cut result of the pair $(\pi,\tau)$.

Case 1. Suppose $ \pi$ or $ \tau$ consists only of one initial sequent. Then $\Sigma;\Gamma\Rightarrow \Delta$ is also an initial sequent. We define $\mathscr M (\mathsf a)(\pi,\tau)$ as the $\infty$-proof consisting of the sequent $\Sigma;\Gamma\Rightarrow \Delta$. 

Case 2. If $ \pi$ and $ \tau$ consist of more than one sequent and $\Box B$ is not the principle formula of the last application of an inference rule in $\pi$, then $\mathscr M (\mathsf a)(\pi,\tau)$ is defined as shown in Fig. 3, where $Q=\Box B$, $\mathsf{R}=\mathsf a$ and $ \mathsf{wk}_{\Phi; \Psi}$, $\mathsf{li}_{A\to C}$, $\mathsf{ri}_{A\to C}$, $\mathsf{i}_{A\to C}$ are operations from Lemma \ref{strongweakening} and Lemma \ref{inversion}.

Case 3. Suppose that $ \pi$ and $ \tau$ consist of more than one sequent and $\Box B$ is the principle formula of the last inference of $\pi$. In addition, suppose $\Box B$ is not a side formula of the last inference of $\tau$ (if this inference is an application of a modal rule). In this case, the value $\mathscr M (\mathsf a)(\pi,\tau)$ is defined as shown in Fig. 4, where $Q=\Box B$ and $\mathsf{R}= \mathsf a$.

\begin{center}
\begin{sidewaystable}
\begin{adjustwidth}{0cm}{-2.7cm}
\begin{minipage}{\textwidth}
\medskip
\begin{align*}
\left(
\pi \;,
\AXC{$\tau_0$}
\noLine
\UIC{$\vdots$}
\noLine
\UIC{$\Sigma; \Sigma_0,\Lambda, \Pi, \Box^+ \Pi \Rightarrow  A$}
\LeftLabel{$\Box$}
\UIC{$\Sigma;Q, \Phi ,\Box \Lambda, \Box^+ \Pi \Rightarrow \Box A,  \Psi$}
\DisplayProof
\right)
&\longmapsto
\AXC{$\tau_0$}
\noLine
\UIC{$\vdots$}
\noLine
\UIC{$\Sigma;\Sigma_0, \Lambda, \Pi, \Box^+ \Pi \Rightarrow  A$}
\LeftLabel{$\Box$}
\RightLabel{ ,}
\UIC{$\Sigma;\Phi ,\Box \Lambda, \Box^+ \Pi \Rightarrow \Box A,  \Psi$}
\DisplayProof \\\\
\left(
\pi \; ,
\AXC{$\tau_0$}
\noLine
\UIC{$\vdots$}
\noLine
\UIC{$\Sigma;\Sigma_0, \Lambda, \Pi, \Box^+ \Pi \Rightarrow  A$}
\AXC{$\tau_1$}
\noLine
\UIC{$\vdots$}
\noLine
\UIC{$\Sigma;\Sigma_0, \Lambda, \Pi, \Box^+ \Pi \Rightarrow \Box^+ A$}
\LeftLabel{$\Box^+$}
\BIC{$\Sigma;Q,\Phi ,\Box \Lambda, \Box^+ \Pi \Rightarrow \Box^+ A,  \Psi$}
\DisplayProof
\right)
&\longmapsto
\AXC{$\tau_0$}
\noLine
\UIC{$\vdots$}
\noLine
\UIC{$\Sigma;\Sigma_0, \Lambda, \Pi, \Box^+ \Pi \Rightarrow  A$}
\AXC{$\tau_1$}
\noLine
\UIC{$\vdots$}
\noLine
\UIC{$\Sigma;\Sigma_0, \Lambda, \Pi, \Box^+ \Pi \Rightarrow \Box^+ A$}
\LeftLabel{$\Box^+$}
\RightLabel{ ,}
\BIC{$\Sigma;\Phi ,\Box \Lambda, \Box^+ \Pi \Rightarrow \Box^+ A,  \Psi$}
\DisplayProof \\\\
\left(
\pi \; ,
\AXC{$\tau_0$}
\noLine
\UIC{$\vdots$}
\noLine
\UIC{$\Sigma;Q,\Gamma,A\Rightarrow C,\Lambda$}
\LeftLabel{$\mathsf{\to_R}$}
\UIC{$\Sigma;Q, \Gamma\Rightarrow A\to C,\Lambda$}
\DisplayProof
\right)
&\longmapsto
\AXC{$\mathsf{R}(\mathsf{i}_{A \to C}(\pi), \tau_0)$}
\noLine
\UIC{$\vdots$}
\noLine
\UIC{$\Sigma;\Gamma,A\Rightarrow C,\Lambda$}
\LeftLabel{$\mathsf{\to_R}$}
\RightLabel{ ,}
\UIC{$\Sigma;\Gamma\Rightarrow A\to C,\Lambda$}
\DisplayProof \\\\
\left(
\pi \; ,
\AXC{$\tau_0$}
\noLine
\UIC{$\vdots$}
\noLine
\UIC{$\Sigma;Q, \Lambda, C\Rightarrow \Delta$}
\AXC{$\tau_1$}
\noLine
\UIC{$\vdots$}
\noLine
\UIC{$\Sigma;Q,\Lambda \Rightarrow A,\Delta$}
\LeftLabel{$\mathsf{\to_L}$}
\BIC{$\Sigma;Q, \Lambda, A\to C\Rightarrow \Delta$}
\DisplayProof
\right)
&\longmapsto
\AXC{$\mathsf{R} ( \mathsf{li}_{A\to C} (\pi),\tau_0)$}
\noLine
\UIC{$\vdots$}
\noLine
\UIC{$\Sigma;\Lambda, C\Rightarrow \Delta$}
\AXC{$\mathsf{R} ( \mathsf{ri}_{A\to C} (\pi), \tau_1)$}
\noLine
\UIC{$\vdots$}
\noLine
\UIC{$\Sigma;\Lambda \Rightarrow A,\Delta$}
\LeftLabel{$\mathsf{\to_L}$}
\RightLabel{ ,}
\BIC{$\Sigma;\Lambda, A\to C\Rightarrow \Delta$}
\DisplayProof \\\\
\left(
\pi , \;
\AXC{$\tau_0$}
\noLine
\UIC{$\vdots$}
\noLine
\UIC{$\Sigma;Q,\Gamma\Rightarrow \Delta, A$}
\AXC{$\tau_1$}
\noLine
\UIC{$\vdots$}
\noLine
\UIC{$\Sigma;Q, A, \Gamma \Rightarrow \Delta$}
\LeftLabel{$\mathsf{cut}$}
\BIC{$\Sigma;Q, \Gamma\Rightarrow \Delta$}
\DisplayProof
\right)
&\longmapsto
\AXC{$\mathsf{R} (\mathsf{wk}_{\emptyset;A} (\pi), \tau_0)$}
\noLine
\UIC{$\vdots$}
\noLine
\UIC{$\Sigma;\Gamma\Rightarrow \Delta, A$}
\AXC{$\mathsf{R} ( \mathsf{wk}_{A;\emptyset} (\pi), \tau_1)$}
\noLine
\UIC{$\vdots$}
\noLine
\UIC{$\Sigma;A, \Gamma\Rightarrow \Delta$}
\LeftLabel{$\mathsf{cut}$}
\RightLabel{ .}
\BIC{$\Sigma;\Gamma\Rightarrow \Delta$}
\DisplayProof 
\end{align*}
\medskip 
\end{minipage}
\\\\
\center{\textbf{Fig. 4}}
\end{adjustwidth}
\end{sidewaystable}
\end{center}

Case 4. Suppose $\pi$ and $\tau$ have the following forms: 
\begin{gather*}
\AXC{$\pi_0$}
\noLine
\UIC{$\vdots$}
\noLine
\UIC{$\Sigma; \Sigma_0, \Lambda, \Pi, \Box^+\Pi \Rightarrow B$}
\LeftLabel{$\Box$}
\RightLabel{ ,}
\UIC{$\Sigma;\Phi, \Box \Lambda, \Box^+ \Pi \Rightarrow \Box B,  \Delta$}
\DisplayProof \qquad
\AXC{$\tau_0$}
\noLine
\UIC{$\vdots$}
\noLine
\UIC{$\Sigma; \Sigma^\prime_0,B, \Lambda^\prime, \Pi^\prime, \Box^+ \Pi^\prime \Rightarrow  A$}
\LeftLabel{$\Box$}
\RightLabel{ .}
\UIC{$\Sigma;\Phi^\prime ,\Box B,\Box \Lambda^\prime, \Box^+ \Pi^\prime \Rightarrow \Box A,  \Delta^\prime$}
\DisplayProof
\end{gather*}
Then $\mathscr{M}(\mathsf{a}) (\pi, \tau)$ is defined as
\[\AXC{$\mathsf{re}_B (\pi^\prime_0,\tau^\prime_0)$}
\noLine
\UIC{$\vdots$}
\noLine
\UIC{$\Sigma; \Sigma_0\setminus \Sigma^\prime_0,\Sigma^\prime_0, \Lambda \setminus \Lambda^\prime,  \Lambda^\prime, \Pi \setminus \Pi^\prime, \Pi^\prime, \Box^+ (\Pi \setminus \Pi^\prime), \Box^+ \Pi^\prime \Rightarrow  A$}
\LeftLabel{$\Box$}
\RightLabel{ ,}
\UIC{$\Sigma;\Phi ,\Box \Lambda, \Box^+ \Pi \Rightarrow \Box A,  \Delta^\prime$}
\DisplayProof\]
where $\pi^\prime_0=\mathsf{wk}_{\Sigma^\prime_0\setminus \Sigma_0,\Lambda^\prime \setminus \Lambda, \Pi^\prime \setminus \Pi, \Box^+ (\Pi^\prime \setminus \Pi) ;A}(\pi_0)$ and $\tau^\prime_0=\mathsf{wk}_{\Sigma_0\setminus \Sigma^\prime_0,\Lambda \setminus \Lambda^\prime, \Pi \setminus \Pi^\prime, \Box^+ (\Pi \setminus \Pi^\prime);\emptyset}(\tau_0)$.

Case 5. Suppose $\pi$ and $\tau$ have the forms: 
\begin{gather}\label{form4}
\AXC{$\pi_0$}
\noLine
\UIC{$\vdots$}
\noLine
\UIC{$\Sigma; \Sigma_0, \Lambda, \Pi, \Box^+\Pi \Rightarrow B$}
\LeftLabel{$\Box$}
\RightLabel{ ,}
\UIC{$\Sigma;\Phi, \Box \Lambda, \Box^+ \Pi \Rightarrow \Box B,  \Delta$}
\DisplayProof 
\end{gather}
\begin{gather}\label{form5}
\AXC{$\tau_0$}
\noLine
\UIC{$\vdots$}
\noLine
\UIC{$\Sigma; \Sigma^\prime_0,B, \Lambda^\prime, \Pi^\prime, \Box^+ \Pi^\prime \Rightarrow  A$}
\AXC{$\tau_1$}
\noLine
\UIC{$\vdots$}
\noLine
\UIC{$\Sigma; \Sigma^\prime_0, B, \Lambda^\prime, \Pi^\prime, \Box^+ \Pi^\prime \Rightarrow \Box^+ A$}
\LeftLabel{$\Box^+$}
\RightLabel{ .}
\BIC{$\Sigma;\Phi^\prime ,\Box B, \Box \Lambda^\prime, \Box^+ \Pi^\prime \Rightarrow \Box^+ A,  \Delta^\prime$}
\DisplayProof
\end{gather}
Then $\mathscr{M}(\mathsf{a}) (\pi, \tau)$ is defined as
\begin{gather}\label{form6}
\AXC{$\xi_0$}
\noLine
\UIC{$\vdots$}
\noLine
\UIC{$\Sigma; \Theta\Rightarrow  A$}
\AXC{$\xi_1$}
\noLine
\UIC{$\vdots$}
\noLine
\UIC{$\Sigma; \Theta\Rightarrow \Box^+ A$}
\LeftLabel{$\Box^+$}
\RightLabel{ ,}
\BIC{$\Sigma;\Phi^\prime , \Box \Lambda^\prime, \Box^+ \Pi^\prime \Rightarrow \Box^+ A,  \Delta^\prime$}
\DisplayProof
\end{gather}
where $\Theta =  \Sigma_0\setminus \Sigma^\prime_0,\Sigma^\prime_0,\Lambda \setminus \Lambda^\prime,  \Lambda^\prime, \Pi \setminus \Pi^\prime, \Pi^\prime, \Box^+ (\Pi \setminus \Pi^\prime), \Box^+ \Pi^\prime$, $\xi_0 =\mathsf{re}_B (\pi^\prime_0,\tau^\prime_0)$ and $\xi_1=\mathsf{re}_B (\pi^{\prime\prime}_0,\tau^\prime_1)$. Besides, $\pi^\prime_0=\mathsf{wk}_{\Sigma^\prime_0\setminus \Sigma_0, \Lambda^\prime \setminus \Lambda, \Pi^\prime \setminus \Pi, \Box^+ (\Pi^\prime \setminus \Pi) ; A}(\pi_0)$, $\tau^\prime_0=\mathsf{wk}_{\Sigma_0\setminus \Sigma^\prime_0,\Lambda \setminus \Lambda^\prime, \Pi \setminus \Pi^\prime, \Box^+ (\Pi \setminus \Pi^\prime);\emptyset}(\tau_0)$, $\pi^{\prime\prime}_0= \mathsf{wk}_{\Sigma^\prime_0\setminus \Sigma_0,\Lambda^\prime \setminus \Lambda, \Pi^\prime \setminus \Pi, \Box^+ (\Pi^\prime \setminus \Pi) ;\Box^+ A}(\pi_0)$ and $\tau^\prime_1=\mathsf{wk}_{\Sigma_0\setminus \Sigma^\prime_0,\Lambda \setminus \Lambda^\prime, \Pi \setminus \Pi^\prime, \Box^+ (\Pi \setminus \Pi^\prime);\emptyset}(\tau_1)$.

The mapping $\mathscr{M}$ is well defined. It can be shown that $\mathscr M\colon \mathscr{RE}_{\Box B} \to \mathscr{RE}_{\Box B}$ is contractive in the same way as in the proof of Lemma \ref{repadeq}. We claim that 
\[\mathscr{M}(\mathscr{NE}_{\alpha, n, k}\cap\mathscr{RE}_{\Box B}) \subset \mathscr{NE}_{\alpha, n, k+1}.\] 
For a $\Box B$-removing $(\alpha,n,k)$-non-expansive operation $\mathsf a$ and two pairs of $\infty$-proofs $\pi\sim^s_i \sigma $ and $\tau\sim^s_i \eta $, we shall prove that
\begin{itemize}
\item $\mathscr M (\mathsf a)(\pi,\tau) \sim^s_i \mathscr M (\mathsf a)(\sigma,\eta)$ whenever $\max \{\lVert \pi\rVert_s\oplus \lVert \tau\rVert_s, \lVert \sigma\rVert_s\oplus \lVert \eta\rVert_s\} < \alpha$;
\item $\mathscr M (\mathsf a)(\pi,\tau) \sim^s_i \mathscr M (\mathsf a)(\sigma,\eta)$ whenever $\max \{\lVert \pi\rVert_s\oplus \lVert \tau\rVert_s, \lVert \sigma\rVert_s\oplus \lVert \eta\rVert_s\} = \alpha$ and $i \leqslant n$;
\item $\mathscr M (\mathsf a)(\pi,\tau) \sim^s_i \mathscr M (\mathsf a)(\sigma,\eta)$ whenever $\max \{\lVert \pi\rVert_s\oplus \lVert \tau\rVert_s, \lVert \sigma\rVert_s\oplus \lVert \eta\rVert_s\}=\alpha$, $i = n+1$ and $\max \{\lvert \pi \rvert+ \lvert \tau \rvert,\lvert \sigma\rvert + \lvert \eta\rvert\}<k+1$.
\end{itemize}

We consider only the case when the $\infty$-proofs $\pi $ and $\tau$ have forms (\ref{form4}) and (\ref{form5}) respectively and $i>0$. Suppose $s\neq A$. Since $\pi\sim^s_i \sigma $, $\tau\sim^s_i \eta $ and $i>0$, we obtain $\pi=\sigma$ and $\tau=\eta$. Hence, $\mathscr{M}(\mathsf{a})(\pi,\tau)=\mathscr{M}(\mathsf{a})(\sigma,\eta)$. 
If $s= A$, then $\pi=\sigma$, $\tau_0=\eta_0$ and $\tau_1\sim^s_{i-1}\eta_1$. By the definition of $\mathscr{M}$, the $\infty$-proofs $\mathscr{M}(\mathsf{a})(\pi,\tau)$ and $\mathscr{M}(\mathsf{a})(\sigma,\eta)$ have the form of (\ref{form6}). Since $\mathsf{wk}_{\Sigma_0\setminus \Sigma^\prime_0,\Lambda \setminus \Lambda^\prime, \Pi \setminus \Pi^\prime, \Box^+ (\Pi \setminus \Pi^\prime);\emptyset}$ and $\mathsf{re}_B$ are non-expansive, we have 
\[\tau^\prime_1=\mathsf{wk}_{\Sigma_0\setminus \Sigma^\prime_0,\Lambda \setminus \Lambda^\prime, \Pi \setminus \Pi^\prime, \Box^+ (\Pi \setminus \Pi^\prime);\emptyset}(\tau_1)\sim^A_{i-1}\mathsf{wk}_{\Sigma_0\setminus \Sigma^\prime_0,\Lambda \setminus \Lambda^\prime, \Pi \setminus \Pi^\prime, \Box^+ (\Pi \setminus \Pi^\prime);\emptyset}(\eta_1)=\eta^\prime_1\]
and $\mathsf{re}_B (\pi^{\prime\prime}_0,\tau^\prime_1)= \mathsf{re}_B (\sigma^{\prime\prime}_0,\tau^\prime_1)\sim ^A_{i-1} \mathsf{re}_B (\sigma^{\prime\prime}_0,\eta^\prime_1)$. Besides, $\mathsf{re}_B (\pi^\prime_0,\tau^\prime_0)=\mathsf{re}_B (\sigma^\prime_0,\eta^\prime_0)$. It follows that $\mathscr{M}(\mathsf{a})(\pi,\tau\sim^s_i\mathscr{M}(\mathsf{a})(\sigma,\eta)$.

The case has been checked. We obtain that 
\[\mathscr{M}(\mathscr{NE}_{\alpha, n, k}\cap\mathscr{RE}_{\Box B}) \subset \mathscr{NE}_{\alpha, n, k+1}.\] 

Now we prove the assertion 
\[ \mathscr{M}(\mathscr{NP}_{\alpha, n, k} \cap\mathscr{RE}_{\Box B}) \subset \mathscr{NP}_{\alpha, n, k+1}.\]
Assume we have a $\Box B$-removing $(\alpha, n, k)$-non-pollutive operation $\mathsf{a}$. For a pair of $\infty$-proofs $(\pi,\tau)\in \mathsf{CF}_i(s)\times \mathsf{CF}_i(s)$, we shall check that
\begin{itemize}
\item $\mathscr{M}(\mathsf{a})(\pi,\tau) \in \mathsf{CF}_i(s)$ whenever $\lVert \pi\rVert_s \oplus \lVert \tau\rVert_s  < \alpha$;
\item $\mathscr{M}(\mathsf{a})(\pi,\tau) \in \mathsf{CF}_i(s)$ whenever $\lVert \pi\rVert_s \oplus \lVert \tau\rVert_s  = \alpha$ and $i \leqslant n$;
\item $\mathscr{M}(\mathsf{a})(\pi,\tau) \in \mathsf{CF}_i(s)$ whenever $\lVert \pi\rVert_s \oplus \lVert \tau\rVert_s  = \alpha$, $i = n+1$ and $\lvert \pi\rvert + \lvert \tau\rvert <k+1$.
\end{itemize}

Let us consider the case when $\pi $ and $\tau$ have forms (\ref{form4}) and (\ref{form5}) respectively and $i>0$. If $s\neq A$, then $\pi$ and $\tau$ do not contain applications of the rule ($\mathsf{cut}$). Since operations $\mathsf{wk}_{\Psi;\Psi^\prime}$ and $\mathsf{re}_B$ are non-pollutive, the $\infty$-proof $\mathscr{M}(\mathsf{a})(\pi,\tau)$ does not contain applications of the rule ($\mathsf{cut}$). If $s= A$, then  $\pi$ and $\tau_0$ do not contain applications of the rule ($\mathsf{cut}$). In addition, $\tau_1\in \mathsf{CF}_{i-1}(s)$. Notice that $\tau^\prime_1=\mathsf{wk}_{\Sigma_0\setminus \Sigma^\prime_0,\Lambda \setminus \Lambda^\prime, \Pi \setminus \Pi^\prime, \Box^+ (\Pi \setminus \Pi^\prime);\emptyset}(\tau_1) \in \mathsf{CF}_{i-1}(s)$ and $\mathsf{re}_B (\pi^{\prime\prime}_0,\tau^\prime_1)\in  \mathsf{CF}_{i-1}(s)$. Moreover, $\mathsf{re}_B (\pi^\prime_0,\tau^\prime_0)$ does not contain applications of the rule ($\mathsf{cut}$). Hence, $\mathscr{M}(\mathsf{a})(\pi,\tau)\in\mathsf{CF}_{i}(s)$.

All other cases are similar to the one considered above or to the cases considered in the proof of Lemma \ref{repadeq}. Therefore, we omit them. Now we see that \[ \mathscr{M}(\mathscr{NP}_{\alpha, n, k} \cap\mathscr{RE}_{\Box B}) \subset \mathscr{NP}_{\alpha, n, k+1}.\]

It remains to verify the assertion
\[ \mathscr{M}(\mathscr{NF}_{\alpha, n, k} \cap\mathscr{RE}_{p}) \subset \mathscr{NF}_{\alpha, n, k+1},\]
but we leave this to the reader.

Finally, we define $\mathsf{re}_{\Box B}$ as a fixed-point of the contractive mapping $\mathscr{M}$. By Theorem \ref{explicit fixed-point3}, the operation $\mathsf{re}_{\Box B}$ is non-expansive, non-pollutive and non-fattening.

\end{proof}

\begin{lemma}\label{rebox+adeq}
Suppose there exists a non-expansive, non-pollutive and non-fattening operation $\mathsf{re}_B\in \mathscr{RE}_B$. Then there exists a non-expansive, non-pollutive and non-fattening operation $\mathsf{re}_{\Box^+ B}\in \mathscr{RE}_{\Box^+ B}$.
\end{lemma}
\begin{proof}
The plan of this proof is the same as that of the proofs of the previous two lemmas.
Assume we have a non-expansive, non-pollutive and non-fattening $B$-removing operation $\mathsf{re}_B\colon \mathsf{P}\times\mathsf{P}\to \mathsf{P}$. We introduce the required $\Box^+ B$-removing operation as a fixed-point of a contractive mapping $\mathscr{M}\colon\mathscr{RE}_{\Box^+ B}\to\mathscr{RE}_{\Box^+ B}$.

Given an operation $\mathsf a\in \mathscr{RE}_{\Box^+ B}$, we define $\mathscr M (\mathsf a)\in \mathscr{RE}_{\Box^+ B}$ in the following way. For any two $\infty$-proofs $\pi$ and $\tau$, if the pair $(\pi,\tau)$ is not a cut pair or is a cut pair with the cut formula being not $\Box^+ B$, then $\mathscr M(\mathsf a)(\pi,\tau):=\pi$. Now assume $(\pi,\tau)$ is a cut pair with the cut formula $\Box^+ B$ and the cut result $\Sigma;\Gamma\Rightarrow \Delta$. We define $\mathscr M (\mathsf a)(\pi,\tau)$ according to the following cases, the first three of which repeat the corresponding cases from the proof of Lemma \ref{reboxadeq}.

Case 1. Suppose $ \pi$ or $ \tau$ consists only of one initial sequent. Then $\Sigma;\Gamma\Rightarrow \Delta$ is also an initial sequent. We define $\mathscr M (\mathsf a)(\pi,\tau)$ as the $\infty$-proof consisting of the sequent $\Sigma;\Gamma\Rightarrow \Delta$. 

Case 2. If $ \pi$ and $ \tau$ consist of more than one sequent and $\Box^+ B$ is not the principle formula of the last application of an inference rule in $\pi$, then $\mathscr M (\mathsf a)(\pi,\tau)$ is defined as shown in Fig. 3, where $Q=\Box^+ B$, $\mathsf{R}=\mathsf a$ and $ \mathsf{wk}_{\Phi; \Psi}$, $\mathsf{li}_{A\to C}$, $\mathsf{ri}_{A\to C}$, $\mathsf{i}_{A\to C}$ are operations from Lemma \ref{strongweakening} and Lemma \ref{inversion}.


Case 3. Suppose that $ \pi$ and $ \tau$ consist of more than one sequent and $\Box^+ B$ is the principle formula of the last inference of $\pi$. In addition, suppose $\Box^+ B$ is not a side formula of the last inference of $\tau$ (if this inference is an application of a modal rule). In this case, $\mathscr M (\mathsf a)(\pi,\tau)$ is defined as shown in Fig. 4, where $Q=\Box^+ B$ and $\mathsf{R}= \mathsf a$. 

Case 4. Suppose $\pi$ and $\tau$ have the following forms: 
\begin{gather*}
\AXC{$\pi_0$}
\noLine
\UIC{$\vdots$}
\noLine
\UIC{$\Sigma; \Sigma_0, \Lambda, \Pi, \Box^+ \Pi \Rightarrow  B$}
\AXC{$\pi_1$}
\noLine
\UIC{$\vdots$}
\noLine
\UIC{$\Sigma; \Sigma_0,\Lambda, \Pi, \Box^+ \Pi \Rightarrow \Box^+ B$}
\LeftLabel{$\Box^+$}
\RightLabel{ ,}
\BIC{$\Sigma;\Phi, \Box \Lambda, \Box^+ \Pi \Rightarrow \Box^+ B,  \Delta$}
\DisplayProof \\\\
\AXC{$\tau_0$}
\noLine
\UIC{$\vdots$}
\noLine
\UIC{$\Sigma; \Sigma^\prime_0,B,\Box^+ B, \Lambda^\prime, \Pi^\prime, \Box^+ \Pi^\prime \Rightarrow  A$}
\LeftLabel{$\Box$}
\RightLabel{ .}
\UIC{$\Sigma;\Phi^\prime ,\Box^+ B,\Box \Lambda^\prime, \Box^+ \Pi^\prime \Rightarrow \Box A,  \Delta^\prime$}
\DisplayProof
\end{gather*}

Subcase 4A. If the multiset $\Delta$ contains a copy of the formula $\Box^+ B$ (i.e., $\Delta = \Box^+ B, \Delta_0$), then we erase this copy from $\Delta$ and define $\mathscr{M}(\mathsf{a}) (\pi, \tau)$ as
\[\AXC{$\pi_0$}
\noLine
\UIC{$\vdots$}
\noLine
\UIC{$ \Sigma;\Sigma_0, \Lambda, \Pi, \Box^+ \Pi \Rightarrow  B$}
\AXC{$\pi_1$}
\noLine
\UIC{$\vdots$}
\noLine
\UIC{$\Sigma; \Sigma_0, \Lambda, \Pi, \Box^+ \Pi \Rightarrow \Box^+ B$}
\LeftLabel{$\Box^+$}
\RightLabel{ .}
\BIC{$\Sigma;\Phi, \Box \Lambda, \Box^+ \Pi \Rightarrow \Box^+ B,  \Delta_0$}
\DisplayProof\]

Subcase 4B. If $\Box^+ B$ does not belong to $\Delta$, then $\mathscr{M}(\mathsf{a}) (\pi, \tau)$ is defined as
\[\AXC{$\mathsf{re}_B (\pi^\prime_0,\mathsf{a}(\pi^\prime_1,\tau^\prime_0))$}
\noLine
\UIC{$\vdots$}
\noLine
\UIC{$\Sigma; \Sigma_0\setminus \Sigma^\prime_0,\Sigma^\prime_0, \Lambda \setminus \Lambda^\prime,  \Lambda^\prime, \Pi \setminus \Pi^\prime, \Pi^\prime, \Box^+ (\Pi \setminus \Pi^\prime), \Box^+ \Pi^\prime \Rightarrow  A$}
\LeftLabel{$\Box$}
\RightLabel{ ,}
\UIC{$\Sigma;\Phi^\prime ,\Box \Lambda^\prime, \Box^+ \Pi^\prime \Rightarrow \Box A,  \Delta^\prime$}
\DisplayProof\]
where $\pi^\prime_0=\mathsf{wk}_{\Sigma^\prime_0\setminus \Sigma_0,\Lambda^\prime \setminus \Lambda, \Pi^\prime \setminus \Pi, \Box^+ (\Pi^\prime \setminus \Pi) ;A}(\pi_0)$, $\pi^\prime_1=\mathsf{wk}_{\Sigma^\prime_0\setminus \Sigma_0, B,\Lambda^\prime \setminus \Lambda, \Pi^\prime \setminus \Pi, \Box^+ (\Pi^\prime \setminus \Pi) ;A}(\pi_1)$ and $\tau^\prime_0=\mathsf{wk}_{\Sigma_0\setminus \Sigma^\prime_0,\Lambda \setminus \Lambda^\prime, \Pi \setminus \Pi^\prime, \Box^+ (\Pi \setminus \Pi^\prime);\emptyset}(\tau_0)$.

Case 5. Suppose $\pi$ and $\tau $ have the forms:
\begin{gather}\label{case5pi}
\AXC{$\pi_0$}
\noLine
\UIC{$\vdots$}
\noLine
\UIC{$\Sigma;\Sigma_0,  \Lambda, \Pi, \Box^+ \Pi \Rightarrow  B$}
\AXC{$\pi_1$}
\noLine
\UIC{$\vdots$}
\noLine
\UIC{$\Sigma;\Sigma_0, \Lambda, \Pi, \Box^+ \Pi \Rightarrow\Box^+ B$}
\LeftLabel{$\Box^+$}
\RightLabel{ ,}
\BIC{$\Sigma;\Phi, \Box \Lambda, \Box^+ \Pi \Rightarrow \Box^+ B,  \Delta$}
\DisplayProof 
\end{gather}
\begin{gather}\label{case5tau}
\AXC{$\tau_0$}
\noLine
\UIC{$\vdots$}
\noLine
\UIC{$\Sigma; \Sigma^\prime_0,B,\Box^+ B, \Lambda^\prime, \Pi^\prime, \Box^+ \Pi^\prime \Rightarrow  A$}
\AXC{$\tau_1$}
\noLine
\UIC{$\vdots$}
\noLine
\UIC{$\Sigma; \Sigma^\prime_0, B,\Box^+ B, \Lambda^\prime, \Pi^\prime, \Box^+ \Pi^\prime \Rightarrow \Box^+ A$}
\LeftLabel{$\Box^+$}
\RightLabel{ .}
\BIC{$\Sigma;\Phi^\prime ,\Box^+ B, \Box \Lambda^\prime, \Box^+ \Pi^\prime \Rightarrow \Box^+ A,  \Delta^\prime$}
\DisplayProof
\end{gather} 

Subcase 5A. Suppose $\Delta$ contains a copy of the formula $\Box^+ B$. Then we erase this copy and define $\mathscr{M}(\mathsf{a}) (\pi, \tau)$ as
\[\AXC{$\pi_0$}
\noLine
\UIC{$\vdots$}
\noLine
\UIC{$\Sigma;\Sigma_0,  \Lambda, \Pi, \Box^+ \Pi \Rightarrow  B$}
\AXC{$\pi_1$}
\noLine
\UIC{$\vdots$}
\noLine
\UIC{$\Sigma;\Sigma_0, \Lambda, \Pi, \Box^+ \Pi \Rightarrow \Box^+ B$}
\LeftLabel{$\Box^+$}
\RightLabel{ ,}
\BIC{$\Sigma;\Phi, \Box \Lambda, \Box^+ \Pi \Rightarrow \Box^+ B,  \Delta_0$}
\DisplayProof \]
where $\Delta_0=\Delta\setminus \{\Box^+ B\}$.

Subcase 5B. If $\Delta$ does not contain $\Box^+ B$, then we define $\mathscr{M}(\mathsf{a}) (\pi, \tau)$ as
\begin{gather}\label{form7}
\AXC{$\xi_0$}
\noLine
\UIC{$\vdots$}
\noLine
\UIC{$\Sigma; \Theta\Rightarrow  A$}
\AXC{$\xi_1$}
\noLine
\UIC{$\vdots$}
\noLine
\UIC{$\Sigma; \Theta\Rightarrow \Box^+ A$}
\LeftLabel{$\Box^+$}
\RightLabel{ ,}
\BIC{$\Sigma;\Phi^\prime , \Box \Lambda^\prime, \Box^+ \Pi^\prime \Rightarrow \Box^+ A,  \Delta^\prime$}
\DisplayProof
\end{gather}
where $\Theta =  \Sigma_0\setminus \Sigma^\prime_0,\Sigma^\prime_0,\Lambda \setminus \Lambda^\prime,  \Lambda^\prime, \Pi \setminus \Pi^\prime, \Pi^\prime, \Box^+ (\Pi \setminus \Pi^\prime), \Box^+ \Pi^\prime$, $\xi_0 =\mathsf{re}_B (\pi^\prime_0,\mathsf{a}(\pi^\prime_1,\tau^\prime_0))$ and $\xi_1=\mathsf{re}_B (\pi^{\prime\prime}_0,\mathsf{a}(\pi^{\prime\prime}_1,\tau^\prime_1))$. In addition, $\pi^\prime_0=\mathsf{wk}_{\Sigma^\prime_0\setminus \Sigma_0, \Lambda^\prime \setminus \Lambda, \Pi^\prime \setminus \Pi, \Box^+ (\Pi^\prime \setminus \Pi) ; A}(\pi_0)$, $\pi^\prime_1=\mathsf{wk}_{\Sigma^\prime_0\setminus \Sigma_0,B,\Lambda^\prime \setminus \Lambda, \Pi^\prime \setminus \Pi, \Box^+ (\Pi^\prime \setminus \Pi) ;A}( \pi_1)$, $\tau^\prime_0=\mathsf{wk}_{\Sigma_0\setminus \Sigma^\prime_0,\Lambda \setminus \Lambda^\prime, \Pi \setminus \Pi^\prime, \Box^+ (\Pi \setminus \Pi^\prime);\emptyset}(\tau_0)$, $\pi^{\prime\prime}_0= \mathsf{wk}_{\Sigma^\prime_0\setminus \Sigma_0,\Lambda^\prime \setminus \Lambda, \Pi^\prime \setminus \Pi, \Box^+ (\Pi^\prime \setminus \Pi) ;\Box^+ A}(\pi_0)$,  $\pi^{\prime\prime}_1=\mathsf{wk}_{\Sigma^\prime_0\setminus \Sigma_0,B,\Lambda^\prime \setminus \Lambda, \Pi^\prime \setminus \Pi, \Box^+ (\Pi^\prime \setminus \Pi) ;\Box^+ A}(\pi_1)$ and $\tau^\prime_1=\mathsf{wk}_{\Sigma_0\setminus \Sigma^\prime_0,\Lambda \setminus \Lambda^\prime, \Pi \setminus \Pi^\prime, \Box^+ (\Pi \setminus \Pi^\prime);\emptyset}(\tau_1)$.

The mapping $\mathscr M\colon \mathscr{RE}_{\Box^+ B} \to \mathscr{RE}_{\Box^+ B}$ is well defined. We claim that $\mathscr M$ is contractive. Given two $\Box^+ B$-removing operations $\mathsf a$ and $\mathsf b$ such that $\mathsf a\sim_{\alpha, n,k}\mathsf b$, we shall prove $\mathscr M (\mathsf a)\sim_{\alpha, n,k+1}\mathscr M(\mathsf b)$. In other words, for any pair of $\infty$-proofs $(\pi,\tau)$ and any  $s\in\mathit{Ann}(\pi)\cap \mathit{Ann}(\tau)$, we shall check 
\begin{itemize}
\item $\mathscr M(\mathsf a)(\pi, \tau) = \mathscr M(\mathsf b)(\pi, \tau)$ whenever $\lVert (\pi,\tau)\rVert  < \alpha$,
\item $\mathscr M(\mathsf a)(\pi, \tau)\sim^s_n\mathscr M(\mathsf b)(\pi, \tau)$ whenever $\lVert \pi\rVert_s \oplus \lVert \tau\rVert_s = \alpha$,
\item $\mathscr M(\mathsf a)(\pi, \tau)\sim^s_{n+1}\mathscr M(\mathsf b)(\pi, \tau)$ whenever $\lVert \pi\rVert_s \oplus \lVert \tau\rVert_s = \alpha$ and $\lvert\pi\rvert + \rvert \tau \rvert< k+1$.
\end{itemize}

By analyzing cases in accordance with the definition of $\mathscr M$, the aforementioned
conditions are easily checked. We consider only the main case when $\pi$ and $\tau$ have forms (\ref{case5pi}) and (\ref{case5tau}) respectively and $\Box^+ B\nin \Delta$. In this case, the $\infty$-proofs $\mathscr{M}(\mathsf{a})(\pi,\tau)$ and $\mathscr{M}(\mathsf{b})(\pi,\tau)$ have the form of (\ref{form7}), $\lvert\pi\rvert=\lvert\tau\rvert=0$ and $B\nin \mathit{Ann}(\tau)$.

Let us fix $s^\prime\in \mathit{Fm}_\circ $ such that $\lVert (\pi,\tau)\rVert=\lVert \pi\rVert_{s^\prime} \oplus \lVert \tau\rVert_{s^\prime}$. Since $B\nin \mathit{Ann}(\tau)$, we have $s^\prime\neq B$ and $\lVert \pi_1\rVert_B<\lVert \pi\rVert_{s^\prime}$. 
Besides, $\lVert \pi^{\prime}_1\rVert_B\leqslant\lVert \pi_1\rVert_B$ and $\lVert \pi^{\prime\prime}_1\rVert_B\leqslant\lVert \pi_1\rVert_B$ from Lemma \ref{strongweakening}. Applying Lemma \ref{annotation lemma}, we see that $\lVert \pi^{\prime}_1\rVert_\circ\leqslant\lVert \pi^{\prime}_1\rVert_B+1\leqslant \lVert \pi_1\rVert_B+1\leqslant \lVert \pi\rVert_{s^\prime}$ and $\lVert \pi^{\prime\prime}_1\rVert_A\leqslant\lVert \pi^{\prime\prime}_1\rVert_B+1\leqslant\lVert \pi_1\rVert_B+1 \leqslant\lVert \pi\rVert_{s^\prime}$. From Lemma \ref{strongweakening}, we have $\lVert \tau^\prime_0\rVert_\circ \leqslant\lVert \tau_0\rVert_\circ<\lVert \tau\rVert_{s^\prime}$ and $\lVert \tau^\prime_1\rVert_A\leqslant\lVert \tau_1\rVert_A\leqslant\lVert \tau\rVert_{s^\prime} $. Consequently, 
\begin{gather*}
\lVert (\pi^{\prime}_1,\tau^\prime_0)\rVert \leqslant \lVert \pi^{\prime}_1 \rVert_\circ\oplus\lVert \tau^\prime_0\rVert_\circ<\lVert \pi\rVert_{s^\prime} \oplus \lVert \tau\rVert_{s^\prime}=\lVert (\pi,\tau)\rVert,\\
\lVert (\pi^{\prime\prime}_1,\tau^\prime_1)\rVert \leqslant \lVert \pi^{\prime\prime}_1 \rVert_A\oplus\lVert \tau^\prime_1\rVert_A\leqslant\lVert \pi\rVert_{s^\prime} \oplus \lVert \tau\rVert_{s^\prime}=\lVert (\pi,\tau)\rVert.
\end{gather*}

If $\lVert (\pi,\tau)\rVert<\alpha$, then $\lVert (\pi^{\prime}_1,\tau^\prime_0)\rVert <\alpha$ and $\lVert (\pi^{\prime\prime}_1,\tau^\prime_1)\rVert  <\alpha$. Since $\mathsf a\sim_{\alpha, n,k}\mathsf b$, we obtain $\mathsf{a} (\pi^{\prime}_1,\tau^\prime_0)=\mathsf{b} (\pi^{\prime}_1,\tau^\prime_0)$ and $\mathsf{a} (\pi^{\prime\prime}_1,\tau^\prime_1)=\mathsf{b} (\pi^{\prime\prime}_1,\tau^\prime_1)$. Therefore, $\mathscr{M}(\mathsf{a})(\pi,\tau)=\mathscr{M}(\mathsf{b})(\pi,\tau)$.

Suppose $\lVert \pi\rVert_s\oplus \lVert \tau\rVert_s=\alpha$. We see that $\lVert (\pi^{\prime}_1,\tau^\prime_0)\rVert <\alpha$ and $\mathsf{a} (\pi^{\prime}_1,\tau^\prime_0)=\mathsf{b} (\pi^{\prime}_1,\tau^\prime_0)$, because $\mathsf a\sim_{\alpha, n,k}\mathsf b$. In addition,  $s\neq B$, $\lVert \pi^{\prime\prime}_1\rVert_A\leqslant\lVert \pi^{\prime\prime}_1\rVert_B+1\leqslant\lVert \pi_1\rVert_B+1 \leqslant\lVert \pi\rVert_{s}$ and $\lVert \tau^\prime_1\rVert_A\leqslant\lVert \tau_1\rVert_A\leqslant\lVert \tau\rVert_{s} $.
If $s\neq A$, then $\lVert \tau^\prime_1\rVert_A\leqslant\lVert \tau_1\rVert_A<\lVert \tau\rVert_{s} $ and 
\[\lVert (\pi^{\prime\prime}_1,\tau^\prime_1)\rVert  \leqslant  \lVert \pi^{\prime\prime}_1 \rVert_A\oplus\lVert \tau^\prime_1\rVert_A<\lVert \pi\rVert_s\oplus \lVert \tau\rVert_s= \alpha.\]
Since $\mathsf a\sim_{\alpha, n,k}\mathsf b$, we have $\mathsf{a} (\pi^{\prime\prime}_1,\tau^\prime_1)=\mathsf{b} (\pi^{\prime\prime}_1,\tau^\prime_1)$. Hence, $\mathscr{M}(\mathsf{a})(\pi,\tau)=\mathscr{M}(\mathsf{b})(\pi,\tau)$.
If $s=A$, then $\lVert (\pi^{\prime\prime}_1,\tau^\prime_1)\rVert_A  \leqslant \alpha$ and $\mathsf{a} (\pi^{\prime\prime}_1,\tau^\prime_1)\sim^A_n\mathsf{b} (\pi^{\prime\prime}_1,\tau^\prime_1)$. We recall that the operation $\mathsf{re}_B$ is non-expansive and obtain $\mathsf{re}_B (\pi^{\prime\prime}_0,\mathsf{a}(\pi^{\prime\prime}_1,\tau^\prime_1))\sim^A_n \mathsf{re}_B (\pi^{\prime\prime}_0,\mathsf{b}(\pi^{\prime\prime}_1,\tau^\prime_1)) $.
Consequently, $\mathscr{M}(\mathsf{a})(\pi,\tau)\sim^s_{n+1}\mathscr{M}(\mathsf{b})(\pi,\tau)$.

The case has been checked. We see that the mapping $\mathscr M\colon \mathscr{RE}_{\Box^+ B} \to \mathscr{RE}_{\Box^+ B}$ is contractive. Now we claim that 
\[\mathscr{M}(\mathscr{NE}_{\alpha, n, k}\cap\mathscr{RE}_{\Box^+ B}) \subset \mathscr{NE}_{\alpha, n, k+1}.\] 
For a $\Box^+ B$-removing $(\alpha,n,k)$-non-expansive operation $\mathsf a$ and two pairs of $\infty$-proofs $\pi\sim^s_i \sigma $ and $\tau\sim^s_i \eta $, we shall prove that
\begin{itemize}
\item $\mathscr M (\mathsf a)(\pi,\tau) \sim^s_i \mathscr M (\mathsf a)(\sigma,\eta)$ whenever $\max \{\lVert \pi\rVert_s\oplus \lVert \tau\rVert_s, \lVert \sigma\rVert_s\oplus \lVert \eta\rVert_s\} < \alpha$;
\item $\mathscr M (\mathsf a)(\pi,\tau) \sim^s_i \mathscr M (\mathsf a)(\sigma,\eta)$ whenever $\max \{\lVert \pi\rVert_s\oplus \lVert \tau\rVert_s, \lVert \sigma\rVert_s\oplus \lVert \eta\rVert_s\} = \alpha$ and $i \leqslant n$;
\item $\mathscr M (\mathsf a)(\pi,\tau) \sim^s_i \mathscr M (\mathsf a)(\sigma,\eta)$ whenever $\max \{\lVert \pi\rVert_s\oplus \lVert \tau\rVert_s, \lVert \sigma\rVert_s\oplus \lVert \eta\rVert_s\}=\alpha$, $i = n+1$ and $\max \{\lvert \pi \rvert+ \lvert \tau \rvert,\lvert \sigma\rvert + \lvert \eta\rvert\}<k+1$.
\end{itemize}

Let us consider only the main case when $\pi$ and $\tau$ have forms (\ref{case5pi}) and (\ref{case5tau}) respectively, $\Box^+ B\nin \Delta$ and $i>0$. We see that $s\neq B$. If $s\neq A$, then  $\pi=\sigma$ and $\tau=\eta$ since $\pi\sim^s_i \sigma $, $\tau\sim^s_i \eta $ and $i>0$. We obtain that $\mathscr{M}(\mathsf{a})(\pi,\tau)=\mathscr{M}(\mathsf{a})(\sigma,\eta)$ and the aforementioned conditions hold. Otherwise, $s= A$. In this case, $\pi=\sigma$, $\tau_0=\eta_0$ and $\tau_1\sim^s_{i-1}\eta_1$. By the definition of $\mathscr{M}$, the $\infty$-proofs $\mathscr{M}(\mathsf{a})(\pi,\tau)$ and $\mathscr{M}(\mathsf{a})(\sigma,\eta)$ have the form of (\ref{form7}). We see that $\pi^{\prime}_0=\sigma^{\prime}_0$, $\pi^{\prime}_1=\sigma^{\prime}_1$ and $\tau^\prime_0=\eta^\prime_0$. Thus, $\mathsf{a} (\pi^{\prime}_1,\tau^\prime_0)=\mathsf{a} (\sigma^{\prime}_1,\eta^\prime_0)$ and $\mathsf{re}_B (\pi^\prime_0,\mathsf{a}(\pi^\prime_1,\tau^\prime_0))= \mathsf{re}_B (\sigma^\prime_0,\mathsf{a}(\sigma^\prime_1,\eta^\prime_0)) $. In addition, $\pi^{\prime\prime}_1=\sigma^{\prime\prime}_1$ and $\tau^\prime_1 \sim^s_{i-1} \eta^\prime_1$, because the operation $\mathsf{wk}_{\Sigma_0\setminus \Sigma^\prime_0,\Lambda \setminus \Lambda^\prime, \Pi \setminus \Pi^\prime, \Box^+ (\Pi \setminus \Pi^\prime);\emptyset}$ is non-expansive. Notice that $\lVert \pi^{\prime\prime}_1\rVert_s\leqslant\lVert \pi^{\prime\prime}_1\rVert_B+1\leqslant\lVert \pi_1\rVert_B+1 \leqslant\lVert \pi\rVert_{s}$ and $\lVert \tau^\prime_1\rVert_s\leqslant\lVert \tau_1\rVert_s\leqslant\lVert \tau\rVert_{s} $. Therefore, $\lVert \pi^{\prime\prime}_1\rVert_s\oplus \lVert \tau^\prime_1\rVert_s \leqslant \lVert \pi\rVert_s\oplus \lVert \tau\rVert_s$. Analogously, $\lVert \sigma^{\prime\prime}_1\rVert_s\oplus \lVert \eta^\prime_1\rVert_s \leqslant \lVert \sigma\rVert_s\oplus \lVert \eta\rVert_s$.

Suppose $\max \{\lVert \pi\rVert_s\oplus \lVert \tau\rVert_s, \lVert \sigma\rVert_s\oplus \lVert \eta\rVert_s\} < \alpha$. Then $\max\{\lVert \pi^{\prime\prime}_1\rVert_s\oplus \lVert \tau^\prime_1\rVert_s, \lVert \sigma^{\prime\prime}_1\rVert_s\oplus \lVert \eta^\prime_1\rVert_s \}<\alpha$. Since $\pi^{\prime\prime}_1=\sigma^{\prime\prime}_1$, $\tau^\prime_1 \sim^s_{i-1} \eta^\prime_1$ and $\mathsf{a}$ is  $(\alpha,n,k)$-non-expansive, we obtain $\mathsf{a} (\pi^{\prime\prime}_1,\tau^\prime_1)\sim^s_{i-1}\mathsf{a} (\sigma^{\prime\prime}_1,\eta^\prime_1)$. Now recall that  $\pi^{\prime\prime}_0=\sigma^{\prime\prime}_0$ and the operation $\mathsf{re}_B$ is non-expansive. Therefore, $\mathsf{re}_B (\pi^{\prime\prime}_0,\mathsf{a}(\pi^{\prime\prime}_1,\tau^\prime_1))\sim^s_{i-1} \mathsf{re}_B (\sigma^{\prime\prime}_0,\mathsf{a}(\sigma^{\prime\prime}_1,\eta^\prime_1)) $ and $\mathscr{M}(\mathsf{a})(\pi,\tau)\sim^s_{i}\mathscr{M}(\mathsf{a})(\sigma,\eta)$.

Suppose $\max \{\lVert \pi\rVert_s\oplus \lVert \tau\rVert_s, \lVert \sigma\rVert_s\oplus \lVert \eta\rVert_s\} = \alpha$ and $i\leqslant n+1$. Then $\max\{\lVert \pi^{\prime\prime}_1\rVert_s\oplus \lVert \tau^\prime_1\rVert_s, \lVert \sigma^{\prime\prime}_1\rVert_s\oplus \lVert \eta^\prime_1\rVert_s \}\leqslant\alpha$ and $i-1\leqslant n$. Since $\pi^{\prime\prime}_1=\sigma^{\prime\prime}_1$, $\tau^\prime_1 \sim^s_{i-1} \eta^\prime_1$ and $\mathsf{a}$ is  $(\alpha,n,k)$-non-expansive, we obtain $\mathsf{a} (\pi^{\prime\prime}_1,\tau^\prime_1)\sim^s_{i-1}\mathsf{a} (\sigma^{\prime\prime}_1,\eta^\prime_1)$. Since  $\pi^{\prime\prime}_0=\sigma^{\prime\prime}_0$ and the operation $\mathsf{re}_B$ is non-expansive, $\mathsf{re}_B (\pi^{\prime\prime}_0,\mathsf{a}(\pi^{\prime\prime}_1,\tau^\prime_1))\sim^s_{i-1} \mathsf{re}_B (\sigma^{\prime\prime}_0,\mathsf{a}(\sigma^{\prime\prime}_1,\eta^\prime_1)) $. We obtain $\mathscr{M}(\mathsf{a})(\pi,\tau)\sim^s_{i}\mathscr{M}(\mathsf{a})(\sigma,\eta)$.

The case has been checked. We see that 
\[\mathscr{M}(\mathscr{NE}_{\alpha, n, k}\cap\mathscr{RE}_{\Box^+ B}) \subset \mathscr{NE}_{\alpha, n, k+1}.\] 

Now we prove
\[ \mathscr{M}(\mathscr{NP}_{\alpha, n, k} \cap\mathscr{RE}_{\Box^+ B}) \subset \mathscr{NP}_{\alpha, n, k+1}.\]
Given a $\Box B$-removing $(\alpha, n, k)$-non-pollutive operation $\mathsf{a}$ and a pair of $\infty$-proofs $(\pi,\tau)\in \mathsf{CF}_i(s)\times \mathsf{CF}_i(s)$, we shall check that
\begin{itemize}
\item $\mathscr{M}(\mathsf{a})(\pi,\tau) \in \mathsf{CF}_i(s)$ whenever $\lVert \pi\rVert_s \oplus \lVert \tau\rVert_s  < \alpha$;
\item $\mathscr{M}(\mathsf{a})(\pi,\tau) \in \mathsf{CF}_i(s)$ whenever $\lVert \pi\rVert_s \oplus \lVert \tau\rVert_s  = \alpha$ and $i \leqslant n$;
\item $\mathscr{M}(\mathsf{a})(\pi,\tau) \in \mathsf{CF}_i(s)$ whenever $\lVert \pi\rVert_s \oplus \lVert \tau\rVert_s  = \alpha$, $i = n+1$ and $\lvert \pi\rvert + \lvert \tau\rvert <k+1$.
\end{itemize}

We consider only the main case when $\pi$ and $\tau$ have forms (\ref{case5pi}) and (\ref{case5tau}) respectively, $\Box^+ B\nin \Delta$ and $i>0$.

Suppose $\lVert \pi\rVert_s\oplus \lVert \tau\rVert_s \leqslant \alpha$ and $s\neq A$. Then $ \lVert \pi^{\prime}_1 \rVert_\circ\oplus\lVert \tau^\prime_0\rVert_\circ<\lVert \pi\rVert_s\oplus \lVert \tau\rVert_s \leqslant\alpha $, $\lVert \pi^{\prime\prime}_1\rVert_A\oplus \lVert \tau^\prime_1\rVert_A< \lVert \pi\rVert_s\oplus \lVert \tau\rVert_s \leqslant\alpha$. Note that $\pi$ and $\tau$ are cut-free $\infty$-proofs, because $s\neq B$, $s\neq A$, $\pi\in  \mathsf{CF}_i(s)$, $\tau\in  \mathsf{CF}_i(s)$ and $i>0$. Since $\mathsf{a}$ is  $(\alpha,n,k)$-non-pollutive and the operations of the form $\mathsf{wk}_{\Psi ; \Psi^\prime}$ and $\mathsf{re}_B$ are non-pollutive, the $\infty$-proofs $\mathsf{a} (\pi^{\prime}_1,\tau^\prime_0) $, $\mathsf{a} (\pi^{\prime\prime}_1,\tau^\prime_1)$ and $\mathscr{M}(\mathsf{a})(\pi,\tau)$ do not contain applications of the rule ($\mathsf{cut}$). 

Suppose $\lVert \pi\rVert_s\oplus \lVert \tau\rVert_s < \alpha$ and $s= A$. In this case, $ \lVert \pi^{\prime}_1 \rVert_\circ\oplus\lVert \tau^\prime_0\rVert_\circ<\lVert \pi\rVert_s\oplus \lVert \tau\rVert_s <\alpha $ and $\lVert \pi^{\prime\prime}_1\rVert_A\oplus \lVert \tau^\prime_1\rVert_A\leqslant \lVert \pi\rVert_s\oplus \lVert \tau\rVert_s <\alpha$. Besides, the $\infty$-proofs $\pi$ and $\tau_0$ are cut-free. Since $\mathsf{a}$ is  $(\alpha,n,k)$-non-pollutive and the operations of the form $\mathsf{wk}_{\Psi ; \Psi^\prime}$ are non-pollutive, we see that $\mathsf{a} (\pi^{\prime}_1,\tau^\prime_0) $ does not contain applications of the rule ($\mathsf{cut}$), $(\pi^{\prime\prime}_1,\tau^\prime_1)\in \mathsf{CF}_{i-1}(s)\times  \mathsf{CF}_{i-1}(s)$ and $\mathsf{a} (\pi^{\prime\prime}_1,\tau^\prime_1)\in \mathsf{CF}_{i-1}(s)$. 
Consequently, the $\infty$-proof $\mathsf{re}_B (\pi^{\prime}_0,\mathsf{a} (\pi^{\prime}_1,\tau^\prime_0))$ is cut-free, $(\pi^{\prime\prime}_0,\mathsf{a}(\pi^{\prime\prime}_1,\tau^\prime_1))\in \mathsf{CF}_{i-1}(s)\times\mathsf{CF}_{i-1}(s)$ and $\mathsf{re}_B (\pi^{\prime\prime}_0,\mathsf{a}(\pi^{\prime\prime}_1,\tau^\prime_1))\in \mathsf{CF}_{i-1}(s)$, because the operation $\mathsf{re}_B$ is non-pollutive. Therefore, $\mathscr{M}(\mathsf{a})(\pi,\tau)\in \mathsf{CF}_{i}(s)$.

Suppose $\lVert \pi\rVert_s\oplus \lVert \tau\rVert_s = \alpha$, $s= A$ and $i\leqslant n+1$. We have $ \lVert \pi^{\prime}_1 \rVert_\circ\oplus\lVert \tau^\prime_0\rVert_\circ<\lVert \pi\rVert_s\oplus \lVert \tau\rVert_s =\alpha $ and $\lVert \pi^{\prime\prime}_1\rVert_A\oplus \lVert \tau^\prime_1\rVert_A\leqslant \lVert \pi\rVert_s\oplus \lVert \tau\rVert_s =\alpha$. In addition, the $\infty$-proofs $\pi$, $\tau_0$ and $\mathsf{a} (\pi^{\prime}_1,\tau^\prime_0)$ are cut-free. Since $\mathsf{a}$ is  $(\alpha,n,k)$-non-pollutive, $(\pi^{\prime\prime}_1,\tau^\prime_1)\in \mathsf{CF}_{i-1}(s)\times  \mathsf{CF}_{i-1}(s)$ and $i-1\leqslant n$, we have $\mathsf{a} (\pi^{\prime\prime}_1,\tau^\prime_1)\in \mathsf{CF}_{i-1}(s)$. 
Hence, the $\infty$-proof $\mathsf{re}_B (\pi^{\prime}_0,\mathsf{a} (\pi^{\prime}_1,\tau^\prime_0))$ is cut-free and $\mathsf{re}_B (\pi^{\prime\prime}_0,\mathsf{a}(\pi^{\prime\prime}_1,\tau^\prime_1))\in \mathsf{CF}_{i-1}(s)$. Therefore, $\mathscr{M}(\mathsf{a})(\pi,\tau)\in \mathsf{CF}_{i}(s)$.

The case has been checked. We obtain that \[ \mathscr{M}(\mathscr{NP}_{\alpha, n, k} \cap\mathscr{RE}_{\Box^+ B}) \subset \mathscr{NP}_{\alpha, n, k+1}.\]

It remains to show that
\[ \mathscr{M}(\mathscr{NF}_{\alpha, n, k} \cap\mathscr{RE}_{p}) \subset \mathscr{NF}_{\alpha, n, k+1},\]
but we leave it to the reader.

Finally, applying Theorem \ref{explicit fixed-point3}, we define $\mathsf{re}_{\Box^+ B}$ as the fixed-point of $\mathscr{M}$. From Theorem \ref{explicit fixed-point3}, $\mathsf{re}_{\Box^+ B}$ is non-expansive, non-pollutive and non-fattening.
\end{proof}

\begin{proof}[Proof of Proposition \ref{reabadeq}]
We define the required operation $\mathsf{re}_A$ by induction on the structure of $A $.
If $A=p$, then $\mathsf{re}_{p} $ is defined by Lemma \ref{repadeq}. If $A=\bot$, then we put $\mathsf{re}_{\bot}(\pi,\tau):= \mathsf{i}_\bot (\pi)$, where $\mathsf{i}_\bot$ is the operation from Lemma \ref{inversion}.

Suppose $A=B\to C$. Then     
\[\mathsf{re}_{B \to C}(\pi,\tau):=\mathsf{re}_C(\mathsf{re}_B(\mathsf{wk}_{\emptyset; C}(\mathsf{ri}_{B\to C}(\tau)),\mathsf{i}_{B\to C}(\pi)),\mathsf{li}_{B\to C}(\tau)),\]
where $\mathsf{wk}_{\emptyset; C}, \mathsf{ri}_{B \to C}$, $\mathsf{i}_{B \to C}$, $\mathsf{li}_{B \to C}$ are operations from Lemma \ref{strongweakening} and Lemma \ref{inversion}.

If $A=\Box B$ ($A=\Box^+ B$), then the operation $\mathsf{re}_{\Box B}$ ($\mathsf{re}_{\Box^+ B}$) is given by Lemma \ref{reboxadeq} (Lemma \ref{rebox+adeq}).
\end{proof}

\section{Elimination of all cuts and slimming}
\label{s4.5}
In this section, we define a unary operation on the set $\mathsf{P}$ that maps any $\infty$-proof to a slim cut-free $\infty$-proof of the same sequent. 
\begin{theorem}[cut elimination]
\label{infcuttoinf}
If $\mathsf{S} +\mathsf{cut}\vdash\Sigma;\Gamma\Rightarrow\Delta$, then there is a slim cut-free $\infty$-proof of the sequent $\Sigma;\Gamma\Rightarrow\Delta$.
\end{theorem}
Before proving the theorem, consider the inference rule 
\begin{gather*}
\AXC{$\Sigma ;\Gamma , \Phi \Rightarrow \Psi ,  \Delta$}
\LeftLabel{$\mathsf{ctr}_{\Phi;\Psi}$}
\RightLabel{ ,}
\UIC{$\Sigma ;\Gamma , \Phi^{\prime} \Rightarrow \Psi^{\prime},  \Delta$}
\DisplayProof
\end{gather*}
where the multiset $ \Phi^{\prime}$ ($\Psi^{\prime}$) is obtained from $\Phi$ ($\Psi$) by removing all repetitions.
\begin{lemma}\label{contaction}
For any finite multisets of formulas $\Phi$ and $\Psi$, the inference rule ($\mathsf{ctr}_{\Phi;\Psi}$) is admissible in $\mathsf{S}+\mathsf{cut}$. Moreover, the corresponding operation on $\mathsf{P}$ is non-expansive, non-pollutive and non-fattening.
\end{lemma}
\begin{proof}
It sufficient to prove that there are non-expansive, non-pollutive and non-fattening operations corresponding to the contraction rules:
\begin{gather*}
\AXC{$\Sigma ;\Gamma , A,A \Rightarrow   \Delta$}
\LeftLabel{$\mathsf{cl}_A$}
\RightLabel{ ,}
\UIC{$\Sigma ;\Gamma ,A \Rightarrow \Delta$}
\DisplayProof \qquad
\AXC{$\Sigma ;\Gamma \Rightarrow  A,A,  \Delta$}
\LeftLabel{$\mathsf{cr}_A$}
\RightLabel{ .}
\UIC{$\Sigma ;\Gamma \Rightarrow A,\Delta$}
\DisplayProof 
\end{gather*}
Assume $\pi$ and $\tau$ are $\infty$-proofs of the sequents $\Sigma ;\Gamma , A,A \Rightarrow  \Delta$ and $\Sigma ;\Gamma \Rightarrow A,A, \Delta$ respectively. By Lemma \ref{AtoA}, we can find an $\infty$-proof of the sequent $\Sigma ;\Gamma , A \Rightarrow A,\Delta$, which is slim and cut-free. Let us denote this $\infty$-proof by $\sigma$ and define the required operations by putting $\mathsf{cl}_A(\pi) = \mathsf{re}_A(\sigma,\pi)$ and $\mathsf{cr}_A(\tau) = \mathsf{re}_A(\tau,\sigma)$, where $\mathsf{re}_A$ is the operation from Lemma \ref{reabadeq}. Since $\sigma$ is slim and cut-free, and $\mathsf{re}_A$ is non-expansive, non-pollutive and non-fattening, the operations $\mathsf{cl}_A$ and $\mathsf{cr}_A$ are also non-expansive, non-pollutive and non-fattening.
\end{proof}


\begin{proof}[Proof of Theorem \ref{infcuttoinf}]
We shall prove that there exists a cut-eliminating root-preserving operation $\mathsf{ce}$, which is also slimming. For this operation to be cut-eliminating, it suffices to define the operation $\mathsf{ce}$ as a fixed point of a contractive mapping $\mathscr{M}\colon \mathscr{RP}\to\mathscr{RP}$ so that the operation $\mathsf{ce}$ commutes with all inference rules except ($\mathsf{cut}$) and satisfies the condition: 
\[
\mathsf{ce}\left(
\AXC{$\pi_0$}
\noLine
\UIC{$\vdots$}
\noLine
\UIC{$\Sigma;\Gamma \Rightarrow A, \Delta$}
\AXC{$\pi_1$}
\noLine
\UIC{$\vdots$}
\noLine
\UIC{$\Sigma;\Gamma, A \Rightarrow \Delta$}
\LeftLabel{$\mathsf{cut}$}
\BIC{$\Sigma;\Gamma\Rightarrow  \Delta$}
\DisplayProof 
\right)
=
\mathsf {re}_A(\mathsf {ce}(\pi_0),\mathsf {ce}(\pi_1)).
\] 
However, in order to make the operation slimming, in the definition of $\mathsf{ce}$, we also use operations of the form $\mathsf{ctr}_{\Phi ;\Psi}$ from Lemma \ref{contaction}.

Given an operation $\mathsf a\in \mathscr{RP}$, we define $\mathscr M (\mathsf a)\in \mathscr{RP}$ as shown in Fig. 5, where $\mathsf{R}= \mathsf a$. Note that multisets $\Lambda^\prime$, $\Pi^\prime$ and $\Box^+ \Pi^\prime$ in Fig. 5 denote the result of removing repetitions from $\Lambda$, $\Pi$ and $\Box^+ \Pi$ respectively. 

We claim that the mapping $\mathscr M \colon \mathscr{RP}\to \mathscr{RP}$ is contractive. Let us consider a pair of root-preserving operations $\mathsf a$ and $\mathsf b$ such that $\mathsf a\sim_{\alpha, n,k}\mathsf b$. We shall check that $\mathscr M (\mathsf a)\sim_{\alpha, n,k+1}\mathscr M(\mathsf b)$. In other words, we prove  
\begin{itemize}
\item $\mathscr M(\mathsf a)(\pi) = \mathscr M(\mathsf b)(\pi)$ whenever $\lVert \pi\rVert  < \alpha$,
\item $\mathscr M(\mathsf a)(\pi)\sim^s_n\mathscr M(\mathsf b)(\pi)$ whenever $\lVert \pi\rVert_s = \alpha$,
\item $\mathscr M(\mathsf a)(\pi)\sim^s_{n+1}\mathscr M(\mathsf b)(\pi)$ whenever $\lVert \pi\rVert_s = \alpha$ and $\lvert\pi\rvert< k+1$,
\end{itemize}
for any $\infty$-proof $\pi$ and any $s\in\mathit{Ann}(\pi)$. We consider only the main case when $\pi$ has the form 
\begin{gather}\label{form8}
\AXC{$\pi_0$}
\noLine
\UIC{$\vdots$}
\noLine
\UIC{$\Sigma;\Gamma \Rightarrow A, \Delta$}
\AXC{$\pi_1$}
\noLine
\UIC{$\vdots$}
\noLine
\UIC{$\Sigma;\Gamma, A \Rightarrow \Delta$}
\LeftLabel{$\mathsf{cut}$}
\RightLabel{ .}
\BIC{$\Sigma;\Gamma\Rightarrow  \Delta$}
\DisplayProof 
\end{gather}

\begin{center}
\begin{sidewaystable}
\begin{adjustwidth}{0cm}{-2.7cm}
\begin{minipage}{\textwidth}
\medskip
\begin{align*}
\AXC{$\mathsf{Ax}$}
\noLine
\UIC{$\Sigma;\Gamma  \Rightarrow  \Delta$}
\DisplayProof 
&\longmapsto
\AXC{$\mathsf{Ax}$}
\noLine
\UIC{$\Sigma;\Gamma  \Rightarrow  \Delta$}
\DisplayProof , \\\\
\AXC{$\pi_0$}
\noLine
\UIC{$\vdots$}
\noLine
\UIC{$\Sigma;\Gamma , B \Rightarrow  \Delta$}
\AXC{$\pi_1$}
\noLine
\UIC{$\vdots$}
\noLine
\UIC{$\Sigma;\Gamma \Rightarrow  A, \Delta$}
\LeftLabel{$\mathsf{\rightarrow_L}$}
\BIC{$\Sigma;\Gamma , A \rightarrow B \Rightarrow  \Delta$}
\DisplayProof 
&\longmapsto
\AXC{$\mathsf R(\pi_0)$}
\noLine
\UIC{$\vdots$}
\noLine
\UIC{$\Sigma;\Gamma , B \Rightarrow  \Delta$}
\AXC{$\mathsf R(\pi_1)$}
\noLine
\UIC{$\vdots$}
\noLine
\UIC{$\Sigma;\Gamma \Rightarrow  A, \Delta$}
\LeftLabel{$\mathsf{\rightarrow_L}$}
\RightLabel{ ,}
\BIC{$\Sigma;\Gamma , A \rightarrow B \Rightarrow  \Delta$}
\DisplayProof 
\end{align*}
\begin{gather*}
\AXC{$\pi_0$}
\noLine
\UIC{$\vdots$}
\noLine
\UIC{$\Sigma;\Gamma, A \Rightarrow  B , \Delta$}
\LeftLabel{$\mathsf{\rightarrow_R}$}
\UIC{$\Sigma;\Gamma \Rightarrow  A \rightarrow B , \Delta$}
\DisplayProof 
\longmapsto
\AXC{$\mathsf R(\pi_0)$}
\noLine
\UIC{$\vdots$}\noLine
\UIC{$\Sigma;\Gamma, A \Rightarrow  B , \Delta$}
\LeftLabel{$\mathsf{\rightarrow_R}$}
\RightLabel{ ,}
\UIC{$\Sigma;\Gamma \Rightarrow  A \rightarrow B , \Delta$}
\DisplayProof \qquad
\AXC{$\pi_0$}
\noLine
\UIC{$\vdots$}
\noLine
\UIC{$\Sigma;\Sigma_0,\Lambda, \Pi, \Box^+ \Pi \Rightarrow A$}
\LeftLabel{$\mathsf{\Box}$}
\UIC{$\Sigma; \Phi, \Box \Lambda, \Box^+ \Pi \Rightarrow \Box A , \Psi$}
\DisplayProof   
\longmapsto
\AXC{$\mathsf{ctr}_{\Lambda;\emptyset}(\mathsf{ctr}_{ \Pi;\emptyset}(\mathsf{ctr}_{\Box^+ \Pi;\emptyset}(\mathsf R(\pi_0))))$}
\noLine
\UIC{$\vdots$}
\noLine
\UIC{$\Sigma;\Sigma_0,\Lambda^\prime, \Pi^\prime, \Box^+ \Pi^\prime \Rightarrow A$}
\LeftLabel{$\mathsf{\Box}$}
\RightLabel{ ,}
\UIC{$\Sigma; \Phi, \Box \Lambda, \Box^+ \Pi \Rightarrow \Box A , \Psi$}
\DisplayProof
\end{gather*}
\begin{align*}
\AXC{$\pi_0$}
\noLine
\UIC{$\vdots$}
\noLine
\UIC{$\Sigma;\Sigma_0,\Lambda, \Pi, \Box^+ \Pi \Rightarrow A$}
\AXC{$\pi_1$}
\noLine
\UIC{$\vdots$}
\noLine
\UIC{$\Sigma;\Sigma_0,\Lambda, \Pi, \Box^+ \Pi \Rightarrow \Box^+ A$}
\LeftLabel{$\Box^+$}
\BIC{$\Sigma; \Phi, \Box \Lambda, \Box^+ \Pi \Rightarrow  \Box^+ A ,\Psi$}
\DisplayProof
&\longmapsto
\AXC{$\mathsf{ctr}_{\Lambda;\emptyset}(\mathsf{ctr}_{ \Pi;\emptyset}(\mathsf{ctr}_{\Box^+ \Pi;\emptyset}(\mathsf{R}(\pi_0))))$}
\noLine
\UIC{$\vdots$}
\noLine
\UIC{$\Sigma;\Sigma_0,\Lambda^\prime, \Pi^\prime, \Box^+ \Pi^\prime \Rightarrow A$}
\AXC{$\mathsf{ctr}_{\Lambda;\emptyset}(\mathsf{ctr}_{ \Pi;\emptyset}(\mathsf{ctr}_{\Box^+ \Pi;\emptyset}(\mathsf{R}(\pi_1))))$}
\noLine
\UIC{$\vdots$}
\noLine
\UIC{$\Sigma;\Sigma_0,\Lambda^\prime, \Pi^\prime, \Box^+ \Pi^\prime \Rightarrow \Box^+ A$}
\LeftLabel{$\Box^+$}
\RightLabel{ ,}
\BIC{$\Sigma; \Phi, \Box \Lambda, \Box^+ \Pi \Rightarrow  \Box^+ A ,\Psi$}
\DisplayProof\\\\
\AXC{$\pi_0$}
\noLine
\UIC{$\vdots$}
\noLine
\UIC{$\Sigma;\Gamma \Rightarrow A, \Delta$}
\AXC{$\pi_1$}
\noLine
\UIC{$\vdots$}
\noLine
\UIC{$\Sigma;\Gamma, A \Rightarrow \Delta$}
\LeftLabel{$\mathsf{cut}$}
\BIC{$\Sigma;\Gamma\Rightarrow  \Delta$}
\DisplayProof 
&\longmapsto
\mathsf {re}_A(\mathsf R(\pi_0),\mathsf R(\pi_1)).
\end{align*}
\medskip 
\end{minipage}
\\\\
\center{\textbf{Fig. 5}}
\end{adjustwidth}
\end{sidewaystable}
\end{center}

Suppose $\lVert \pi\rVert< \alpha$. Then $ \lVert \pi_0 \rVert<\alpha $ and $\lVert \pi_1\rVert<\alpha$. Since $\mathsf a\sim_{\alpha, n,k}\mathsf b$, we have $\mathsf{a}(\pi_0)= \mathsf{b}(\pi_0)$ and $\mathsf{a}(\pi_1)=\mathsf{b}(\pi_1)$. Trivially, $\mathscr{M}(\mathsf{a})(\pi)= \mathsf{re}_A(\mathsf{a}(\pi_0),\mathsf{a}(\pi_1))=\mathsf{re}_A(\mathsf{b}(\pi_0),\mathsf{b}(\pi_1))=\mathscr{M}(\mathsf{b})(\pi)$. 

Suppose $\lVert \pi\rVert_s= \alpha$. Then $ \lVert \pi_0 \rVert\leqslant\lVert \pi_0 \rVert_s\leqslant\alpha $ and $\lVert \pi_1 \rVert\leqslant\lVert \pi_1\rVert_s\leqslant\alpha$. In this case, $\mathsf{a}(\pi_0)\sim^s_n \mathsf{b}(\pi_0)$ and $\mathsf{a}(\pi_1)\sim^s_n\mathsf{b}(\pi_1)$. Now recall that $\mathsf{re}_A$ is non-expansive. Therefore, $\mathscr{M}(\mathsf{a})(\pi)= \mathsf{re}_A(\mathsf{a}(\pi_0),\mathsf{a}(\pi_1))\sim^s_n\mathsf{re}_A(\mathsf{b}(\pi_0),\mathsf{b}(\pi_1))=\mathscr{M}(\mathsf{b})(\pi)$.

Suppose $\lVert \pi\rVert_s= \alpha$ and $\lvert\pi\rvert< k+1$. Then $ \lVert \pi_0 \rVert\leqslant\lVert \pi_0 \rVert_s\leqslant\alpha $, $\lVert \pi_1 \rVert\leqslant\lVert \pi_1\rVert_s\leqslant\alpha$, $\lvert\pi_0\rvert< k$ and $\lvert\pi_1\rvert< k$. We obtain $\mathsf{a}(\pi_0)\sim^s_{n+1} \mathsf{b}(\pi_0)$ and $\mathsf{a}(\pi_1)\sim^s_{n+1}\mathsf{b}(\pi_1)$. Since $\mathsf{re}_A$ is non-expansive, $\mathscr{M}(\mathsf{a})(\pi)= \mathsf{re}_A(\mathsf{a}(\pi_0),\mathsf{a}(\pi_1))\sim^s_{n+1}\mathsf{re}_A(\mathsf{b}(\pi_0),\mathsf{b}(\pi_1))=\mathscr{M}(\mathsf{b})(\pi)$.

The case has been checked. The mapping $\mathscr M \colon \mathscr{RP}\to \mathscr{RP}$ is contractive.

Now we claim that 
\[\mathscr{M}(\mathscr{CE}_{\alpha, n, k}\cap\mathscr{RP}) \subset \mathscr{CE}_{\alpha, n, k+1}.\]
For any root-preserving $(\alpha, n, k)$-cut-eliminating operation $\mathsf{a}$ and any $\infty$-proof $\pi$, we shall check that
\begin{itemize}
\item $\mathscr{M}(\mathsf{a})(\pi) $ is a cut-free $\infty$-proof whenever $\lVert \pi\rVert  < \alpha$,
\item $\mathscr{M}(\mathsf{a})(\pi)\in \mathsf{CF}_n(s)$ whenever $s\in\mathit{Ann}(\pi)$ and $ \lVert \pi\rVert_s  = \alpha$,
\item $\mathscr{M}(\mathsf{a})(\pi)\in \mathsf{CF}_{n+1}(s)$ whenever $s\in\mathit{Ann}(\pi)$, $ \lVert \pi\rVert_s  = \alpha$ and $\lvert\pi\rvert< k+1$.
\end{itemize}
We consider only the main case when $\pi$ has the form of (\ref{form8}).

Suppose $\lVert \pi\rVert< \alpha$. Then $ \lVert \pi_0 \rVert<\alpha $ and $\lVert \pi_1\rVert<\alpha$. Since $\mathsf{a}$ is $(\alpha, n, k)$-cut-eliminating, the $\infty$-proofs $\mathsf{a}(\pi_0)$ and $\mathsf{a}(\pi_1)$ do not contain applications of the rule ($\mathsf{cut}$). Now recall that $\mathsf{re}_A$ is non-pollutive. Therefore, $\mathscr{M}(\mathsf{a})(\pi)= \mathsf{re}_A(\mathsf{a}(\pi_0),\mathsf{a}(\pi_1))$ is a cut-free $\infty$-proof.

Suppose $\lVert \pi\rVert_s= \alpha$ for $s\in\mathit{Ann}(\pi)$. Then $ \lVert \pi_0 \rVert\leqslant\lVert \pi_0 \rVert_s\leqslant\alpha $ and $\lVert \pi_1 \rVert\leqslant\lVert \pi_1\rVert_s\leqslant\alpha$. In this case, the $\infty$-proofs $\mathsf{a}(\pi_0)$ and $\mathsf{a}(\pi_1)$ belong to $\mathsf{CF}_n(s)$. Since $\mathsf{re}_A$ is non-pollutive, $\mathscr{M}(\mathsf{a})(\pi)= \mathsf{re}_A(\mathsf{a}(\pi_0),\mathsf{a}(\pi_1))\in \mathsf{CF}_n(s)$. 

Now suppose $\lVert \pi\rVert_s= \alpha$ for $s\in\mathit{Ann}(\pi)$ and $\lvert\pi\rvert< k+1$. Then $ \lVert \pi_0 \rVert\leqslant\lVert \pi_0 \rVert_s\leqslant\alpha $, $\lVert \pi_1 \rVert\leqslant\lVert \pi_1\rVert_s\leqslant\alpha$, $\lvert\pi_0\rvert< k$ and $\lvert\pi_1\rvert< k$. Consequently, $\mathsf{a}(\pi_0)$ and $\mathsf{a}(\pi_1)$ belong to $\mathsf{CF}_{n+1}(s)$. Since $\mathsf{re}_A$ is non-pollutive, $\mathscr{M}(\mathsf{a})(\pi)= \mathsf{re}_A(\mathsf{a}(\pi_0),\mathsf{a}(\pi_1))\in \mathsf{CF}_{n+1}(s)$. 

The case has been checked. We obtain that 
\[\mathscr{M}(\mathscr{CE}_{\alpha, n, k}\cap\mathscr{RP}) \subset \mathscr{CE}_{\alpha, n, k+1}.\]

We leave it to the reader to show that
\[\mathscr{M}(\mathscr{SM}_{\alpha, n, k}\cap\mathscr{RP}) \subset \mathscr{SM}_{\alpha, n, k+1}.\]

Let $\mathsf{ce}$ be the fixed-point of $\mathscr{M}$. By Theorem \ref{explicit fixed-point2}, the operation $\mathsf{ce}$ is cut-eliminating and slimming.
\end{proof}

\section{Ordinary derivations in the logic $\mathsf{K}^+$ and regular $\infty$-proofs}
\label{s5}
In this section, we consider the logic $\mathsf{K}^+$ (without the rule ($\omega$)) and prove that the given system corresponds to the fragment of the calculus $\mathsf{S}+ \mathsf{cut}$ obtained by allowing only regular cut-free $\infty$-proofs. We also show that the inference rule ($\omega$) is admissible in $\mathsf{K}^+$.


In the case of the logic $\mathsf{K}^+$, as well as for other systems with non-well-founded proofs, regular $\infty$-proofs have finite representations called cyclic (or circular) proofs. For what follows, it will be enough for us to introduce these proofs only for the annotated case. An \emph{annotated cyclic proof} is a pair $(\kappa, d)$, where $\kappa$ is a finite tree of annotated sequents constructed in accordance with annotated versions of inference rules of $\mathsf{S}+\mathsf{cut}$ and $d$ is a function with the following properties: the function $d$ is defined on the set of all leaves of $\kappa$ that are not marked by initial sequents; the image $d(c)$ of a leaf $c$ lies on the path
from the root of $\kappa$ to the leaf $c$ and is not equal to $c$; $d(c)$ and $c$ are marked by the same sequents; all sequents on the path from $d(c)$ to $c$ have the same annotation; this path intersects an application of the rule ($\Box^+$) on the right premise.  If the function $d$ is defined at a leaf $c$, then we say that the nodes $c$ and $d(c)$ are connected by a back-link. An annotated cyclic proof is \emph{cut-free} if there are no applications of the rule ($\mathsf{cut}$) in this proof. 

For example, consider an annotated cyclic proof of the sequent $p, \Box p , \Box^+(p \rightarrow \Box p) \Rightarrow_p \Box^+ p$:
\begin{gather*}
\AXC{$\mathsf{Ax}$}
\noLine
\UIC{$\Sigma ; p , F,   \Box^+ F  \Rightarrow_\circ   p$}
\AXC{$\Sigma ; p ,  \Box p, \Box^+ F  \Rightarrow_p   \Box^+ p^{\;\,\tikzmark{a}} $ }
\AXC{$\mathsf{Ax}$}
\noLine
\UIC{$\Sigma ; p ,   \Box^+ F  \Rightarrow_p    p, \Box^+ p$}
\LeftLabel{$\mathsf{\rightarrow_L}$}
\BIC{$\Sigma ; p , F, \Box^+ F  \Rightarrow_p   \Box^+ p$}
\LeftLabel{$\mathsf{\Box}^+$}
\RightLabel{ ,}
\BIC{$\Sigma ; p, \Box p , \Box^+ F  \Rightarrow_p  \Box^+ p $ \tikzmark{b}}
\DisplayProof 
\begin{tikzpicture}[remember picture, overlay, >=latex, distance=7.7cm]
    \draw[->, thick] (pic cs:a) to [out=24,in=-10] (pic cs:b);
  \end{tikzpicture}
\end{gather*}
where $F= p \to \Box p$.

Notice that annotated cyclic proofs define the same provability relation as annotated regular $\infty$-proofs. Obviously, every annotated cyclic proof can be unravelled into a regular one. It can be shown that the converse is also true.

\begin{lemma}\label{annslimcutfreetocyclic}
Suppose a set of formulas $\Sigma$ is finite and a sequent $\Sigma;\Gamma\Rightarrow_s\Delta$ has an annotated $\infty$-proof $\kappa$ that is slim and cut-free. Then there is an annotated cyclic proof of $\Sigma ; \Gamma\Rightarrow_s\Delta$ that is also cut-free. 
\end{lemma}
\begin{proof}
We prove that there exists a annotated cyclic cut-free proof of $\Sigma;\Gamma\Rightarrow_s\Delta$ by induction on $\lVert \kappa \rVert$. If $s=\circ$, then the required cyclic proof is easily obtained by subinduction on the local height of $\kappa$. We omit the details and move on to the next case.

Suppose $s=C\in \mathit{Fm}$. Note that all formulas from $\kappa$ are subformulas of the formulas from $\Sigma;\Gamma\Rightarrow_s\Delta$. Since $\kappa$ is slim, $\kappa$ contains only finitely many different sequents that occur as right premises of the rule ($\Box^+$). We denote the number of these sequents by $n$. The required cyclic proof is defined into two steps from the tree $\kappa$.

Step 1. Let us denote the root of $\kappa$ by $r$ and the equivalence class of $r$ with respect to $\approx_\kappa$ by $R$. Consider an arbitrary node $a$ such that all nodes lying on the path from $r$ to $a$ excluding $a$ belong to $R$ and $a\nin R$. 
For the subtree $\kappa_a$ of $\kappa$ defined by $a$, note that $\lVert \kappa_a \rVert< \lVert \kappa \rVert$. Applying the induction hypothesis, we replace $\kappa_a$ in $\kappa$ with a cyclic cut-free proof of the corresponding sequent and then repeat this transformation for each such node of the tree $\kappa$.      

Step 2. Consider an arbitrary application of the rule ($\Box^+$) of level $n+2$. The path connecting $r$ with the conclusion of the application contains a pair of different nodes $b$ and $c$ such that these two nodes are right premises of the rule ($\Box^+$) marked by the same sequent. Assuming that $c$ is
further from $r$ than $b$, we cut the path under consideration at the node $c$ and connect $c$, which has become a leaf, with $b$ by a back-link. Applying the same transformation successively to all applications of the rule ($\Box^+$) of level $n+2$, we obtain the required cyclic cut-free proof of the sequent $\Sigma;\Gamma\Rightarrow_s\Delta$.
\end{proof}

\begin{propos}\label{someinftocyclic}
If $\mathsf{S}+\mathsf{cut}\vdash\Sigma;\Gamma\Rightarrow\Delta$, where $\Sigma$ is finite, then there is a regular cut-free $\infty$-proof of $\Sigma ; \Gamma\Rightarrow\Delta$. 
\end{propos}
\begin{proof}
Assume $\mathsf{S}+\mathsf{cut}\vdash\Sigma ; \Gamma \Rightarrow \Delta$ and $\Sigma$ is finite. By Theorem \ref{infcuttoinf}, there is a slim cut-free $\infty$-proof $\pi$ of the sequent $\Sigma ; \Gamma \Rightarrow \Delta$. Applying Lemma \ref{annslimcutfreetocyclic} to the annotated $\infty$-proof $\pi^\circ$, we obtain an annotated cyclic proof of $\Sigma ; \Gamma \Rightarrow \Delta$, which is also cut-free. Now we unravel this proof and erase all annotations in it. The resulting tree is a regular cut-free $\infty$-proof of $\Sigma ; \Gamma\Rightarrow\Delta$. 
\end{proof}
\begin{propos}[cut elimination for regular $\infty$-proofs]
If a sequent is provable by a regular $\infty$-proof, then it is provable by a regular cut-free $\infty$-proof.
\end{propos}
\begin{proof}
Assume $\pi $ is a regular $\infty$-proof of $\Sigma ; \Gamma \Rightarrow \Delta $. Since $\pi$ contains only a finite number of different sequents, only finitely many elements of $\Sigma$ occur in the premises of modal rules in  $\pi$. Hence, there is a finite subset $\Sigma_0$ of $\Sigma$ such that we can replace $\Sigma$ with $\Sigma_0$ in all sequents of the $\infty$-proof $\pi$ and obtain an $\infty$-proof of the sequent $\Sigma_0 ; \Gamma \Rightarrow \Delta $. By Proposition \ref{someinftocyclic}, there is a regular cut-free $\infty$-proof of $\Sigma_0 ; \Gamma \Rightarrow \Delta $. Now we extend $\Sigma
_0$ to $\Sigma$ in all sequents of this proof. The resulting tree is a regular cut-free $\infty$-proof of $\Sigma ; \Gamma \Rightarrow \Delta $.
\end{proof}

We put $\Sigma;\Gamma \vdash A$ if $\Sigma;\Gamma \vdash_\omega A$ and the corresponding $\omega$-derivation of $A$ can be chosen so that it does not contain applications of the rule ($\omega$).

\begin{lemma}\label{adm}
If $\Sigma;\emptyset \vdash B \rightarrow \Box (A \wedge B)$, then $\Sigma;\emptyset \vdash B \rightarrow \Box^+ A$.
\end{lemma}
\begin{proof}
Assume $\Sigma;\emptyset   \vdash B \rightarrow \Box( A \wedge B)$. It follows that $\Sigma;\emptyset\vdash A\wedge B \rightarrow \Box (A \wedge B)$ and $\Sigma;\emptyset\vdash \Box^+(A\wedge B \rightarrow \Box (A \wedge B))$. From Axiom (v), we obtain $\Sigma;\emptyset \vdash \Box (A\wedge B) \rightarrow \Box^+ (A \wedge B)$. Note that $\mathsf{K}^+ \vdash \Box^+ (A \wedge B)\to \Box^+ A$. Hence, $\Sigma;\emptyset \vdash \Box (A\wedge B) \to \Box^+ A$. Since $\Sigma;\emptyset \vdash B \to \Box (A\wedge B)$, we conclude $ \Sigma;\emptyset \vdash B \rightarrow \Box^+ A$.
\end{proof}

\begin{lemma}\label{reginftoord}
Suppose there is a regular annotated $\infty$-proof $\kappa$ of a sequent $\Sigma ; \Gamma \Rightarrow_s \Delta$. Then $\Sigma ;\emptyset \vdash \bigwedge \Gamma\to \bigvee \Delta $.
\end{lemma}
\begin{proof}
The proof of this lemma is similar to the proof of Lemma \ref{inftoomega}. The only difference is that, in Case 2, the set $\{G_a\mid a\in R\}$ is finite. We put $H\coloneq \bigvee \{G_a\mid a\in R \}$ and claim that $\Sigma ;\emptyset \vdash  H\to  \Box (C\wedge H)$. The claim is proved in the same manner as the assertion $\Sigma ;\emptyset \vdash_\omega  H_i\to \Box (C\wedge H_{i+1}) $. The end of the argument is exactly the same as before, except that we use $H$ instead of $H_i$ and Lemma \ref{adm} instead of the rule ($\omega$).  
\end{proof}

\begin{propos}\label{reginfeqord}
We have $\Sigma ;\Gamma \vdash A$, where $\Gamma$ is finite, if and only if the sequent $\Sigma ; \Gamma \Rightarrow A$ is provable in $\mathsf{S}+\mathsf{cut}$ by a regular $\infty$-proof.
\end{propos}
\begin{proof}
Assume $\Sigma ;\Gamma \vdash A $ and the set $\Gamma$ is finite. Note that that $\Sigma_0 ;\Gamma \vdash  A $ for some finite subset $\Sigma_0$ of $\Sigma$. We see that $\Sigma_0 ;\Gamma \vdash_\omega  A $ and $\mathsf{S}+\mathsf{cut}\vdash \Sigma_0 ; \Gamma \Rightarrow A$ by Lemma \ref{omegatoinfcut}.  Applying Proposition \ref{someinftocyclic}, we find a regular cut-free $\infty$-proof of the sequent $\Sigma_0 ; \Gamma \Rightarrow A$. We extend $\Sigma_0$ to $\Sigma$ in all sequents of this proof and obtain a regular $\infty$-proof of $\Sigma ; \Gamma \Rightarrow A$.

Now assume $\pi$ is a regular $\infty$-proof of $\Sigma ; \Gamma \Rightarrow A$. Applying Lemma  \ref{reginftoord} to the annotated $\infty$-proof $\pi^\circ$, we obtain $\Sigma ;\emptyset \vdash \bigwedge \Gamma\to A $.
Consequently, $\Sigma ;\Gamma  \vdash \bigwedge \Gamma\to A $, $\Sigma ;\Gamma  \vdash \bigwedge \Gamma$ and $\Sigma ;\Gamma \vdash  A$.
\end{proof}

\begin{propos}
Suppose $\Sigma ; \Gamma\vdash_\omega A$, where $\Sigma$ is finite. Then $\Sigma ; \Gamma\vdash A$.
\end{propos}
\begin{proof}
Assume $\Sigma ; \Gamma\vdash_\omega A$ and $\Sigma$ is finite. Since any $\omega$-derivation has only finitely many non-boxed assumption leaves, $\Sigma ; \Gamma_0\vdash_\omega A$ for a finite subset $\Gamma_0$ of $\Gamma$. By Lemma  \ref{omegatoinfcut}, $\mathsf{S}+\mathsf{cut}\vdash\Sigma;\Gamma_0\Rightarrow A$. By Proposition \ref{someinftocyclic}, the sequent $\Sigma;\Gamma_0\Rightarrow A$ has a regular $\infty$-proof. From Proposition \ref{reginfeqord}, we obtain $\Sigma ; \Gamma_0\vdash A$. Consequently, $\Sigma ; \Gamma\vdash A$.
\end{proof}
\begin{corollary}
The inference rule ($\omega$) is admissible in $\mathsf{K}^+$.
\end{corollary}

\subsubsection*{Acknowledgements.}
I thank my colleague Yuri Savateev for valuable discussions on the topic of the article. Heartily thanks to my beloved wife Maria for her warm and constant
support. SDG 
\bibliographystyle{amsplain}
\bibliography{Collected}

\end{document}